\documentclass[reqno,11pt]{amsart}
\usepackage{amsfonts}
\usepackage{graphicx}
\usepackage{amsfonts,amsmath, amssymb,mathrsfs}
\usepackage{mathtools}
\usepackage{dsfont}
\usepackage{hyperref}
\usepackage{marginnote}
\usepackage{color}
\usepackage[numbers,sort&compress]{natbib} 

\numberwithin{equation}{section}

\newcommand{\R}{\mathbb{R}}

\newtheorem{theorem}{Theorem}[section]
\newtheorem{corollary}[theorem]{Corollary}
\newtheorem{lemma}[theorem]{Lemma}

\newtheorem{remark}[theorem]{Remark}
\newtheorem{definition}[theorem]{Definition}

\def\v{\varepsilon}

\def\b{\beta}

\def\d{{\rm d}}

\def\r{\rho}

\def\dd{{\, \rm d}}
\def\M{{\mathcal{M}}}

\def\bp{\boldsymbol{\psi}}
\def\supp{{\rm supp}}
\newcommand{\longrightharpoonup}{\rm - \!\!\! \rightharpoonup}
\usepackage{tikz}

\usepackage{pgfplots}
\pgfplotsset{compat=1.18}

\usepgfplotslibrary{external}
\usepackage{float}

\begin{document}

\title[Global Solutions of the Compressible Euler-Riesz Equations]{Global Existence and Nonlinear Stability of Finite-Energy Solutions of the Compressible Euler-Riesz Equations with Large Initial Data of Spherical Symmetry}

\author[J. A. Carrillo]{Jos\'e A. Carrillo}
\address[J. A. Carrillo]{Mathematical Institute, University of Oxford,
 Oxford OX2 6GG, UK}
\email{carrillo@maths.ox.ac.uk}

\author[S. R. Charles]{Samuel R. Charles}
\address[S. R. Charles]{Mathematical Institute, University of Oxford,
 Oxford OX2 6GG, UK}
\email{samuel.charles@maths.ox.ac.uk}

\author[G.-Q. Chen]{Gui-Qiang G. Chen}
\address[G.-Q. Chen]{Mathematical Institute, University of Oxford,
 Oxford OX2 6GG, UK, and
 Academy of Mathematics and Systems Science, Chinese Academy of Sciences,
 Beijing 100190, China}
\email{chengq@maths.ox.ac.uk}

\author[D.~F. Yuan]{Difan Yuan}
\address[D.~F. Yuan]{School of Mathematical Sciences, Beijing Normal University and
 Laboratory of Mathematics, and Complex Systems, Ministry of Education, Beijing 100875, China; Mathematical Institute, University of Oxford,
 Oxford OX2 6GG, UK}
\email{yuandf@amss.ac.cn}

\begin{abstract}
The compressible Euler-Riesz equations are fundamental
with wide applications in astrophysics, plasma physics, and mathematical biology. In this paper,
we are concerned with the global existence and nonlinear stability of finite-energy solutions of the
multidimensional (M-D) Euler-Riesz equations with large initial data of spherical symmetry.
We consider both attractive and repulsive interactions for a wide range of $\it{Riesz}$ and $\it{logarithmic}$ potentials
for dimensions larger than or equal to two.
This is achieved by the inviscid limit of the solutions of the corresponding Cauchy problem
for the Navier-Stokes-Riesz equations.
The strong convergence of the vanishing viscosity solutions is achieved through delicate uniform estimates in $L^p$.
It is observed that, even if the attractive potential is super-Coulomb,
no concentration is formed near the origin in the inviscid limit.
Moreover, we prove that the nonlinear stability of global finite-energy solutions for the Euler-Riesz equations
is $\it{unconditional}$ under a spherically symmetric perturbation around the steady solutions.
Unlike the Coulomb case where the potential can be represented locally, the singularity and regularity of
the nonlocal radial Riesz potential near the origin require careful analysis,
which is a crucial step. We also establish uniform energy estimates of global solutions
for not only the Riesz potential
but also the limiting end-point case: the logarithmic potential.
Finally, unlike the Coulomb case, a Gr\"onwall type estimate is required to overcome the difficulty
of the appearance of boundary terms in the sub-Coulomb case and the singularity of the super-Coulomb potential.
Furthermore, we prove the nonlinear stability of global finite-energy solutions for the compressible Euler-Riesz equations
around steady states by employing concentration compactness arguments. Steady states properties are obtained by variational
arguments connecting to recent advances in aggregation-diffusion equations.
\end{abstract}

\date{\today}
\maketitle

\setcounter{tocdepth}{1}
\tableofcontents
\thispagestyle{empty}

\section{Introduction}

We are concerned with the global existence and nonlinear stability of spherically symmetric solutions
of the multi-dimensional (M-D) compressible Euler-Riesz equations (CEREs) with large initial data.
The equations are of the form:
\begin{align}\label{0.0}
	\begin{cases}
            \displaystyle
		\partial_t \rho+\nabla \cdot \M=0, \\[1mm]
            \displaystyle
		\partial_t \M+ \nabla \cdot \Big(\frac{\M\otimes\M}{\rho}\Big)
        +\nabla p + \kappa \rho \nabla \mathcal{W}_\alpha = \boldsymbol{0},
	\end{cases}
 \end{align}
 for $t>0$, $\boldsymbol{x} \in \mathbb{R}^n$, and $n \geq 2$,
 where $\rho: [0,\infty) \times \mathbb{R}^n \to \mathbb{R}_+$ is the density,
 $\M: [0,\infty) \times \mathbb{R}^n \to \mathbb{R}^n$ is the momentum,
 and $p: [0,\infty) \times \mathbb{R}^n \to \mathbb{R}$ is the pressure that is related to
 the density via a power law:
 \[
 p = p(\rho) = a_0 \rho^\gamma
 \]
 for the adiabatic exponent $\gamma > 1$.
 By scaling, we may choose
 $a_0 = \frac{(\gamma-1)^2}{4 \gamma} > 0$.
 $\mathcal{W}_\alpha : \mathbb{R}^n \to \mathbb{R}$ is the Riesz potential for $\alpha > 0$
 and the logarithmic potential for $\alpha = 0$.
 In this paper, we are particularly interested in a specific region $\alpha \in (-1,n)$ so that
 $\mathcal{W}_\alpha$ is given by $\mathcal{W}_\alpha=\Phi_\alpha *\rho$ and
 \[
 \nabla \mathcal{W}_\alpha(\boldsymbol{x}) =
 \begin{cases}
    \displaystyle
     (\nabla \Phi_\alpha * \rho) (\boldsymbol{x}) & \text{for } \alpha \in (-1,n-1), \\[1mm]
     \displaystyle
     \int_{\mathbb{R}^n} \nabla \Phi_\alpha(\boldsymbol{x} - \boldsymbol{y}) (\rho(\boldsymbol{y}) - \rho(\boldsymbol{x})) \dd \boldsymbol{y} \quad & \text{for } \alpha \in [n-1,n),
 \end{cases}
 \]
 where $\Phi_\alpha: \mathbb{R}^n \to \mathbb{R}$ is a Riesz kernel for $\alpha \in (0,n)$
 and a logarithmic kernel for $\alpha = 0$; that is, the kernel given by
 \[
 \Phi_\alpha(\boldsymbol{x}) =
 \begin{cases}
     -\frac{|\boldsymbol{x}|^{-\alpha}}{\alpha} &\quad \text{for } \alpha \neq 0, \\
     \log|\boldsymbol{x}| &\quad \text{for } \alpha = 0.
 \end{cases}
 \]
 Here, the logarithmic kernel can be obtained as a limit of the Riesz kernels, as $\alpha \to 0$,
 in the sense that
 \begin{equation}\label{kernelcon}
    \lim_{\alpha \to 0} \big(\Phi_\alpha (\boldsymbol{x})+ \frac1{\alpha}\big)
    = \lim_{\alpha \to 0} \frac{1-|\boldsymbol{x}|^{-\alpha}}{\alpha}
    = \log |\boldsymbol{x}|.
\end{equation}
Whenever $\alpha$ appears, we use the convention that $\alpha = 0$ represents the logarithmic case.
The range $\alpha \in (-1,n)$ includes the important Coulomb case $\alpha = n-2$.
$\kappa > 0$ represents the attractive potential case, and $\kappa < 0$ represents the repulsive potential case.
By scaling without loss of generality, we can choose $\kappa \in \{-1,1\}$.
When $\alpha < n$, the kernel $\Phi_\alpha$ is integrable near $0$ and hence can be analyzed
via the mean-field theory and the potential theory.
However, for $\alpha \geq n$, the kernel $\Phi_\alpha$ is not integrable near the origin so that
the potential theory is not amenable, in which case the kernel is called $\it{hypersingular}$
and behaves more like a short-range interaction problem; see \cite{Hardin_2018}.

\smallskip
Riesz gases arise
in many different areas of physics:
To list a few, plasma physics, star or galaxy dynamics, solid state physics, ferrofluids,
elasticity, and random matrices; see \cite{Campa_2014}.
The jellium and uniform electron gas have also been studied for the Riesz interaction;
see \cite{Lewin_2018,Cotar_2019}.
We refer to the recent survey \cite{Lewin_2022} with many open problems.
There are also applications in mathematical biology such as swarming models;
see \cite{Carrillo_2021,Carrillo_2017}.

\smallskip
 In the case of the super-Coulomb Riesz potential, for $\alpha \in (n-2,n)$, $\Phi_\alpha$ is known to be the kernel of the fractional Laplacian operator:
\[
    (-\Delta)^\frac{n-\alpha}{2} \mathcal{W}_\alpha = - c_{n,\alpha} \rho
\]
in the sense of distributions, where
\[
c_{n,\alpha} = \begin{cases}
    \frac{2^{n-\alpha} \pi^\frac{n}{2} \Gamma\left (\frac{n-\alpha}{2} \right )}{\alpha \Gamma\left (\frac{\alpha}{2} \right )} &\quad\mbox{for $\alpha > \max\{0, n-2\}$}, \\
    2\pi  &\quad\mbox{for $\alpha = 0$ and $n = 2$},
\end{cases}
\]
with $\Gamma(\cdot)$ as the gamma function;
see \cite{Gelfand_1958} for the details of the calculation.
The fractional Laplacian is a non-local operator and can be defined by Fourier multipliers
or in the real space through a singular integral.
Serfaty \cite{Serfaty_2020} first proved
that the Euler equations of form \eqref{0.0} in the pressure-less case
and repulsive regime, with Coulomb or super-Coulomb Riesz potential,
can be derived via the mean-field limits of particle systems.
Then
Nguyen-Rosenzweig-Serfaty \cite{Nguyen_2022} treated the whole range $\alpha \in [0,n)$
by developing a modulated-energy approach; see also Serfaty \cite{Serfaty:2024ojp}.

\smallskip
For $\alpha = 0$, $\Phi_0(\boldsymbol{x}) = \log |\boldsymbol{x}|$.
This interaction potential is called the logarithmic potential.
This case is ubiquitous, with many
applications in statistical and quantum mechanics, random matrix theory, self-avoiding random walks, random tilings, and the proofs of functional inequalities;
see \cite{Lewin_2022,Anderson_2010} for more details.

\smallskip
For $\alpha = n-2$ and $n\geq 3$, the interaction potential is called the Coulomb potential.
The Coulomb gas can be used as a toy model to study classical matters without the effects of
quantum mechanics. To give an example, Gamov's {\it liquid drop model} is a simplified model
used to study atomic nuclei, electrons, and atoms; see \cite{Serfaty:2024ojp}.
In the case when one phase is dominant, this can be reduced to a system of particles
interacting via a Coulomb potential; see \cite{Alberti_2009}.
Since the 1970s, the Coulomb gas has been regarded as an incredibly important area of study
in the field of statistical mechanics; this can be seen in
\cite{Lieb_1969} and the references cited therein.
The Coulomb potential has also offered many incredibly important applications
in the density functional theory;
see \cite{Lieb_2010,Lewin_2018,Cotar_2013,Cotar_2019}.
In particular, the {\it indirect Coulomb energy} and its bounds have been obtained
via the Lieb-Oxford inequality and the $n$-marginal optimal transport with Coulomb costs.
When $n=2$, the Coulomb case coincides with the logarithmic case;
this is also known as \textit{log gas, {\rm 2-D} one-component plasma, {\rm 2-D} jelium, or Dyson gas}
in physics.
The 2-D one-component plasma is seen as a natural toy model for plasmas
in 2-D statistical physics, the fractional quantum Hall effect, and random matrices.

\smallskip
We consider the Cauchy problem for \eqref{0.0} with
the initial data
$\rho_0 : \mathbb{R}^n \to \mathbb{R}$ and $\mathcal{M}_0 : \mathbb{R}^n \to \mathbb{R}^n$
that satisfy
 \begin{equation}\label{0.1}
     (\rho,\M)|_{t=0}
     = (\rho_0,\M_0)(\boldsymbol{x}) \longrightarrow (0,\boldsymbol{0})
     \qquad \mbox{ as $|\boldsymbol{x}| \to \infty$}.
 \end{equation}
 Since the global solutions of CEREs can contain vacuum states
 $\{ (t,\boldsymbol{x}) \, : \, \rho(t,\boldsymbol{x}) = 0 \}$ where the fluid velocity is ill-defined, we use the momentum
 $\M(t,\boldsymbol{x})$, which will be shown to be well-defined globally,
 instead of the fluid velocity $u(t,\boldsymbol{x})$.
 Under these assumptions, the energy of CEREs is given by
\begin{equation*}
     \mathcal{E}(\rho,\mathcal{M})(t):
     =\int_{\mathbb{R}^n}\Big(\rho(t, \boldsymbol{x}) e(\rho(t, \boldsymbol{x}))
     + \frac{1}{2} \Big|\frac{\mathcal{M}}{\sqrt{\rho}}\Big|^2(t, \boldsymbol{x})
     + \frac{\kappa}{2}\rho(\boldsymbol{x})(\Phi_\alpha \ast \rho)(t, \boldsymbol{x})\Big)\,\
     \dd\boldsymbol{x}.
\end{equation*}
Much work has been done on this in the Coulomb case or Newtonian potential ($\alpha = n-2$),
known as the compressible Euler-Poisson equations (CEPEs).
In the case of plasma ($\kappa = -1$, the repulsive case),
in a series of work by Guo \cite{Guo_1998}, Guo-Pausader \cite{Guo_2011},
Guo-Ionescu-Pausader \cite{Guo_2016}, and Ionescu-Pausader \cite{Ionescu2011},
it was proved that there exist global smooth irrotational solutions close to some constant background with small amplitude.
Meanwhile, for the gravitational gaseous star problem ($\kappa = 1$, the attractive case),
in Goldreich-Weber \cite{Goldreich_1980}, a family of smooth compact polytropic collapsing star
solutions of the 3-D CEPEs was found in the critical case of $\gamma = \frac{4}{3}$.
It was then proved that such solutions are
nonlinearly stable
under a small radially symmetric perturbation in
Hadzi\'{c}-Jang \cite{Had_i__2017}.
It was also proved by Hadzi\'c-Jang \cite{Had_i__2019} in the case of
both plasma and gaseous stars ($\kappa=\pm 1$)
that there exists a family of expanding
global-in-time solutions in Lagrangian coordinates
to the 3-D CEPEs in the case that $\gamma = 1 + \frac{1}{k}$ for $k\in \mathbb{N}$.

\smallskip
However, since CEPEs are strongly hyperbolic and nonlinear,
there is a breakdown of smooth solutions in finite time for large initial data, in general.
The challenging behaviors we might encounter are the formation of shock waves causing discontinuities and the occurrence of the cavitation and concentration of the density, which are a major area of interest,
still not fully understood.
In particular,
concentration
of the density at the origin for spherically symmetric solutions exhibiting gravitational collapse
can occur for
$\gamma \in [1, \frac{4}{3})$
in the gaseous star case ({\it cf}. \cite{Guo_2020,Guo_2021,Guo_2022}).
As such, we have to focus on the weak solutions of CEPEs.
Indeed, it was proved by Chen-He-Wang-Yuan \cite{Chen2021} that, for $\gamma > 1$
in the repulsive case $(\kappa = -1)$ and $\gamma > \frac{2n}{n + 2}$ with some conditions
on the initial mass $M < M_{\rm c}$ for $\gamma \in (\frac{2n}{n + 2}, \frac{2n - 2}{n})$
in the attractive case $(\kappa = 1)$, given the initial data of finite energy,
there exist global finite-energy solutions of CEPEs for $n\ge 3$, which implies that no concentration for the density is formed at the origin.

\smallskip
The difficulty of studying the general case of CEREs is that
the Riesz potential can no longer
be represented by using a local formula, due to the non-locality of the potential.
This provides us with multiple issues when trying to generalize from the Coulomb case to the general Riesz case.
As such, the existence results for CEREs were only recently proved,
starting with that of Choi-Jeong \cite{Choi_2022a}, where they proved that CEREs are locally well-posed away from the vacuum for the attractive and repulsive cases
for $\alpha \in (n-2,n)$ and showed the finite-time blowup of classical solutions.
Danchin-Ducomet \cite{Danchin_2022} proved that,
under the assumption that the initial density is small enough
and the initial velocity is close to some vector field
$v_0$
such that the spectrum of $Dv_0$ is positive and bounded away from zero,
there exists a global classical solution for $\alpha \in [n-2,n-1]$.
Choi-Jung \cite{Choi_2022b} proved that, in the pressureless case, under the conditions that the initial data
are small enough in some regular enough Sobolev space, there exists a unique global classical solution
for $\alpha \in (n-2,n)$.
Choi-Jung-Lee \cite{Choi_2024} extended the pressureless case to isentropic flows with pressure and proved that, under the condition that the initial density is small,
the initial velocity is close enough to some reference velocity $v_0$ in some regular enough Sobolev
space and, for a range of $\gamma$, there exists a unique global-in-time solution.
We also refer to \cite{Bobylev,Cercignani,Illner} for the critical value $\alpha=n-1,$
which enjoys similar physical phenomena
with the Vlasov-Riesz system.

\smallskip
It is of physical interest to understand if the solutions of CEREs are stable close to some steady
states of the equations.
For a fixed $\alpha$, such a steady state is usually given by a minimizer of
the energy functional:
 \begin{equation}\label{funcy}
\mathcal{G}(\varrho)
=\int_{\mathbb{R}^n} \big(\varrho(\boldsymbol{x}) e(\varrho(\boldsymbol{x}))
+\frac{1}{2}\varrho(\boldsymbol{x})(\Phi_\alpha\ast \varrho)(\boldsymbol{x})\big)\,
\dd\boldsymbol{x},
\end{equation}
when it exists.
Much work has been done towards this in the Coulomb case, that is,
CEPEs. Firstly, Rein \cite{Rein_2003} proved that, under the condition that the minimizer is unique
and a solution exists to the 3-D CEPEs with a condition on the pressure (which equates to
$\gamma > \frac{2n - 2}{n}$ in the isentropic case) and is {\it close} to the minimizer
at time $t=0$ in some sense, then the solution remains {\it close} to the unique minimizer for all time.
This is known as a nonlinear stability result; in this case, we refer to the unique minimizer
as nonlinearly stable.
In
Luo-Smoller \cite{Luo_2008,Luo_2009}, they obtained the nonlinear stability
results for the 3-D rotating Euler-Poisson equations and proved the existence of the solutions
and minimizers.
Recently, Lin-Zeng \cite{Lin_2022} established a turning point principle, which provides
sharp linear stability and instability criteria for the non-rotating gaseous stars.
Based on the linear stability criterion, Lin-Wang-Zhu \cite{Lin_2023} obtained nonlinear stability
results for CEPEs with a general pressure law.
We also refer to \cite{Deng,Jang2008,Jang2014} for the nonlinear instability
of the polytropic stars $(n=3,\alpha=n-2)$
for $\gamma\in[\frac{6}{5},\frac{4}{3}]$.
It is well-known that the steady states of CEREs relate to the steady states
of the aggregation-diffusion equation with a Riesz potential,
which can be seen by the fact that the corresponding energy functional is identical.
Carrillo-Hittmeir-Volzone-Yao \cite{Carrillo_2019}
proved that stationary states of a large class of aggregation-diffusion equations must be radial.
As a consequence, they proved that there exists a unique radially symmetric stationary state,
that is the unique radial minimizer,
to the Keller-Segel equations, a special case of the aggregation-diffusion equations,
where the non-local potential is given by the Coulomb potential.
We refer to \cite{CCY19} and the references therein for further properties of
the aggregation-diffusion equations.

\smallskip
The occurrence of cavitation and the formation of shock waves can lead
to difficulties with the global existence of problem \eqref{0.0}--\eqref{0.1},
as well as a lack of regularity of the solutions.
Hence, in order to prove the global existence of finite-energy solutions with spherical symmetry,
we use a vanishing physical viscosity method through the approximate PDEs,
called the compressible Navier-Stokes-Riesz equations (CNSREs), of the form:
\begin{align}\label{0.2}
\begin{cases}
\displaystyle
\partial_t \rho+\nabla \cdot \M=0, \\[1mm]
\displaystyle
\partial_t \M+ \nabla \cdot \Big(\frac{\M\otimes\M}{\rho}\Big)
+\nabla p + \kappa \rho \nabla \mathcal{W}_\alpha
= \varepsilon \nabla \cdot \Big(\mu(\rho)D \big(\frac{\M}{\rho} \big) \Big)
+ \varepsilon \nabla \Big( \lambda(\rho) \nabla \cdot \big(\frac{\M}{\rho} \big) \Big),
\end{cases}
\end{align}
 where $D\big( \frac{\M}{\rho} \big)
 = \frac{1}{2} \big( \nabla (\frac{\M}{\rho})
  + \big(\nabla ( \frac{\M}{\rho}) \big)^\intercal \big)$
is the stress tensor,
$\varepsilon > 0$ can be seen as the inverse of the Reynolds number,
and $\mu(\rho)$ and $\lambda(\rho)$ are the shear and bulk viscosity coefficients respectively
that satisfy
 \[
 \mu(\rho) \geq 0, \quad \mu(\rho) + n\lambda(\rho) \geq 0 \qquad \text{ for $\rho \geq 0$}.
 \]
Here the viscosity coefficients satisfy the BD relation: $\lambda(\rho) = \rho \mu'(\rho) - \mu(\rho)$,
needed to derive a BD-type entropy estimate to gain information on the derivatives of the density;
see Bresch-Desjardins \cite{Bresch02} and Bresch-Desjardins-Lin \cite{Bresch_2003}.
We will choose $(\mu(\rho),\lambda(\rho)) = (\rho,0)$, which satisfies the required properties,
for simplicity.
Formally, one can see that, as $\varepsilon \to 0$, solutions of \eqref{0.2} converge to solutions of \eqref{0.0}.
However, the rigorous proof of this has remained open; see
Chen-Feldman \cite{Chen-Feldman} and Dafermos \cite{Dafermos_2010} for the details.
Significant progress had been made towards this in the Coulomb case $\alpha = n-2$.

\smallskip
Since system \eqref{0.0} is an M-D system of hyperbolic balance laws, not much is understood about such equations.
Such equations usually do not always admit classical solutions due
to the formation of discontinuities and singularities.
A more classical theory in the area was established by Glimm \cite{Glimm_1965},
which provides the global-in-time existence of entropy solutions of bounded variation
for the 1-D general hyperbolic systems of conservation laws
in the case of initial conditions close to some constant function.
In Bianchini-Bressan \cite{Bianchini_2005}, the existence of entropy solutions of bounded variation
was established for the 1-D general hyperbolic systems of
conservation laws via the artificial vanishing viscosity method.
DiPerna \cite{DiPerna83} first proved the existence of entropy solutions for the isentropic Euler equations
for polytropic gases with $\gamma=1+\frac{2}{2k+1}$, with integer $k\geq2$,
by developing the method of compensated compactness ({\it cf}. Murat \cite{Murat} and Tartar \cite{Tartar}).
Chen \cite{Chen1988} and Ding-Chen-Luo \cite{Ding1989} established the global existence of entropy solutions
in $L^\infty$
for the general case $\gamma\in(1,\frac{5}{3}]$ by developing new techniques through careful entropy analysis
that involves fractional derivatives, the Hilbert transform,
and the compensated compactness argument.
Then Lions-Perthame-Tadmor \cite{Lions_1994} proved the global existence result
for entropy solutions in $L^\infty$
for $\gamma \in [3,\infty)$.
The gap $\gamma \in (\frac{5}{3},3)$ between the ranges was then bridged
by Lions-Perthame-Souganidis \cite{Lions_1998}.
The existence result for entropy solutions in $L^\infty$
was then proved by Huang-Wang \cite{Huang_2002} for $\gamma = 1$.

\smallskip
The idea of using the vanishing viscosity method to study
solutions
of hyperbolic problems goes back decades.
Gilbarg \cite{Gilbarg} proved that, for the viscous approximation for
1-D heat-conducting fluids,
there exists a unique shock layer that converges to a shock wave of the inviscid equations
in the inviscid limit.
Hoff-Liu \cite{Hoff-Liu} proved that the solutions of the 1-D Navier-Stokes equations for compressible isentropic flow
converge to an entropy-satisfying shock-wave solution in the inviscid limit.
Chen-Perepelitsa \cite{Chen_2010} first completed the proof of the inviscid limit rigorously of
the solutions of the 1-D Navier-Stokes
equations to the solutions of the isentropic Euler equations with general large initial data and
relative finite-energy for all the range of $\gamma>1$.
This is achieved by the compensated compactness framework in $L^p$,
studied first by LeFloch-Westdickenberg \cite{LeFloch2007}
for $\gamma\in (1,\frac{5}{3})$ and later further developed by Chen-Perepelitsa \cite{Chen_2010}
to handle all the adiabatic exponent $\gamma>1$ with a simplified proof.
Such a compensated compactness framework has been further studied
in  \cite{Chen_2015,Chen-Schrecker-2018,Schrecker,Chen-Wang-2018}
to prove spherically symmetric solutions of the M-D isentropic Euler equations.

\smallskip
Our aim is to establish the existence of global finite-energy solutions of \eqref{0.0}--\eqref{0.1}
for large spherically symmetric initial data. To this end, we consider the solutions of form:
\begin{equation}\label{0.3}
    \rho(t,\boldsymbol{x}) = \rho(t,r), \quad \M(t,\boldsymbol{x}) = m(t,r) \frac{\boldsymbol{x}}{r}
    \qquad \,\, \text{ for $\,r = |\boldsymbol{x}|$},
\end{equation}
with initial conditions:
\begin{equation}\label{0.4}
    (\rho, \M)|_{t=0}
    = (\rho_0, \M_0)(\boldsymbol{x}),
\end{equation}
where $\rho_0 \in L^1_+(\mathbb{R}^n)$ and $\frac{\M_0}{\sqrt{\rho_0}} \in L^2(\mathbb{R}^n)$.
Under these assumptions, \eqref{0.0} can be written in the form:
 \begin{align*}
	\begin{cases}
            \displaystyle
            \rho_t + m_r + \frac{n-1}{r}m = 0, \\
            \displaystyle
            m_t + \Big(\frac{m^2}{\rho} + p(\rho)\Big)_r
            + \frac{n-1}{r}\frac{m^2}{\rho} + \kappa \rho\,(\mathcal{W}_\alpha)_r = 0.
        \end{cases}
 \end{align*}
Similarly,
in the case of CNSREs, \eqref{0.2} becomes
 \begin{align}\label{0.6}
	\begin{cases}
            \displaystyle
            \rho_t + m_r + \frac{n-1}{r}m = 0, \\
            \displaystyle
            m_t + \Big(\frac{m^2}{\rho} + p(\rho)\Big)_r
            + \frac{n-1}{r}\frac{m^2}{\rho} + \kappa \rho\,(\mathcal{W}_\alpha)_r
             = \varepsilon\Big( (\mu + \lambda) \big( (\frac{m}{\rho})_r
               + \frac{n-1}{r}\frac{m}{\rho} \big) \Big)_r- \varepsilon\frac{n-1}{r}\mu_r\frac{m}{\rho}.
        \end{cases}
 \end{align}

Using the form of spherically symmetric solutions allows us to reduce our equations
from a system of $n+1$ equations for $n+1$ variables to a system of $2$ equations with $2$ variables.
However, this system is still tricky to work with.
For the Coulomb case ($\alpha = n-2$), considerable work has been done toward
the global existence of the compressible Navier-Stokes-Poisson equations (CNSPEs).
It was proved by Duan-Li \cite{Duan_2015} that, for spherically symmetric compactly supported
arbitrarily large initial data, there exist global weak spherically symmetric solutions
of CNSPEs with a stress-free boundary condition and $\gamma \in (\frac{6}{5},\frac{4}{3}]$.
An important result in Chen-He-Wang-Yuan \cite{Chen2021} is the global existence
of spherically symmetric solutions of
CEPEs ($i.e.,$ the Coulomb case $\alpha = n-2$ for $n\ge 3$).
Different from \cite{Chen2021}, in this paper, we prove the global existence and nonlinear stability of finite-energy
solutions for CEREs with a
general class of Riesz potentials and the logarithmic potential
for dimension $n \geq 2$.

\smallskip
One of the main difficulties we have to overcome has been mentioned above as the
non-locality of the potential for CEREs.
The potential for CEPEs with the Coulomb potential $\alpha = n-2$ can be represented locally as an exact formula
given by
\[
(\Phi_{n-2} * \rho)_r(t,r) = \frac{\omega_n}{r^{n-1}} \int^r_0 \rho(t,\eta) \eta^{n-1} \dd \eta,
\]
where $r$ is the radial coordinate and $\omega_n$ is the surface area of the $n$-D unit ball.
This is fundamental in proving the uniform energy estimates, the BD-type entropy,
and the higher integrability
in \cite{Chen2021}. Such a formula is derived from the fact that the Coulomb potential satisfies
the following Poisson equation:
\[
\Delta (\Phi_{n-2} * \rho) = \omega_n \rho,
\]
which is a local equation.
However, for general $\alpha$, we no longer have this type of structure, hence such a local formula
is not possible.
Even in the case $\alpha \in (n-2,n)$, the potential solves a fractional Poisson-type equation,
which is a non-local equation and so cannot be employed to derive a local estimate.
However, to overcome this problem, we use a radial representation for the Riesz potential
to obtain the estimates on the potential (see Lemma 3.1), which is incredibly important
in obtaining the estimates on the approximate solutions.
This is particularly so in the case of higher integrability estimates,
due to the non-locality of the potential, where the local estimates are required.
\smallskip

Moreover, when analyzing the convergence of the potential energies
as solutions transit from the free boundary problem to the Cauchy problem of global weak solutions,
the absence of the local representation formula for the potential terms necessitates
an alternative approach
to that of \cite{Chen2021}.
Specifically, to understand the behavior near the origin, we adopt a method distinct from the previous work.
Although the earlier studies provided the local convergence of the density away from the origin in the limit,
we establish a stronger result by proving the global convergence of the density in this paper; see Lemma \ref{Lq}.
The key point stems from the combination of the local convergence and the mass conservation in the limit,
which ensures that there is neither a blow-up of mass at the origin nor a loss of mass at infinity.
These enable us to derive the global convergence result.

\smallskip
Other important technical novelties of our work concern the BD-type entropy estimates.
When the Riesz potential is sub-Coulombic or Coulombic,
{\it i.e.}, $\alpha \leq n - 2$, we want to obtain a BD-type entropy estimate.
However, due to the nonlocality of the potential, we discover that,
when the integration by parts is applied,
we obtain a nonzero boundary term only in the case that $\alpha < n - 2$,
as observed in \eqref{1.22}; see also Remark \ref{3.10a}.
In the repulsive case $\kappa = -1$, the boundary term has a favorable sign
and so can be ignored.
However, in the attractive case $\kappa = 1$,
the boundary term cannot be controlled.
Thus, we use a Gr\"onwall-type inequality to bypass the use of
the
integration by parts
method
that gives us the problematic boundary term, which allows us to obtain a BD-type entropy estimate; see Lemma \ref{oldBD}.
In the case of the super-Coulombic Riesz potentials when $\alpha > n-2$,
the potential is no longer twice differentiable, so it is difficult to apply the method of integration by parts
of \cite{Chen2021} to obtain the BD entropy estimate.
In this paper, we are able to overcome such a problem by applying the same Gr\"{o}nwall-type inequality
as before to obtain the estimate for not only the attractive case but also the repulsive case;
see Lemma \ref{newbd}.

\smallskip
Finally, we deal with the potential for the case that $\alpha \in (-1,0]$, in particular the logarithmic potential $\alpha = 0$. This logarithmic case appears in many applications in the physical and mathematical biology literature.
However, it has been very problematic when trying to resolve the well-posedness of not only CEREs in general,
but also CEPEs in the $2$-D case.
Unlike the Riesz potential, there is no longer decay at infinity and, in the logarithmic case,
the potential has no definitive sign.
In the case $\alpha \in (0,n)$, we can employ the Hardy-Littlewood-Sobolev (HLS) inequality
to derive the $L^p$ estimate for $\Phi_\alpha * \rho$.
The parallel in the logarithmic case $\alpha = 0$ is that of the logarithmic
Hardy-Littlewood-Sobolev (logHLS) inequality, where we only obtain a one-sided bound
and which no longer gives us an $L^p$ bound.
In the case $\alpha \in (-1,0)$, we have the reverse HLS inequality that gives a lower bound
estimate for the convolution in terms of an $L^q$ bound for $q \in (0,1)$.
Such difficulties have made the study of CEREs indubitably difficult.
However, by studying the behavior of the logarithmic potential around the origin
and at the far-field separately,
we are able to locally control the $L^\infty$ norm of $\Phi_0 * \rho$,
allowing us to overcome the previous difficulties that arose by the use of the logHLS inequality only.
It is also shown that the energy estimates for the logarithmic potential can, at least formally,
be seen to be a limit of the energy estimates for the solutions of the general Riesz potentials
as $\alpha \to 0$, showing the connection between the Riesz and logarithmic potentials;
see Remark \ref{riesztolog}.

\smallskip
We note that the previous results were focused on proving the conditional stability
for radial perturbations of stationary solutions that rely
on the pre-assumed global existence of weak solutions.
One of our novelties here is that we can employ the first part of this paper to show that
the assumption of the global existence of finite-energy solutions close to such stationary solutions
is non-empty for the whole family of CEREs.
The nonlinear stability of stationary solutions for CEPEs
has been analyzed in \cite{Rein_2003,Luo_2008,Luo_2009,Lin_2023}.
Notice further that our results for CEREs are {\it unconditional} now,
by using the global existence results
obtained in this paper (see also \cite{Chen2021}). Moreover, we obtain the properties
of steady states by variational arguments building the connection
with the aggregation-diffusion equations; see Remark 2.7
and Theorem \ref{uniqueness of steady states}.

\smallskip
The main strategy and structure of the paper are as follows:
In \S \ref{S2}, we present the main results of this paper.
In \S \ref{S3}, we first derive the local estimates
for the radial representation of the Riesz potential and its derivatives.
Subsequently, we prove the uniform energy estimates of solutions of the approximate
free boundary problem to obtain the uniform $L^p$ estimates.
The next step is to prove a BD-type inequality to derive
the uniform $L^p$ estimates on the derivative of the density.
Finally, we use the local estimates for the Riesz potential
to prove the uniform local higher-integrability estimates
for the density and the velocity.
In \S \ref{S4}, we prove the main results concerning the global finite-energy solutions.
We also apply the Aubin-Lions lemma via the use of the BD-type estimates
to pass from the solutions
of the free boundary problem to the global weak solutions of CNSREs,
which satisfy the same uniform estimates.
We also employ the higher integrability estimates
in order to use the compensated compactness argument
from Chen-Perepelitsa \cite{Chen_2010} to pass from the global weak solutions of CNSREs to
the global weak solutions of CEREs.
Finally, in \S \ref{sec:stab}, we develop a simplified variational approach to deal with the nonlinear stability. We employ the concentration compactness argument from Lions \cite{Lions_1984}
to prove the existence of minimizers of \eqref{funcy} and the nonlinear stability near such minimizers.
The appendices contain details on the convolution inequalities used in the main text
and the construction of the approximate initial data.

\section{Mathematical Problems and Construction of Approximate Solutions}\label{S2}
In this section, we first introduce the notion of weak solutions and steady states of CERES and CNSRES,
then describe the main results of this paper, and finally
present the construction of approximate solutions
to obtain the global weak solutions.

\subsection{Weak solutions and steady states of CEREs and CNSREs}

We first assume that the given initial data function $(\rho_0,\M_0)(\boldsymbol{x})$ has finite total energy and
total mass:
\begin{align}
&E_0 := \int_{\mathbb{R}^n}\rho_0 \Big( \frac{1}{2} \Big| \frac{\M_0}{\rho_0}\Big|^2
+ e(\rho_0) + \frac{\kappa}{2} (\Phi_\alpha * \rho_0) \Big)(\boldsymbol{x})
\dd \boldsymbol{x} < \infty,\label{0.7}\\[2mm]
&M := \int_{\mathbb{R}^n} \rho_0(\boldsymbol{x}) \dd \boldsymbol{x} = \omega_n \int^\infty_0 \rho_0(r) r^{n-1} \dd r < \infty.\label{0.8}
\end{align}
When
$\alpha \in (-1,0]$, we also assume that $(\rho_0,\M_0)(\boldsymbol{x})$ has
finite $(-\alpha)$-moment:
\begin{equation}\label{0.81}
\int_{\mathbb{R}^n} \rho_0(\boldsymbol{x}) k_\alpha (1 + |\boldsymbol{x}|^2) \dd \boldsymbol{x} = \omega_n \int^\infty_0 \rho_0(r) k_\alpha (1 + r^2) r^{n-1} \dd r < \infty,
\end{equation}
where
\begin{equation*}
    k_\alpha(r) = \begin{cases}
        r^{- \frac{\alpha}{2}} \quad &\mbox{for $\alpha \in (-1,0)$}, \\
        \log(r) &\mbox{for $\alpha = 0$},
    \end{cases}
\end{equation*}
$\omega_n$ represents the surface area of the unit sphere in $\mathbb{R}^n$,
and $e(\rho) = \frac{a_0}{\gamma - 1}\rho^{\gamma - 1}$
is the internal energy satisfying the relation: $p(\rho) = \rho^2e'(\rho)$.
Here we require the initial data
such that the convolution $\Phi_\alpha*\rho_0$
is well-defined and differentiable.
We now give the definition of global finite-energy weak solutions
of \eqref{0.0}--\eqref{0.1}.

\begin{definition}\label{definition}
A measurable vector function $(\rho,\M)(t, \boldsymbol{x})$ is a finite-energy solution
of the Cauchy problem \eqref{0.0}--\eqref{0.1} if
$(\rho,\M)(t, \boldsymbol{x})$ satisfies the following conditions{\rm :}
\begin{enumerate}
\item[(a)] $\rho(t,\boldsymbol{x}) \geq 0$ a.e. and $(\M, \frac{\M}{\sqrt{\rho}})(t,\boldsymbol{x}) = \boldsymbol{0}$ a.e. on the vacuum states $\{ (t,\boldsymbol{x}) \, : \, \rho(t,\boldsymbol{x}) = 0 \}$,
\begin{align*}
			&\rho\in L^{\infty}(0,T; L^1\cap L^\gamma(\mathbb{R}^n)), \quad\,\,\,
			\frac{\M}{\sqrt{\rho}}\in L^{\infty}(0,T; L^2(\mathbb{R}^n)),\\
             &\Phi_\alpha * \rho\in L^\infty(0,T;L^{\frac{2n}{\alpha}}(\mathbb{R}^n) ),\quad
			\nabla\Phi_\alpha * \rho\in L^\infty(0,T;L^{\frac{2n}{\alpha + 2}}(\mathbb{R}^n)).
    \end{align*}
\item[(b)] For a.e. $t>0$, the total energy is finite{\rm :}

\smallskip
\begin{itemize}
 \item For $\kappa = -1$,
\begin{equation*}
 \int_{\mathbb{R}^n} \Big( \frac{1}{2} \Big| \frac{\M}{\sqrt{\rho}}\Big|^2
 + \rho e(\rho) - \frac{1}{2} \rho (\Phi_\alpha * \rho) \Big)(t,\boldsymbol{x})
 \dd \boldsymbol{x} \leq E_0.
\end{equation*}
\item For $\kappa = 1$,
\begin{align*}
\begin{cases}
\displaystyle
\int_{\mathbb{R}^n} \Big( \frac{1}{2} \Big| \frac{\M}{\sqrt{\rho}}\Big|^2
+ \rho e(\rho) - \frac{1}{2} \rho (\Phi_\alpha * \rho) \Big)(t,\boldsymbol{x})
\dd \boldsymbol{x} \leq C(E_0,M), \\[4mm]
\displaystyle
\int_{\mathbb{R}^n} \Big( \frac{1}{2}\Big|\frac{\M}{\sqrt{\rho}}\Big|^2
+ \rho e(\rho) + \frac{1}{2} \rho (\Phi_\alpha * \rho) \Big)(t,\boldsymbol{x})
\dd \boldsymbol{x} \leq E_0.
\end{cases}
\end{align*}
\end{itemize}
\item[(c)] For any $\zeta \in C^1_0(\mathbb{R} \times \mathbb{R}^n)$,
\begin{equation*}
\int_{\mathbb{R}_+} \int_{\mathbb{R}^n} \big(\rho \zeta_t
+ \M \cdot \nabla \zeta\big) \dd \boldsymbol{x} \dd t
+ \int_{\mathbb{R}^n} (\rho_0 \zeta)(0,\boldsymbol{x}) \dd \boldsymbol{x} = 0.
\end{equation*}
\item[(d)] For any $\boldsymbol{\psi} \in (C^1_0(\mathbb{R} \times \mathbb{R}^n))^n$,
    \begin{align*}
        & \int_{\mathbb{R}_+} \int_{\mathbb{R}^n}
        \Big(\M \cdot \boldsymbol{\psi}_t
        + \frac{\M}{\sqrt{\rho}} \cdot \big( \frac{\M}{\sqrt{\rho}} \cdot \nabla \big)
        \boldsymbol{\psi} + p(\rho)\nabla \cdot \boldsymbol{\psi} \Big)(t,\boldsymbol{x})
        \dd \boldsymbol{x} \dd t
        + \int_{\mathbb{R}^n} \M_0(\boldsymbol{x}) \cdot \boldsymbol{\psi}(0,\boldsymbol{x}) \dd \boldsymbol{x} \\
        & = \int_{\mathbb{R}_+} \int_{\mathbb{R}^n} \kappa(\rho \nabla \mathcal{W}_\alpha \cdot \boldsymbol{\psi})(t,\boldsymbol{x}) \dd \boldsymbol{x} \dd t.
    \end{align*}
\end{enumerate}
\end{definition}

We first construct global weak solutions of CNSREs with appropriately adapted density-dependent
viscosity terms and approximate initial conditions, constructed in Appendix B
with the properties as in Lemmas \ref{appendix1}--\ref{appendix5}, given by
\begin{equation}\label{initialns}
    (\rho^\varepsilon, \M^\varepsilon)(0,\boldsymbol{x}) = (\rho^\varepsilon_0, \M^\varepsilon_0)(\boldsymbol{x}).
\end{equation}
Now we provide the definition of global weak solutions of the Cauchy problem \eqref{0.2} and \eqref{initialns}.

\begin{definition}\label{CNSREs}
A measurable vector function $(\rho^\varepsilon,\M^\varepsilon)(t,\boldsymbol{x})$
is a weak solution
of the Cauchy problem \eqref{0.2} and \eqref{initialns} with $(\mu,\lambda) = (\rho,0)$
if $(\rho^\varepsilon,\M^\varepsilon)(t,\boldsymbol{x})$ satisfies
the following conditions{\rm :}
\begin{enumerate}
\item[(a)] $\rho^\varepsilon(t,\boldsymbol{x}) \geq 0$ a.e. and
$(\M^\varepsilon, \frac{\M^\varepsilon}{\sqrt{\rho^\varepsilon}})(t,\boldsymbol{x})
= \boldsymbol{0}$ a.e. on the vacuum states
$\{ (t,\boldsymbol{x}) \, : \, \rho^\varepsilon(t,\boldsymbol{x}) = 0 \}$,
    \begin{align*}
			&\rho^\v\in L^{\infty}(0,T; L^1\cap L^\gamma(\mathbb{R}^n)),
			\quad   \nabla\sqrt{\rho^\v}\in  L^{\infty}(0,T; L^2(\mathbb{R}^n)),\\
			&\frac{\M^\v}{\sqrt{\rho^\v}}\in L^{\infty}(0,T; L^2(\mathbb{R}^n)),
            \quad \Phi_\alpha * \rho^\v\in L^\infty(0,T;L^{\frac{2n}{\alpha}}(\mathbb{R}^n) ),\\
			&\nabla\Phi_\alpha * \rho^\v\in L^\infty(0,T;L^{\frac{2n}{\alpha + 2}}(\mathbb{R}^n)).
    \end{align*}
\item[(b)] For any $\zeta \in C^1_0(\mathbb{R} \times \mathbb{R}^n)$ and $t_2 \geq t_1 \geq 0$,
    \begin{equation*}
        \int_{t_1}^{t_2} \int_{\mathbb{R}^n} (\rho^\v \zeta_t + \M^\v \cdot \nabla \zeta) \dd \boldsymbol{x} \dd t = \int_{\mathbb{R}^n} (\rho^\v \zeta)(t_2,\boldsymbol{x}) \dd \boldsymbol{x} - \int_{\mathbb{R}^n} (\rho^\v \zeta)(t_1,\boldsymbol{x}) \dd \boldsymbol{x}.
    \end{equation*}
\item[(c)] For any $\boldsymbol{\psi} \in (C^1_0(\mathbb{R} \times \mathbb{R}^n))^n$,
    \end{enumerate}
    \begin{align*}
        \begin{split}
            & \int_{\mathbb{R}_+} \int_{\mathbb{R}^n}
            \Big( \M^\v \cdot \boldsymbol{\psi}_t
            + \frac{\M^\v}{\sqrt{\rho^\v}} \cdot \big( \frac{\M^\v}{\sqrt{\rho^\v}} \cdot \nabla\big) \boldsymbol{\psi} + p(\rho^\v)\nabla \cdot \boldsymbol{\psi} \Big)(t,\boldsymbol{x}) \dd \boldsymbol{x} \dd t + \int_{\mathbb{R}^n} \M^\v_0(\boldsymbol{x}) \cdot \boldsymbol{\psi}(0,\boldsymbol{x}) \dd \boldsymbol{x} \\
            & = - \v \int_{\mathbb{R}_+} \int_{\mathbb{R}^n}
            \Big( \frac{1}{2} \M^\v \cdot \big( \Delta \boldsymbol{\psi} + \nabla (\nabla \cdot \boldsymbol{\psi} ) \big)
            + \frac{\M^\v}{\sqrt{\rho^\v}} \cdot (\nabla \sqrt{\rho^\v} \cdot \nabla) \boldsymbol{\psi} + \nabla \sqrt{\rho^\v} \cdot \big( \frac{\M^\v}{\sqrt{\rho^\v}} \cdot \nabla \big) \boldsymbol{\psi} \Big) \dd \boldsymbol{x} \dd t \\
            & \quad\, + \int_{\mathbb{R}_+} \int_{\mathbb{R}^n} \kappa(\rho^\v \nabla \mathcal{W}^\v_\alpha \cdot \boldsymbol{\psi})(t,\boldsymbol{x}) \dd \boldsymbol{x} \dd t,
        \end{split}
    \end{align*}
    where
     \[
 \nabla \mathcal{W}^\v_\alpha(\boldsymbol{x}) =
 \begin{cases}
 \displaystyle
     (\nabla \Phi_\alpha * \rho^\v) (\boldsymbol{x}) & \text{for } \alpha \in (-1,n-1), \\[1mm]
     \displaystyle
     \int_{\mathbb{R}^n} \nabla \Phi_\alpha(\boldsymbol{x} - \boldsymbol{y}) \big(\rho^\v(\boldsymbol{y}) - \rho^\v(\boldsymbol{x})\big) \dd \boldsymbol{y}
     \quad& \text{for } \alpha \in [n-1,n).
 \end{cases}
 \]
\end{definition}

One of the important results of this paper is the nonlinear stability
of the unique minimizer to the energy functional relating to CEREs.
For $\alpha \in (0,n)$ and $0<M<\infty$, define the set of admissible
functions{\rm:}
\begin{equation}\label{2.3a}
X_M := \Big\{ \varrho \in L^1_+(\mathbb{R}^n) \cap L^\frac{n + \alpha}{n}(\mathbb{R}^n)
\, : \, \int_{\mathbb{R}^n} \varrho(\boldsymbol{x}) \, \dd \boldsymbol{x} = M \Big\}.
\end{equation}
For $\varrho\in X_M$, define the energy functional $\mathcal{G}$:
\begin{equation}\label{2.4a}
\mathcal{G}(\varrho)=\int_{\mathbb{R}^n} \Big(\varrho(\boldsymbol{x}) e(\varrho(\boldsymbol{x}))+\frac{1}{2}\varrho(\boldsymbol{x})(\Phi_\alpha\ast \varrho)(\boldsymbol{x})\Big)\,\dd\boldsymbol{x}.
\end{equation}
For $\varrho, \tilde{\rho} \in X_M,$  define the distance:
\begin{equation}\label{2.5a}
d(\varrho,\tilde{\rho}) =\int_{\mathbb{R}^n}\big((\varrho e(\varrho)-\tilde{\rho}e(\tilde{\rho}))+(\Phi_\alpha\ast\tilde{\rho})(\varrho-\tilde{\rho})\big)\,\dd \boldsymbol{x}.
\end{equation}
Moreover, it can be shown that, in some cases, the unique minimizer of the energy
functional $\mathcal{G}$ in $X_M$ exactly corresponds to the unique steady state
of \eqref{0.0}, that is,
\begin{equation*}
    \nabla p + \kappa \varrho \nabla \mathcal{W}_\alpha = \boldsymbol{0};
\end{equation*}
see \S \ref{sec:stab} for more details.
Then this, paired with a nonlinear stability result for the minimizer of $\mathcal{G}$,
allows us to show that the steady state is nonlinearly stable.

\begin{definition}\label{steady}
$\tilde{\rho} \in L^1_+(\mathbb{R}^n) \cap L^\infty(\mathbb{R}^n)$ is
called a steady state of \eqref{0.0}
if $\tilde{\rho}^\gamma \in {W}^{1,2}_\text{\rm loc}(\mathbb{R}^n)$ and $\mathcal{\tilde{W}}_\alpha=\Phi_\alpha\ast \tilde{\rho}$ with
$\nabla \mathcal{\tilde{W}}_\alpha \in L^1_\text{\rm loc}(\mathbb{R}^n)$, and
\[
\nabla p(\tilde{\rho}) + \kappa \tilde{\rho} \nabla \mathcal{\tilde{W}}_\alpha = \boldsymbol{0}
\]
is satisfied in the sense of distributions{\rm ;} in addition,
when $\alpha \in [n-1,n)$,
$\tilde{\rho} \in C^{0,\beta}(\mathbb{R}^n)$ for some $\beta \in (1 + \alpha - n,1)$.
\end{definition}

Notice that the notion of being a steady state of \eqref{0.0} is equivalent to the notion of being a steady state of the aggregation-diffusion equation
\[
\frac{\partial \rho}{\partial t}=\nabla\cdot(\nabla p(\rho) + \kappa \rho \nabla \mathcal{W}_\alpha)),
\]
see for instance \cite{Carrillo_2018,Carrillo_2019,CCY19}. We will summarize the state-of-the-art for steady states in Theorem \ref{uniqueness of steady states} in \S\ref{sec:stab}.

In the case that, for any fixed $\alpha$, there exists a unique steady solution $\rho_s$ modulo mass normalization and translations, then $\rho_s$ is referred to as
the Lane-Emden-Fowler type solution.
A detailed discussion on the case of the existence and uniqueness modulo symmetries
is summarized in Theorem \ref{exisunisteady} and postponed to \S \ref{sec:stab}.
The steady state
$\rho_s(|\boldsymbol{x}|)$ has compact support and is determined by the equation:
\begin{equation*}
\nabla p_s + \kappa \rho_s\,\nabla \Phi_{\alpha} * \rho_s=\boldsymbol{0},
\end{equation*}
where $p_s(|\boldsymbol{x}|)=a_0(\rho_s(|\boldsymbol{x}|))^{\frac{n + \alpha}{n}}$.
In particular, for $\alpha = n-2$, the steady solution is the Lane-Emden solution.

\subsection{Main results}
We first have the following global existence theorem:

\begin{theorem}[Existence of Spherically Symmetric Solutions of CEREs]\label{existence}
Consider problem \eqref{0.0}--\eqref{0.1} with $\alpha \in (-1,n-1)$
and the initial data of spherical symmetry as defined in \eqref{0.3}--\eqref{0.4}.
Assume that $(\rho_0,\M_0)$ satisfy \eqref{0.7}--\eqref{0.8}, $\gamma > \frac{1}{n-1-\alpha}$, and
\begin{itemize}
    \item When $\kappa = -1$,  $\,\,\rho_0 \in L^{\frac{2n}{2n - \alpha}}(\mathbb{R}^n)$ and
    \begin{itemize}
        \item [(a)] $\gamma > 1$ for $\alpha \in (-1,n-2]$,
        \item[(b)] $\gamma \geq \frac{3n}{3n - 2(1+\alpha)}$ for $\alpha \in (n-2,n-1)${\rm ;}
    \end{itemize}
    \item When $\kappa = 1$, $\,\,\gamma \geq \frac{3n}{3n - 2(1+\alpha)}$
when $\alpha \neq n-2$ and
    \begin{itemize}
        \item[(c)]  $\gamma > 1$ for $\alpha \in (-1,0]$,
        \item [(d)] $\gamma > \frac{n + \alpha}{n}$ for $\alpha \in (0,n-1)$,
        \item[(e)] $M < M_{\rm c}^\alpha(\gamma)$ for
         $\gamma \in (\frac{2n}{2n - \alpha}, \frac{n + \alpha}{n}]$ and $\alpha \in (0,n-1)$.
    \end{itemize}
\end{itemize}
Then there exists a global spherically symmetric finite-energy solution
$(\rho,\M)(t,\boldsymbol{x})$ of problem \eqref{0.0}--\eqref{0.1}
in the sense of {\rm Definition \ref{definition}}.
\end{theorem}

In Theorem \ref{existence}, $M_{\rm c}^\alpha(\frac{n + \alpha}{n})$ for $\alpha \in (0,n)$
is the sharp critical mass, analogous to that of the Chandrasekhar limit
mass $M_{\rm ch}:=M_{\rm c}^{n-2}(\frac{2(n-1)}{n})$ for steady gaseous
stars \cite{Chandrabook}.
It is known that there is no steady white dwarf star of total mass $M > M_{\rm ch}$;
that is, there is no steady solution $(\rho_s(|\boldsymbol{x}|),\boldsymbol{0})$ for $p(\rho)=a_0\rho^{\frac{2n - 2}{n}}$ for this case.
The analogous dichotomy and the sharp critical mass for all $\alpha \in (0,n)$
were obtained in \cite{Carrillo_2018} in relation to variants of the HLS inequalities;
see \cite{Blanchet_2008} for the Coulomb potential case $\alpha=n-2$.

The critical mass for $\alpha > 0$ and $\gamma \in (\frac{2n}{2n-\alpha}, \frac{n + \alpha}{n}]$
is defined by
\begin{equation*}
 M^{\alpha}_{\rm c}(\gamma) = \begin{cases}
    \displaystyle
    B_{n,\gamma,\alpha}^{-\frac{n}{n - \alpha}} & \mbox{when $\gamma = \frac{n + \alpha}{n}$},\\[1mm]
    \displaystyle
     \Big(\frac{\alpha B_{n,\gamma,\alpha}}{n(\gamma - 1)} \Big)^{\frac{n(\gamma - 1)}{2n - \gamma(2n-\alpha)}} \Big(\frac{\alpha E_0}{\alpha - n(\gamma - 1)}\Big)^{\frac{\alpha - n(\gamma - 1)}{2n - \gamma(2n - \alpha)}}
     \quad& \mbox{when $\gamma \in (\frac{2n}{2n - \alpha}, \frac{n + \alpha}{n})$},
\end{cases}
\end{equation*}
where
\begin{equation*}
B_{n,\gamma,\alpha} = \frac{C_*(\alpha,\gamma,n)}{2\alpha}\omega_n^{\frac{\alpha}{n(\gamma - 1)}-1}
\Big(\frac{\gamma - 1}{a_0}\Big)^{\frac{\alpha}{n(\gamma - 1)}},
\end{equation*}
and $C_*(\alpha,\gamma,n) > 0$ is the sharp constant given by the variation of the HLS inequality from Lemma \ref{HLSineq}. For $\alpha=n-2$, the bound on the critical mass
$M_{\rm c}^{n-2}(\gamma)=B_{n,\gamma,n-2}^{-\frac{n}{2}}$ for $\gamma<\frac{2(n-1)}{n}$
was obtained in \cite{Chen2021}.

Furthermore, our results cover the special case that $\alpha=n-2$ and the whole range of adiabatic exponents $\gamma>1$ when $n=2$. In this case, the critical mass condition is not
required even for the attractive case, and there is no concentration formed for the density at the origin.


\begin{remark}
In {\rm Theorem \ref{existence}}, we see that
$\frac{2n}{2n - \alpha} < \frac{3n}{3n - 2(1+\alpha)}$
for all $\alpha \in (-1,n-1)$ and there only exists $\alpha$ such that
$\frac{n + \alpha}{n} > \frac{3n}{3n - 2(1+\alpha)}$ for $n \geq 20$.
In this case, $\frac{n + \alpha}{n} > \frac{3n}{3n - 2(1+\alpha)}$
for $\alpha \in (\alpha_-,\alpha_+)$,
where $\alpha_{\pm} = \frac{n-2 \pm \sqrt{n^2 - 20n + 4}}{4}$ with
$\frac{\alpha_-}{n}\to 0$ and $\frac{\alpha_+}{n} \to \frac{1}{2}$
as $n \to \infty$.
Thus, assumption {\rm (e)} of {\rm  Theorem \ref{existence}} is only relevant for $n \geq 20$
and $\alpha \in (\alpha_-,\alpha_+)$. This can be observed in {\rm Fig. 1}.
\end{remark}

\begin{figure}
    \centering
    \includegraphics[width=0.6\linewidth]{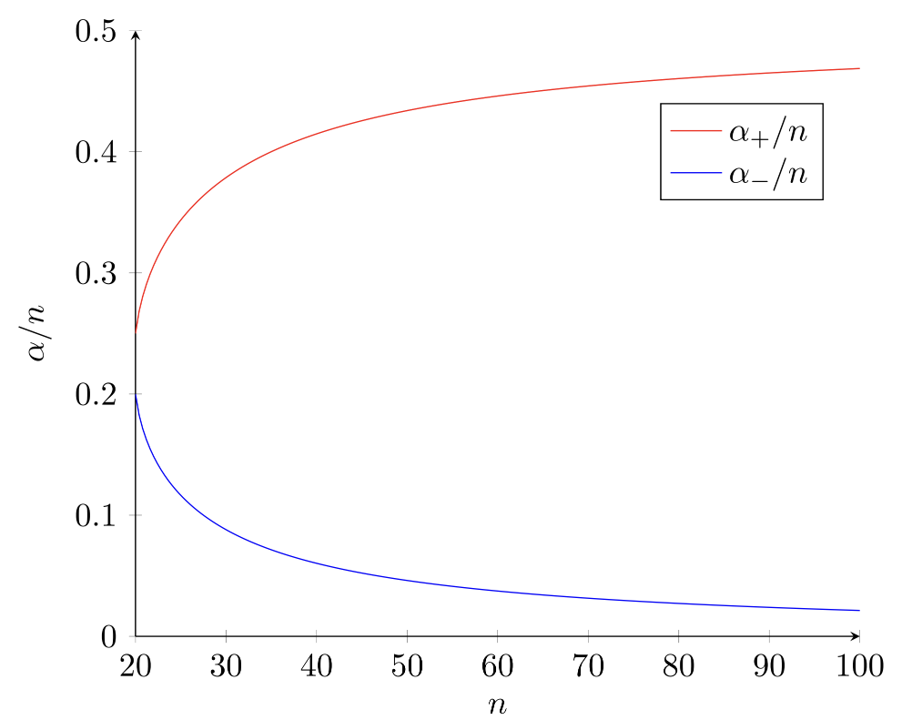}
    \caption{Relation between $\alpha$ and $n$.}
\end{figure}

The nonlinear stability theorem is as follows:

\begin{theorem}[Nonlinear Stability of Steady States for CEREs]\label{ns}
Suppose that $\tilde{\rho}$ is a unique minimizer $($up to translation$)$
of the functional $\mathcal{G}$ in $X_M$.
Let $\alpha \in (0,n - 1)$ with $\gamma > \frac{n + \alpha}{n}$.
Let $(\rho,\mathcal{M})(t,\boldsymbol{x})$ be global finite-energy solutions of the attractive
CEREs {\rm (}$\kappa = 1${\rm )} in the sense of {\rm Definition \ref{definition}}
satisfying
$$
\int_{\mathbb{R}^n} \rho_0(\boldsymbol{x})\,\dd \boldsymbol{x}=\int_{\mathbb{R}^n} \tilde{\rho}(\boldsymbol{x})\,\dd \boldsymbol{x}=M.
$$
Then, for any $\v>0$, there exists $\delta = \delta(\varepsilon)>0$ such that, if
\begin{align}\label{close}
d(\rho_0,\tilde{\rho})+ \| \rho_0 - \tilde{\rho} \|_{L^\frac{2n}{2n - \alpha}}^2
+\frac{1}{2}\int_{\mathbb{R}^n} \Big|\frac{\mathcal{M}_0}{\sqrt{\rho_0}}\Big|^2\,\dd \boldsymbol{x}<\delta,
\end{align}
there exists a translation $\boldsymbol{y}\in \R^n$ such that
\begin{align}\label{closer}
d(\rho(t),T^{\boldsymbol{y}}\tilde{\rho})+ \| \rho(t) - T^{\boldsymbol{y}}\tilde{\rho} \|_{L^\frac{2n}{2n - \alpha}}^2+\frac{1}{2}\int_{\mathbb{R}^n} \Big|\frac{\mathcal{M}}{\sqrt{\rho}}\Big|^2(t)\,\dd \boldsymbol{x}<\v
\qquad \mbox{for a.e. $t > 0$},
\end{align}
 where $T^{\boldsymbol{y}}\tilde{\rho}(\boldsymbol{x})=\tilde{\rho}(\boldsymbol{x}+\boldsymbol{y})$
 and, for a.e. $t>0$, $\rho(t):=\rho(t,\cdot)\in X_M$ for the density $\rho(t,\boldsymbol{x})$ of the global finite-energy solution.
\end{theorem}

\begin{remark} For {\rm Theorem \ref{ns}}, we have the following comments{\rm :}
\begin{itemize}
    \item[i)] The proof follows the same strategy of {\rm \cite{Rein_2003}} for CEPEs. Concerning the existence proof of the minimizers of the free energy with homogeneous interaction potentials,  different from Luo-Smoller {\rm \cite{Luo_2008,Luo_2009}}, we develop a simplified approach by employing the work of {\rm \cite{Lions_1984}}, as it was done in {\rm \cite{Carrillo_2018}}.
    These methods may be further developed to handle the general pressure law case and other related problems, which will be explored elsewhere.
\item[ii)]
We summarize the properties of global minimizers and the uniqueness of steady states
in {\rm Theorem \ref{uniqueness of steady states}}. As already mentioned, steady states of aggregation-diffusion equations are equivalent to steady states of CEREs or CEPEs.
The results we have obtained cover a wider range of potentials than
the previous results for CEREs and CEPEs, in particular, including $n=2.$
\end{itemize}
\end{remark}

\subsection{Approximate PDEs}
Some of the main difficulties we face for problem \eqref{0.0}--\eqref{0.1} and \eqref{0.2} in the whole space $\mathbb{R}^n$ are that the solutions may involve cavitation, concentration,
and shock waves.
In order to avoid these difficulties, we propose the following formation
of a free boundary problem for our approximate solutions:
\begin{align}\label{Eul}\begin{cases}
    \displaystyle
    \rho_t + (\rho u)_r + \frac{n-1}{r}\rho u = 0, \\
    \displaystyle
    (\rho u )_t + (\rho u^2 + p(\rho))_r + \frac{n-1}{r}\rho u^2 + \kappa \rho(\Phi_\alpha * \rho)_r
    = \varepsilon\Big( \rho \big (u_r + \frac{n-1}{r}u \big ) \Big)_r - \varepsilon\frac{n-1}{r}\rho_ru.
\end{cases}
\end{align}
This is defined on the moving domain:
$$
\Omega^T = \big\{ (t,r) \, :\, 0 \leq t \leq T, \, a \leq r \leq b(t)\big\},
$$
where $b(t)$ is a free boundary that is subject to:
\begin{equation}\label{0.19}
\begin{cases}
    b' (t) = u(t,b(t)) & \mbox{ for $t>0$}, \\[1mm]
    b(0) = b,
\end{cases}
\end{equation}
with boundary conditions:
\begin{equation}\label{boundary}
u(t,a) = 0, \quad
\big(p(\rho)-\varepsilon \rho (u_r + \frac{n-1}{r}u) \big)(t,b(t)) = 0
\qquad\,\,\, \text{for $t >0$},
\end{equation}
for $a:= b^{-1}$ with $b \gg 1$.

The initial conditions are prescribed by
\begin{equation}\label{0.20}
    (\rho,u)|_{t=0}(r) = (\rho_0^{\v,b},u_0^{\v,b})(r) \qquad \text{ for $r \in [a,b]$},
\end{equation}
where $(\rho_0^{\v,b},u_0^{\v,b})(r)$ are defined in the appendix
and satisfy that, for each $(\v,b)$, there exists some $C_{\v,b} > 0$ such that
\begin{equation}\label{0.21}
    0 < C_{\v,b}^{-1} \leq \rho_0^{\v,b}(r) \leq C_{\v,b} < \infty.
\end{equation}
We define the spatial domain $\Omega_t$ of the free boundary problem at time $t \geq 0$ to be
\[
\Omega_t = \big\{ \boldsymbol{x} \in \mathbb{R}^n \, : \, a \leq |\boldsymbol{x}| \leq b(t)\big\}.
\]
Here, system \eqref{Eul} is just system \eqref{0.6} with $(\mu,\lambda) = (\rho,0)$
written with $u = \frac{m}{\rho}$.
Owing to the fact that, under this construction as a free boundary problem,
we can prove the existence of smooth solutions via a similar approach to \cite{Duan_2015},
which have a strictly positive density $\rho^{\v,b} > 0$ due to \eqref{0.21}
so that the fluid velocity $u^{\v,b}$ is well defined and hence this construction is equivalent.
More precisely, there exists $\tilde{C}_{\v,b,T} > 0$ such that, for $t \in [0,T]$,
\begin{equation}\label{away}
    0 < \tilde{C}_{\v,b,T}^{-1} \leq \rho^{\v,b}(t,r) \leq \tilde{C}_{\v,b,T} < \infty.
\end{equation}
Note that the boundary condition in \eqref{boundary} on the free boundary is
a stress-free boundary condition, making the problem physically relevant,
due to the fact that, under such an assumption, the boundary moves with
the velocity of the fluid as modeled by the free boundary.
We denote the solution of the free boundary problem \eqref{Eul}--\eqref{0.20}
by $(\rho^{\v,b},u^{\v,b})$, but will suppress the superscript
when no confusion arises until \S 4 for simplicity.

\begin{remark}
From $\eqref{Eul}_1$ and \eqref{boundary}, as in {\rm \cite{Chen2021}}, we have
\begin{equation}\label{ineq}
    \rho(t,b(t)) \leq \rho_0(b).
\end{equation}
\end{remark}

The critical mass, as before, is defined for $\alpha > 0$ as
\begin{align}\label{critical}
M^{\varepsilon,b,\alpha}_{\rm c}(\gamma)
= \begin{cases}
 B_{n,\gamma,\alpha}^{-\frac{n}{n - \alpha}} & \mbox{ for $\gamma = \frac{n + \alpha}{n}$},\\
 \Big(\frac{\alpha B_{n,\gamma,\alpha}}{n(\gamma - 1)} \Big)^{\frac{n(\gamma - 1)}{2n - \gamma(2n-\alpha)}} \Big(\frac{\alpha E^{\varepsilon,b}_0}{\alpha - n(\gamma - 1)}\Big)^{\frac{\alpha - n(\gamma - 1)}{2n - \gamma(2n - \alpha)}} & \mbox{ for $\gamma \in (\frac{2n}{2n - \alpha}, \frac{n + \alpha}{n})$},
\end{cases}
\end{align}
where the total energy $E^{\varepsilon,b}_0$ is defined in Definition \ref{approxenergy}. $M^{\varepsilon,\alpha}_{\rm c}(\gamma)$ is defined the same way as $M^{\varepsilon,b,\alpha}_{\rm c}(\gamma)$ in \eqref{critical}, except $E^{\varepsilon,b}_0$ is replaced by $E^{\varepsilon}_0$, where $E^{\varepsilon}_0$ is defined in Definition \ref{approxenergy}.
The construction of $\rho^{\v,b}_0$ implies that the following equality for mass holds:
\begin{equation*}
    \int^b_a \rho^{\v,b}_0(r)\, r^{n-1} \dd r = \frac{M}{\omega_n}.
\end{equation*}
We can simply use $\eqref{Eul}_1$ and \eqref{0.19}--\eqref{boundary} to prove the
conservation of mass:
\begin{equation*}
    \frac{\d}{\d t} \int^{b(t)}_a \rho(t,r)\, r^{n-1} \dd r = (\rho u)(t,b(t))b(t)^{n-1} - \int^{b(t)}_a (\rho u r^{n-1})_r(t,r) \dd r = 0.
\end{equation*}
Thus,  for all $t \geq 0$,
\begin{equation*}
    \int^{b(t)}_a \rho(t,r)\, r^{n-1} \dd r = \frac{M}{\omega_n}.
\end{equation*}
Using this, we can define the Lagrangian coordinates $(\tau,x)=(t, x(t,r))$ as in \cite{Chen2021} by
\begin{equation*}
    x(t,r) = \int^r_a \rho (t,y)\,y^{n-1} \dd y.
\end{equation*}
It is direct to see that this is a smooth bijective mapping from $[0,T] \times [a,b(t)]$
to $[0,T] \times [0, \frac{M}{\omega_n}]$, which is incredibly useful
due to the fact that the latter domain is fixed.
Then we have
\begin{equation}\label{deriv}
\nabla_{(t,r)} x = (-\rho u r^{n-1},\rho r^{n-1}), \,\,\,\, \nabla_{(t,r)} \tau = (1,0),
\,\,\,\, \nabla_{(\tau,x)} r = (u, \rho^{-1}r^{1-n}), \,\,\,\,\nabla_{(\tau,x)} t = (1,0).
\end{equation}
Under the changes of variables, the boundary conditions become
\begin{equation}\label{bound}
    u(\tau,0) = 0, \quad \big(p-\varepsilon\rho^2(r^{n-1}u)_x\big)(\tau,\frac{M}{\omega_n}) = 0
    \qquad\,\, \text{ for $\tau \in [0,T]$},
\end{equation}
and \eqref{Eul} can be written in the form:
\begin{align}\label{lag}
\begin{cases}
    \rho_\tau + \rho^2(r^{n-1}u)_x = 0,\\[1mm]
    u_\tau + r^{n-1}p_x + \kappa \rho r^{n-1}(\Phi_\alpha * \rho)_x
    = \varepsilon r^{n-1}\big(\rho^2(r^{n-1}u)_x\big)_x - (n-1)\varepsilon r^{n-2}\rho_xu.
\end{cases}
\end{align}

\section{Uniform Estimates of the Approximate Solutions}\label{S3}

In this section, we carefully make uniform estimates of the approximate solutions
constructed in \S 2.3 above.

\subsection{Estimates on the radial non-local potential}
Given a radial function
$f \in L^1(\mathbb{R}^n) \cap L^q_{\rm loc}(\mathbb{R}^n \setminus \{\boldsymbol{0}\})$
for $q \in (\frac{n}{n - 1 - \alpha}, \infty)$,
$(\Phi_\alpha*f)(\boldsymbol{x})$ is a radial function so that we are able to represent
it as a convolution of some potentials with the radial function.
Like in \cite{BALAGUE20135}, we can represent the convolution by
\begin{equation*}
    (\Phi_\alpha*f)(\boldsymbol{x}) = \int^{\infty}_a K(|\boldsymbol{x}|,\eta) f(\eta) \eta^{n-1}
    \dd \eta,
\end{equation*}
where
\begin{equation*}
    K(r,\eta) := \begin{cases}
                        \displaystyle
                        \int_{\mathbb{S}^{n-1}} \frac{|r\boldsymbol{e_1}-\eta \boldsymbol{y}|^{-\alpha}}{-\alpha} \dd \sigma (\boldsymbol{y}) & \mbox{ for $\alpha \neq 0$}, \\[4mm]
                        \displaystyle
                        \int_{\mathbb{S}^{n-1}} \log(|r\boldsymbol{e_1}-\eta \boldsymbol{y}|) \dd \sigma (\boldsymbol{y}) \quad &\mbox{ for $\alpha = 0$}, \\
                \end{cases}
\end{equation*}
with $\boldsymbol{e_1} = (1,0,\cdots,0)$.
We now consider the differentiability of such a convolution and the potential that can be used to represent the derivative.

\begin{lemma}[Continuity of the Potential]\label{potentially}
Let $\alpha \in (-1,n-1)$.
Given a radial function $f \in L^1(\mathbb{R}^n) \cap L^q_{\rm loc}(\mathbb{R}^n \setminus \{\boldsymbol{0}\})$ for some $q \in (\frac{n}{n - 1 - \alpha}, \infty)$ such that $\Phi_\alpha*f$
and $\nabla\Phi_\alpha*f$ are well-defined,
then $(\Phi_\alpha*f)(\boldsymbol{x})$ is continuously differentiable on $\mathbb{R}^n \setminus \{\boldsymbol{0}\}$, and
    \begin{equation*}
        (\nabla \Phi_\alpha*f)(\boldsymbol{x}) \cdot \frac{\boldsymbol{x}}{|\boldsymbol{x}|}
        = (\Phi_\alpha*f)_r(|\boldsymbol{x}|)
        = \int^{\infty}_0 \omega(|\boldsymbol{x}|,\eta) f(\eta)\,\eta^{n-1}\dd \eta
    \end{equation*} is a radial function, where
    \[
    \omega(r,\eta) := \int_{\mathbb{S}^{n-1}} \nabla \Phi_\alpha(r\boldsymbol{e_1} - \eta \boldsymbol{y}) \cdot \boldsymbol{e_1} \dd \sigma(\boldsymbol{y}).
    \]
Moreover, there exists $C = C(\alpha,n) > 0$ such that
\begin{align}\label{continuity.8}
\begin{split}
\begin{cases}
\displaystyle
|\omega(r,\eta)| \leq C (r\eta)^{-\frac{\alpha + 1}{2}} & \mbox{ for $\alpha \in(-1,n-2)$}, \\[1mm]
                \displaystyle
                \omega(r,\eta) = \frac{\omega_n \mathds{1}_{(0,r)}(\eta)}{r^{n-1}}
                 & \mbox{ for $\alpha = n-2$}, \\[1mm]
                \displaystyle
                |\omega(r,\eta)| \leq C(r\eta)^{-\frac{n-1}{2}} |r-\eta|^{n-2-\alpha}
                & \mbox{ for $\alpha \in (n-2,n-1)$},
\end{cases}
\end{split}
\end{align}
    where $\mathds{1}_{(0,r)}$ is the indicator function on interval $(0,r)$.
\end{lemma}

\begin{proof}  We divide the proof into five steps.

\smallskip
1. For $\alpha \in (0,n-1)$, since $q > \frac{n}{n - 1 - \alpha}$, then, by Theorem \ref{pointwise},
there exists $C_1 > 0$ such that, for $\varepsilon > 0$, the function:
\begin{align*}
    \boldsymbol{z}\, \mapsto\, \int_{B_1(\boldsymbol{x})} \Phi_\alpha(\boldsymbol{z} - \boldsymbol{y}) \mathds{1}_{B_\v(\boldsymbol{x})}(\boldsymbol{y}) f(\boldsymbol{y}) \dd \boldsymbol{y}
\end{align*}
is in $C^{0,1}(\overline{B_1(\boldsymbol{x})})$ for $|\boldsymbol{x}| > 0$, with Lipschitz constant
bounded by $C_1 \| f \|_{L^q(B_\v(\boldsymbol{x}))}$ for $\v > 0$ small enough,
where $\overline{B_1(\boldsymbol{x})}$ denotes the closure of the open ball $B_1(\boldsymbol{x})$ of radius 1 centered at $\boldsymbol{x}$.
For $\alpha \in (0,n)$ and $0<|\boldsymbol{z}| \leq 1$,
\begin{align*}
\begin{split}
& \frac{1}{|\boldsymbol{z}|} \bigg|\int_{B_\v(\boldsymbol{x})} (\Phi_\alpha(\boldsymbol{x} + \boldsymbol{z} - \boldsymbol{y}) - \Phi_\alpha(\boldsymbol{x} - \boldsymbol{y})) f(\boldsymbol{y}) \dd \boldsymbol{y} \bigg| \\
& = \frac{1}{\alpha |\boldsymbol{z}|} \bigg| \int_{B_1(\boldsymbol{x})} \left ( |\boldsymbol{x} + \boldsymbol{z} - \boldsymbol{y}|^{-\alpha} - |\boldsymbol{x} - \boldsymbol{y}|^{-\alpha} \right ) \mathds{1}_{B_\v(\boldsymbol{x})}(\boldsymbol{y}) f(\boldsymbol{y}) \dd \boldsymbol{y} \bigg| \\
& \leq \frac{C_1}{\alpha} \| f \|_{L^q(B_\v(\boldsymbol{x}))} \rightarrow 0
\qquad\,\, \mbox{ uniformly in $\boldsymbol{z}\,\,$ as $\v \to 0$}.
\end{split}
\end{align*}

\smallskip
2. For $\alpha \in (-1,0]$, since $q > \frac{n}{n - 1 - \alpha}$, then,  by Theorem \ref{pointwise},
there exists $C_2 > 0$ such that, for $\varepsilon > 0$, the function:
\begin{align*}
    \boldsymbol{z}\, \mapsto\, \int_{B_1(\boldsymbol{x})} |\boldsymbol{z} - \boldsymbol{y}|^{-(\alpha + 1)} \mathds{1}_{B_\v(\boldsymbol{x})}(\boldsymbol{y}) f(\boldsymbol{y}) \dd \boldsymbol{y}
\end{align*}
is in $C(\overline{B_1(\boldsymbol{x})})$ for $|\boldsymbol{x}| > 0$ such that,
for $\v$ small enough,
\begin{align*}
    \Big\|\int_{B_1(\boldsymbol{x})} |\cdot - \boldsymbol{y}|^{-(\alpha + 1)} \mathds{1}_{B_\v(\boldsymbol{x})}(\boldsymbol{y}) f(\boldsymbol{y}) \dd \boldsymbol{y} \Big\|_{C(B_1(\boldsymbol{x}))}
    \leq C_2 \|f\|_{L^q(B_\varepsilon(\boldsymbol{x}))}.
\end{align*}
For $0<|\boldsymbol{z}| \leq 1$,
$\boldsymbol{0} \not \in \{ \boldsymbol{x} + s \boldsymbol{z} - \boldsymbol{y} \, : \, s \in [0,1] \}$
for {\it a.e.} $\mathbf{y}\in B_1(\boldsymbol{x})$.
Thus, for {\it a.e.} $\mathbf{y}\in B_1(\boldsymbol{x})$ (that is,
for $\boldsymbol{y}$ such that $\boldsymbol{0} \not \in \{ \boldsymbol{x} + s \boldsymbol{z} - \boldsymbol{y} \, : \, s \in [0,1] \}$), we have
\begin{equation*}
    |\Phi_\alpha(\boldsymbol{x} + \boldsymbol{z} - \boldsymbol{y}) - \Phi_\alpha(\boldsymbol{x} - \boldsymbol{y})|
    = \Big|\int^1_0 \frac{\d}{\d s} \Phi_\alpha(\boldsymbol{x} + s \boldsymbol{z} - \boldsymbol{y}) \dd s \Big|
    = \Big| \boldsymbol{z} \cdot \int^1_0 \frac{\boldsymbol{x} + s \boldsymbol{z} - \boldsymbol{y}}{|\boldsymbol{x}
    + s \boldsymbol{z} - \boldsymbol{y}|^{\alpha + 2}} \dd s \Big|.
\end{equation*}
Therefore, for $\alpha \in (-1,0]$ and $0<|\boldsymbol{z}| \leq 1$,
\begin{align*}
    \begin{split}
        & \frac{1}{|\boldsymbol{z}|}
        \bigg| \int_{B_\v(\boldsymbol{x})} (\Phi_\alpha(\boldsymbol{x} + \boldsymbol{z} - \boldsymbol{y}) - \Phi_\alpha(\boldsymbol{x} - \boldsymbol{y})) f(\boldsymbol{y}) \dd \boldsymbol{y} \bigg| \\
        & \leq \int^1_0 \int_{B_1(\boldsymbol{x})}|\boldsymbol{x} + s \boldsymbol{z} - \boldsymbol{y}|^{-(\alpha + 1)} \mathds{1}_{B_\v(\boldsymbol{x})}(\boldsymbol{y}) |f(\boldsymbol{y})| \dd \boldsymbol{y} \dd s \\
        & \leq C_2 \| f \|_{L^q(B_\v(\boldsymbol{x}))} \rightarrow 0
        \qquad\,\,\mbox{ as $\v \to 0$.}
    \end{split}
\end{align*}

\smallskip
3. For $\alpha \in (-1,n - 1)$, we see that, for fixed $\v > 0$,
\begin{align*}
    \frac{|\Phi_\alpha(\boldsymbol{x} + \boldsymbol{z} - \boldsymbol{y}) - \Phi_\alpha(\boldsymbol{x}
    - \boldsymbol{y}) - \boldsymbol{z} \cdot \nabla \Phi_\alpha(\boldsymbol{x} - \boldsymbol{y})|}{|\boldsymbol{z}|} \longrightarrow 0
\end{align*}
uniformly for $\boldsymbol{y} \in \mathbb{R}^n \setminus B_\v(\boldsymbol{x})$ as $\boldsymbol{z} \to \boldsymbol{0}$.
Since $f \in L^1$, we obtain
\begin{align*}
\begin{split}
    & \frac{1}{|\boldsymbol{z}|} \bigg| \int_{B_\v^c(\boldsymbol{x})}
    \big(\Phi_\alpha(\boldsymbol{x} + \boldsymbol{z} - \boldsymbol{y})-\Phi_\alpha(\boldsymbol{x}- \boldsymbol{y})\big) f(\boldsymbol{y}) \dd \boldsymbol{y} - \boldsymbol{z} \cdot \int_{B^c_\v(\boldsymbol{x})} \nabla \Phi_\alpha (\boldsymbol{x} - \boldsymbol{y}) f(\boldsymbol{y}) \dd \boldsymbol{y} \bigg| \\[2mm]
    &\,\,\longrightarrow 0 \qquad \text{ as $\boldsymbol{z} \to 0$}.
    \end{split}
\end{align*}
Then we conclude that $(\Phi_\alpha*f)(\boldsymbol{x})$ is differentiable on
$\mathbb{R}^n \setminus \{\mathbf{0}\}$ with $\nabla(\Phi_\alpha*f)(\boldsymbol{x})
= (\nabla \Phi_\alpha*f)(\boldsymbol{x})$.

\smallskip
4. As in \cite{BALAGUE20135}, we can write
    \[
    (\nabla \Phi_\alpha*f)(\boldsymbol{x}) = \int^{\infty}_0 \omega(|\boldsymbol{x}|,\eta) f(\eta) \eta^{n-1}\dd \eta \,  \frac{\boldsymbol{x}}{|\boldsymbol{x}|}.
    \]
Then we can see that
$(\nabla \Phi_\alpha*f)(\boldsymbol{x}) \cdot \frac{\boldsymbol{x}}{|\boldsymbol{x}|}
=(\Phi_\alpha*f)_r(|\boldsymbol{x}|)$
is a radial function.
Now, as in \cite{Carrillo_2022}, we see that, for $r > 0$,
\begin{align*}
(\Phi_\alpha*f)_r( r) = \int^{\infty}_0 \Big( \int_{\mathbb{S}^{n-1}} \frac{r\boldsymbol{e_1} - \eta \boldsymbol{y}}{|r\boldsymbol{e_1} - \eta \boldsymbol{y}|^{\alpha + 2}} \cdot \boldsymbol{e_1} \dd \sigma(\boldsymbol{y}) \Big) f(\eta) \eta^{n-1} \dd\eta.
\end{align*}
Now, considering the inner integral and applying a change of coordinates to
the polyspherical coordinates given by
    \[
    \boldsymbol{y} = (\cos\theta, \sin{\theta} \,\hat{\boldsymbol{z}})
    \]
with $\theta \in [0,\pi)$ and $\hat{\boldsymbol{z}} \in \mathbb{S}^{n-2}$, we have
\begin{align*}
\begin{split}
& \int_{\mathbb{S}^{n-1}} \frac{r\boldsymbol{e_1} - \eta \boldsymbol{y}}{|r\boldsymbol{e_1} - \eta \boldsymbol{y}|^{\alpha + 2}} \cdot\boldsymbol{e_1} \dd \sigma(\boldsymbol{y}) \\
& = \frac{\omega_{n-1}}{n-1} \int^\pi_0 \big((r - \eta \cos{\theta})^2 + (\eta \sin{\theta})^2 \big)^{-\frac{\alpha + 2}{2}}(r - \eta \cos{\theta})|\sin{\theta}|^{n-2}\dd \theta.
\end{split}
\end{align*}
Looking at the integrand, we have
\begin{align}
& \Big| \big((r - \eta \cos{\theta})^2 + (\eta \sin{\theta})^2 \big)^{-\frac{\alpha + 2}{2}}(r - \eta \cos{\theta})|\sin{\theta}|^{n-2} \Big| \nonumber\\
& \leq \big((r - \eta)^2 + 2r\eta(1-\cos{\theta}) \big)^{-\frac{\alpha + 1}{2}}|\sin{\theta}|^{n-2} \nonumber\\
& \leq \big((r - \eta)^2 + r\eta \theta^2(1 - \frac{\pi^2}{12}) \big)^{-\frac{\alpha + 1}{2}}\theta^{n-2}.\label{continuity.3}
\end{align}
For $\zeta > 0$, using \eqref{continuity.3},  we obtain that, for $\alpha > n-2$,
\begin{align}
& \Big| \int^\pi_0 \big((r - \eta \cos{\theta})^2 + (\eta \sin{\theta})^2 \big)^{-\frac{\alpha + 2}{2}}(r - \eta \cos{\theta})|\sin{\theta}|^{n-2}\dd \theta \Big| \nonumber\\
& \leq \int^\pi_0 \big((r - \eta)^2 + r\eta \theta^2(1 - \frac{\pi^2}{12}) \big)^{-\frac{\alpha + 1}{2}}\theta^{n-2}\dd \theta \nonumber\\
 & \leq \int^\pi_\zeta \big(r\eta \theta^2(1 - \frac{\pi^2}{12}) \big)^{-\frac{\alpha + 1}{2}}\theta^{n-2}\dd \theta + \int^\zeta_0 |r - \eta|^{-(\alpha + 1)}\theta^{n-2}\dd \theta\nonumber \\
 & =\big(r\eta (1 - \frac{\pi^2}{12}) \big)^{-\frac{\alpha + 1}{2}}\Big[ \frac{\theta^{n-2-\alpha}}{n-2-\alpha} \Big]^\pi_\zeta + |r - \eta|^{-(\alpha + 1)}\Big[\frac{\theta^{n-1}}{n-1} \Big]^\zeta_0. \label{continuity.4}
\end{align}
Letting $\zeta = \min\{\pi, \frac{|r-\eta|}{\sqrt{r\eta}}\}$,
then there exists $C>0$ such that
\begin{align}\label{continuity.5}
& \Big| \int^\pi_0 \big((r - \eta \cos{\theta})^2 + (\eta \sin{\theta})^2 \big)^{-\frac{\alpha + 2}{2}}
(r - \eta \cos{\theta})|\sin{\theta}|^{n-2}\dd \theta \Big| \nonumber\\
& \leq C(r\eta)^{-\frac{n-1}{2}} |r-\eta|^{n-2-\alpha}.
\end{align}

\medskip
\noindent
5. Let $r>0$. Then, for $r \gg \varepsilon > 0$, we use
\eqref{continuity.3} to obtain that, for $s \in (r-\frac{\varepsilon}{2},r+\frac{\varepsilon}{2})$
and $\eta \in (0,\infty)\setminus (r - \varepsilon, r + \varepsilon)$,
\begin{equation*}
\Big|\big((s - \eta \cos{\theta})^2 + (\eta \sin{\theta})^2 \big)^{-\frac{\alpha + 2}{2}}(s - \eta \cos{\theta})|\sin{\theta}|^{n-2} \Big|
\leq C(\varepsilon).
\end{equation*}
Hence, by the Lebesgue dominated convergence theorem, we have
\begin{align*}
\begin{split}
\lim_{s \to r} \int_{(0,\infty)\setminus(r - \varepsilon, r +\varepsilon)}
\omega(s,\eta) f(\eta)\,\eta^{n-1} \dd\eta
= \int_{(0,\infty)\setminus(r - \varepsilon,r + \varepsilon)} \omega(r,\eta) f(\eta)
\,\eta^{n-1} \dd\eta.
\end{split}
\end{align*}
By \eqref{continuity.5}, we have
        \begin{align*}
                \left |\int^{r + \varepsilon}_{r - \varepsilon} \omega(s,\eta) f(\eta) \eta^{n-1} \dd\eta \right| \leq C \int^{r + \varepsilon}_{r - \varepsilon} |s-\eta|^{n-2-\alpha} \dd\eta          \leq C \varepsilon^{n-1-\alpha}.
        \end{align*}
Since
$\varepsilon^{n-1-\alpha} \to 0$ as $\varepsilon \to 0$, and
     \begin{equation*}
         \lim_{\varepsilon \to 0} \int_{(0,\infty)\setminus(r - \varepsilon,r + \varepsilon)} \omega(r,\eta) f(\eta)\,\eta^{n-1} \dd\eta
         = \int^{\infty}_0 \omega(r,\eta) f(\eta)\,\eta^{n-1} \dd\eta,
     \end{equation*}
we conclude that $\lim_{s \to r} (\Phi_\alpha*f)_r(s) = (\Phi_\alpha*f)_r(r)$.

\smallskip
For $\alpha \in (-1,n-2)$, following the same approach as \eqref{continuity.4}, we have
    \begin{align*}
        \begin{split}
            |\omega(r,\eta)| \leq C \int^\pi_0 \big(r\eta \theta^2 \big)^{-\frac{\alpha + 1}{2}}\theta^{n-2}\dd \theta \leq C (r\eta)^{-\frac{\alpha + 1}{2}}.
        \end{split}
    \end{align*}
    Thus, it follows from the same approach as for $\alpha > n-2$ that $(\Phi_\alpha*f)_r$ is continuous for $\alpha \in (-1,n-2)$.

\smallskip
For $\alpha = n-2$,
since $\Phi_{n-2}$ is the fundamental solution of the Laplace equation up to a constant,
we see that, for any radial function $f \in C^\infty_{\rm c}(\mathbb{R}^n)$,
    \begin{equation*}
        \big(r^{n-1}(\Phi_{n-2} * f)_r\big)_r(r)
        = r^{n-1} \Delta (\Phi_{n-2} * f)(r)
        = \omega_n r^{n-1} f(r).
    \end{equation*}
Since it is classically differentiable, then $r^{n-1}(\Phi_{n-2} * f)_r$ and hence $(\Phi_{n-2} * f)_r$
are continuous. Furthermore, we integrate it over $r$ to obtain
    \[
        \int^{\infty}_0 \omega(r,\eta) f(\eta) \eta^{n-1}\dd \eta = (\Phi_{n-2} * f)_r(r)
        = \frac{\omega_n}{r^{n-1}} \int^r_0 f(\eta) \eta^{n-1} \dd \eta.
    \]
This implies
\[
\int^{\infty}_0 \Big( \omega(r,\eta) - \frac{\omega_n \mathds{1}_{(0,r)}(\eta)}{r^{n-1}} \Big) f(\eta) \eta^{n-1} \dd \eta = 0
\]
for all radial functions $f \in C^\infty_{\rm c}(\mathbb{R}^n)$.
Then, by the fundamental lemma of the calculus of variations, we conclude
    \begin{equation*}
        \omega(r,\eta) = \frac{\omega_n \mathds{1}_{(0,r)}(\eta)}{r^{n-1}}.
    \end{equation*}
\end{proof}

\begin{remark}
Notice that, for fixed $t>0$,
density $\rho(t,\boldsymbol{x})$ satisfies the assumption on $f$ in {\rm Lemma \ref{potentially}},
as $\rho(t,\boldsymbol{x})$ is smooth in $\Omega_t$ and $0$ everywhere else by convention,
$\rho(t,\cdot) \in L^q_{\rm loc}(\mathbb{R}^n \setminus \{\boldsymbol{0}\})$ for all $q \in [1,\infty]$,
and $\rho(t,\cdot)\in L^1(\mathbb{R}^n)$ since it has finite mass $M$.
Thus, we can apply {\rm Lemma \ref{potentially}} for $\rho(t,\boldsymbol{x})$ regarding
the solution of the free boundary problem.
\end{remark}

\begin{remark}
Alves-Grafakos-Tzavaras in {\rm \cite{alves_2024}} proved the uniform estimates
for a bilinear operator that relates to a divergence form of the Riesz potential.
That is, for sufficiently regular $f$ on $\mathbb{R}^n$, it is shown that
there exists $S_\alpha(f) = I_\alpha(f,f)$ such that
    \[
    f \nabla (\Phi_\alpha * f) = \nabla \cdot S_{n-\alpha}(f).
    \]
Moreover, for $1 < p,q < \frac{n}{\alpha}$ and $r \geq 1$ such that
    \[
    \frac{1}{p} + \frac{1}{q} =  \frac{1}{r} + \frac{\alpha}{n},
    \]
    then, for $f \in L^p(\mathbb{R}^n)$ and $g \in L^q(\mathbb{R}^n)$,
    \[
    \| I_\alpha(f,g) \|_{L^r(\mathbb{R}^n)} \leq C \| f \|_{L^p(\mathbb{R}^n)} \| g \|_{L^q(\mathbb{R}^n)}.
    \]
With this,
some compensated integrability can be obtained by using the HLS inequality
to estimate $f \nabla (\Phi_\alpha * f)$, so that
the weak solutions in the divergence form can be framed.
However, our approximate solutions $(\rho,u)$ of the free boundary problem are only defined and differentiable on $\Omega_t$, as such the construction of $S_\alpha$ no longer works in the context of our problem, where it can be seen that $S_\alpha(\rho)$ defined in {\rm \cite{alves_2024}} is not differentiable. Moreover, since we have no control over the derivatives of the gradient of $u$, unlike {\rm \cite{alves_2024}} where it is assumed in the definition of the weak solutions,
it is difficult to control the potential term in the same way.
\end{remark}

\smallskip
\subsection{Energy estimates}
Now we make some uniform energy estimates of the solutions of \eqref{Eul}--\eqref{0.20}.

\begin{lemma}[Energy Estimates]\label{Energy}
For $\alpha \in [0,n - 1)$, then any smooth solution $(\rho,u)(t,r)$ of the free boundary
problem \eqref{Eul}--\eqref{0.20} satisfies the following energy estimate{\rm :}
\begin{align}\label{firstest}
&\int^{b(t)}_a \rho \Big(\frac{1}{2}u^2 + e(\rho) + \frac{\kappa}{2} (\Phi_\alpha * \rho) \Big) r^{n-1} \dd r
    + \varepsilon \int^t_0\int^{b(s)}_a \rho \Big(u_r^2 + (n-1)\frac{u^2}{r^2} \Big) r^{n-1} \dd r\dd s
    \nonumber\\
    &\quad\, + \varepsilon (n-1) \int^t_0 (\rho u^2)(s,b(s))b(s)^{n-2} \dd s\nonumber\\
    &= \int^b_a \rho_0 \Big(\frac{1}{2}u_0^2 + e(\rho_0) + \frac{\kappa}{2} (\Phi_\alpha * \rho_0) \Big)r^{n-1} \dd r.
    \end{align}
In particular, we have the following estimates{\rm :}

\smallskip
\noindent{Case} {\rm 1:} $\alpha \in (0,n-1)$ and $\kappa = -1$ with $\gamma > 1$,
\begin{align*}
&\int^{b(t)}_a \rho \Big(\frac{1}{2}u^2 + e(\rho) - \frac{1}{2} (\Phi_\alpha * \rho) \Big)r^{n-1} \dd r
+ \varepsilon \int^t_0\int^{b(s)}_a \rho \Big(u_r^2 + (n-1)\frac{u^2}{r^2}\Big) r^{n-1} \dd r\dd s \\
&\quad + \varepsilon (n-1) \int^t_0 (\rho u^2)(s,b(s))b(s)^{n-2} \dd s\\
&= \int^b_a \rho_0 \Big(\frac{1}{2}u_0^2 + e(\rho_0) - \frac{1}{2} (\Phi_\alpha * \rho_0)\Big)r^{n-1} \dd r
= E^{\varepsilon,b}_0.
        \end{align*}

\noindent{Case} {\rm 2:} $\alpha \in (0,n-1)$ and $\kappa = 1$
with $\gamma \in (\frac{2n}{2n - \alpha}, \frac{n + \alpha}{n} ]$
and $M < M^{\varepsilon,b,\alpha}_{\rm c}(\gamma)$,
\begin{align}\label{eq2}
&\int^{b(t)}_a \rho \Big(\frac{1}{2}u^2 + C_\gamma e(\rho) \Big)r^{n-1} \dd r
 + \varepsilon \int^t_0\int^{b(s)}_a \rho \Big(u_r^2 + (n-1)\frac{u^2}{r^2} \Big) r^{n-1} \dd r\dd s\nonumber \\
&\quad + \varepsilon(n-1)\int^t_0 (\rho u^2)(s,b(s))b(s)^{n-2} \dd s \nonumber\\
&\leq \int^b_a \rho_0
\Big(\frac{1}{2}u_0^2 + e(\rho_0) + \frac{\kappa}{2} (\Phi_\alpha * \rho_0) \Big)r^{n-1} \dd r
= E^{\varepsilon,b}_0,
\end{align}
where the positive constant $C_\gamma > 0$ is defined as
        \[
        C_\gamma := \begin{cases}
            1 - B_{n,\gamma,\alpha}M^{\frac{n - \alpha}{n}} \quad &
             \mbox{for $\gamma = \frac{n + \alpha}{n}$},\\[1mm]
            \frac{\alpha - n(\gamma-1)}{\alpha} &
            \mbox{for $\gamma \in (\frac{2n}{2n - \alpha}, \frac{n + \alpha}{n} )$}.
        \end{cases}
        \]

    \noindent{Case} {\rm 3:} $\alpha \in (0,n-1)$ and $\kappa = 1$ with $\gamma > \frac{n + \alpha}{n}$,
    \begin{align*}
        & \frac{1}{2} \int^{b(t)}_a \rho \left (u^2 + e(\rho) \right )r^{n-1} \dd r + \varepsilon \int^t_0\int^{b(s)}_a \rho \Big(u_r^2 + (n-1)\frac{u^2}{r^2} \Big) r^{n-1} \dd r\dd s \\
        &\,\, + \varepsilon (n-1)\int^t_0 (\rho u^2)(s,b(s))b(s)^{n-2} \dd s \leq C(M,E^{\varepsilon,b}_0).
    \end{align*}

    \noindent{Case} {\rm 4:} $\alpha \in (-1,0]$, $\kappa \in \{-1,1\}$ with $\gamma > 1$,
    \begin{align}\label{eq4}
        & \frac{1}{2} \int^{b(t)}_a \rho \left (u^2 + e(\rho) + k_\alpha(1 + r^2) \right )r^{n-1} \dd r + \varepsilon \int^t_0\int^{b(s)}_a \rho \Big(u_r^2 + \frac{u^2}{r^2} \Big) r^{n-1} \dd r \dd s\nonumber \\
        & \,\,+ \varepsilon (n-1) \int^t_0 (\rho u^2)(s,b(s)) b(s)^{n-2} \dd s \leq C(M,E^{\varepsilon,b}_0,T).
    \end{align}
\end{lemma}

\begin{proof}  We divide the proof into six steps.

\medskip
1. Multiplying $\eqref{lag}_2$ in Lagrangian coordinates by $u$ and integrating over $x$ yield
\begin{align}\label{1.1}
& \frac{1}{2}\frac{\d}{\d \tau}\int^{\frac{M}{\omega_n}}_0 u^2 \dd x
+ \int^{\frac{M}{\omega_n}}_0 \big(p-\varepsilon \rho^2(r^{n-1}u)_x\big)_xur^{n-1} \dd x
   \nonumber\\
& = -\varepsilon(n-1)\int^{\frac{M}{\omega_n}}_0 \rho_xu^2r^{n-2} \dd x - \kappa \int^{\frac{M}{\omega_n}}_0 \rho u r^{n-1} (\Phi_\alpha * \rho)_x \dd x.
\end{align}
For the last term on the left-hand side (LHS) of \eqref{1.1},
using the change of variables \eqref{deriv}, the boundary conditions \eqref{bound},
the conservation of mass $\eqref{lag}_1$, and integration by parts, we have
\begin{align}\label{1.2}
	&\int_0^{\frac{M}{\omega_n}}\big(p(\rho)-\varepsilon\rho^2(r^{n-1}u)_x\big)_xu\,r^{n-1} \dd x\nonumber\\
	&=-\int_0^{\frac{M}{\omega_n}}\big(p(\rho)-\varepsilon\rho^2(r^{n-1}u)_x\big)(r^{n-1}u)_x \dd x \nonumber\\
	&=-\int_0^{\frac{M}{\omega_n}}p(\rho)(r^{n-1}u)_x \dd x
	+\varepsilon\int_0^{\frac{M}{\omega_n}}\rho^2\big((r^{n-1}u)_x\big)^2\dd x \nonumber\\
    &=a_0\int_0^{\frac{M}{\omega_n}}\rho^{\gamma-2}\rho_\tau\dd x
	+\varepsilon\int_0^{\frac{M}{\omega_n}}\rho^2\big(r^{n-1}u_x+(n-1)r^{n-2}r_xu\big)^2\dd x\nonumber\\
	&=\frac{\d}{\d \tau}\int_0^{\frac{M}{\omega_n}} e(\rho)\dd x+\varepsilon\int_0^{\frac{M}{\omega_n}}\Big(\rho^2(r^{n-1}u_x)^2
	+(n-1)^2\frac{u^2}{r^2}+2(n-1)r^{n-2}\rho u u_x\Big)\dd x.
\end{align}
For the first term on the right-hand side (RHS) of \eqref{1.1}, integrating by parts yields
\begin{align}\label{1.3}
	&(n-1)\varepsilon\int_0^{\frac{M}{\omega_n}}\rho_xu^2r^{n-2} \dd x \nonumber\\
	&=-(n-1)\varepsilon\int_0^{\frac{M}{\omega_n}}\big(2 r^{n-2}\rho uu_x +(n-2)\frac{ u^2}{r^2}\big)\dd x + \varepsilon (n-1)(\rho u^2 r^{n-2})(\tau,\frac{M}{\omega_n}).
\end{align}
For the last term on the RHS of \eqref{1.1},
converting back to the Eulerian coordinates, and applying \eqref{bound} and the conservation of mass $\eqref{Eul}_1$, we have
\begin{align}\label{1.4}
        & \kappa \int^{\frac{M}{\omega_n}}_0  \rho u\, (\Phi_\alpha * \rho)_x\, r^{n-1} \dd x\nonumber \\
        & = \kappa \int^{b(t)}_a  \rho u\, (\Phi_\alpha * \rho)_r\, r^{n-1} \dd r \nonumber\\
        & = \kappa \big (\rho u r^{n-1} (\Phi_\alpha * \rho) \big)(t,b(t))
        - \kappa \int^{b(t)}_a (\rho u r^{n-1})_r (\Phi_\alpha * \rho) \dd r\nonumber \\
        & = \kappa \big (\rho u r^{n-1} (\Phi_\alpha * \rho) \big)(t,b(t)) + \kappa \int^{b(t)}_a \rho_t (\Phi_\alpha * \rho)\,r^{n-1}\dd r\nonumber \\
        & = \frac{\kappa}{2} \big (\rho u r^{n-1} (\Phi_\alpha * \rho) \big)(t,b(t))
        + \frac{\kappa}{2} \int^{b(t)}_a \frac{\partial}{\partial t} \big(\rho (\Phi_\alpha * \rho) \big)\,r^{n-1}\dd r\nonumber \\
        & = \frac{\kappa}{2} \frac{\d}{\d t} \int^{b(t)}_a\rho (\Phi_\alpha * \rho)\, r^{n-1}\dd r\nonumber \\
        & = \frac{\kappa}{2} \frac{\d}{\d \tau} \int^{\frac{M}{\omega_n}}_0\Phi_\alpha * \rho \dd x,
\end{align}
where the 5th line follows from
$K(r,\eta) = K(\eta,r)$, the Fubini theorem, and
\begin{align*}
    \begin{split}
        & \int^{b(t)}_a \rho (\Phi_\alpha * \rho)_t\,r^{n-1}\dd r \\
        & = \int^{b(t)}_a \rho \Big( \int^{b(t)}_a K(r,\eta)\rho(t,\eta)\eta^{n-1}\dd \eta \Big)_t\,r^{n-1}\dd r \\
        & = (\rho u)(t,b(t))b(t)^{n-1} \int^{b(t)}_a K(b(t),r) \rho(t,r)\, r^{n-1}\dd r \\
        &\quad\, + \int^{b(t)}_a \rho \Big( \int^{b(t)}_a K(r,\eta)\rho_t(t,\eta)\eta^{n-1}\dd \eta \Big)\, r^{n-1}\dd r \\
        & = \big (\rho u r^{n-1} (\Phi_\alpha * \rho) \big)(t,b(t))
        + \int^{b(t)}_a \rho_t \Big( \int^{b(t)}_a K(\eta,r)\rho(t,r)r^{n-1}\dd r \Big)\,\eta^{n-1}\dd \eta \\
        & = \big (\rho u r^{n-1} (\Phi_\alpha * \rho) \big)(t,b(t)) + \int^{b(t)}_a \rho_t (\Phi_\alpha * \rho)\,r^{n-1}\dd r.
    \end{split}
\end{align*}
Plugging \eqref{1.2}--\eqref{1.4} into \eqref{1.1}, we obtain
\begin{align}\label{1.5}
 & \frac{\dd}{\dd \tau} \bigg\{ \int^{\frac{M}{\omega_n}}_0
   \Big(\frac{1}{2}u^2 + e(\rho) + \frac{\kappa}{2}(\Phi_\alpha * \rho)\Big) \dd x \bigg\}
   +\varepsilon\int_0^{\frac{M}{\omega_n}}\big(\rho^2(r^{n-1}u_x)^2 + (n-1)\frac{u^2}{r^2}\big)\dd x \nonumber\\
	&\,\,+ \varepsilon(n-1)(\rho u^2 r^{n-2})(\tau,\frac{M}{\omega_n})=0.
\end{align}
Converting \eqref{1.5} into the Eulerian coordinates, we have
\begin{align*}
    \begin{split}
        &\frac{\d}{\d t} \int^{b(t)}_a \rho \Big(\frac{1}{2}u^2 + e(\rho) + \frac{\kappa}{2} (\Phi_\alpha * \rho)\Big)\,r^{n-1} \dd r
        + \varepsilon \int^{b(t)}_a \rho \Big(u_r^2 + (n-1)\frac{u^2}{r^2} \Big)\,r^{n-1} \dd r \\[2mm]
        &\,\, + \varepsilon(n-1) (\rho u^2)(t,b(t))b(t)^{n-2} = 0.
    \end{split}
\end{align*}
Integrating over $[0,t]$, we derive the first part of this lemma:
\begin{align}\label{1.7}
        &\int^{b(t)}_a \rho \Big(\frac{1}{2}u^2 + e(\rho) + \frac{\kappa}{2} (\Phi_\alpha * \rho)\Big)\,r^{n-1} \dd r
        + \varepsilon \int^t_0\int^{b(s)}_a \rho \Big(u_r^2 + (n-1)\frac{u^2}{r^2} \Big)\,r^{n-1} \dd r\dd s \nonumber\\[1mm]
        &\quad+ (n-1)\varepsilon\int^t_0 (\rho u^2)(s,b(s))b(s)^{n-2} \dd s \nonumber\\[1mm]
        &= \int^b_a \rho_0 \Big(\frac{1}{2}u_0^2 + e(\rho_0) + \frac{\kappa}{2} (\Phi_\alpha * \rho_0) \Big)\,r^{n-1} \dd r.
\end{align}
Then Case 1 follows directly from \eqref{1.7}.

\smallskip
2. For Case 2 \& 3 ({\it i.e.}, $\kappa = 1$),
we need to control the potential energy term by the internal energy term.
This is done by applying the H\"older inequality to the potential term and then the variation of the HLS inequality from Lemma \ref{HLSineq},
and finally
by using the interpolation as follows:
\begin{align}\label{1.8}
        & \frac{1}{2}\int^{b(t)}_a \rho (\Phi_\alpha * \rho) r^{n-1} \dd r \nonumber\\
        & = \frac{1}{2 \omega_n}\int_{\Omega_t} \big(\rho(\Phi_\alpha * \rho)\big)(t,\boldsymbol{y}) \dd \boldsymbol{y} \nonumber\\
        & \leq \frac{1}{2 \alpha \omega_n} \bigg| \int_{\Omega_t} \big(\rho(|\cdot|^{-\alpha} * \rho)\big)(t,\boldsymbol{y}) \dd \boldsymbol{y} \bigg|\nonumber\\
        & \leq \frac{C_*(\alpha,\gamma,n)}{2 \alpha \omega_n} \|\rho\|^{\frac{\gamma(2n - \alpha) - 2n}{n(\gamma -1)}}_{L^1(\Omega_t)}\|\rho\|^{\frac{\alpha\gamma}{n(\gamma - 1)}}_{L^\gamma(\Omega_t)} \nonumber\\
        & = \frac{C_*(\alpha,\gamma,n)}{2 \alpha}\omega_n^{\frac{\alpha}{n(\gamma - 1)} - 1}M^{\frac{\gamma(2n - \alpha) - 2n}{n(\gamma -1)}}
           \Big( \int^{b(t)}_a \rho^\gamma r^{n-1} \dd r \Big)^\frac{\alpha}{n(\gamma - 1)}\nonumber\\
        & = \frac{C_*(\alpha,\gamma,n)}{2 \alpha}\omega_n^{\frac{\alpha}{n(\gamma - 1)} - 1}
          \Big(\frac{\gamma - 1}{a_0} \Big)^{\frac{\alpha}{n(\gamma - 1)}}
           M^{\frac{\gamma(2n - \alpha) - 2n}{n(\gamma -1)}}
            \Big( \int^{b(t)}_a \rho e(\rho) r^{n-1} \dd r \Big)^\frac{\alpha}{n(\gamma - 1)} \nonumber\\
        & = B_{n,\gamma,\alpha} M^{\frac{\gamma(2n - \alpha) - 2n}{n(\gamma -1)}}
          \Big( \int^{b(t)}_a \rho e(\rho) r^{n-1} \dd r \Big)^\frac{\alpha}{n(\gamma - 1)}.
\end{align}
Thus, applying \eqref{1.8}, we obtain
\begin{align}\label{1.9}
        & \int^{b(t)}_a \rho \Big(e(\rho) + \frac{1}{2} (\Phi_\alpha * \rho)\Big)r^{n-1} \dd r\nonumber \\
        & \geq \int^{b(t)}_a \rho e(\rho)r^{n-1} \dd r - B_{n,\gamma,\alpha} M^{\frac{\gamma(2n - \alpha) - 2n}{n(\gamma -1)}}
           \Big( \int^{b(t)}_a \rho e(\rho) r^{n-1} \dd r \Big)^\frac{\alpha}{n(\gamma - 1)}.
\end{align}

\smallskip
3. For Case 3, $\gamma > \frac{n + \alpha}{n}$ that implies that $\frac{\alpha}{n(\gamma - 1)} < 1$.
Thus, using \eqref{1.9} and the H\"older inequality, we obtain
\begin{align*}
    \begin{split}
        & \int^{b(t)}_a \rho \Big( e(\rho) + \frac{1}{2} (\Phi_\alpha * \rho) \Big)r^{n-1} \dd r
        \geq \frac{1}{2}\int^{b(t)}_a \rho e(\rho)r^{n-1} \dd r - C(M).
    \end{split}
\end{align*}

\medskip
4. Now, for Case 2, we first consider the case that $\gamma = \frac{n + \alpha}{n}$. Then, from \eqref{1.9}, we have
\begin{align*}
    \begin{split}
        & \int^{b(t)}_a \rho \Big( e(\rho) + \frac{1}{2} (\Phi_\alpha * \rho)\Big)r^{n-1} \dd r \\
        & \geq \int^{b(t)}_a \rho e(\rho)r^{n-1} \dd r - B_{n,\gamma,\alpha} M^{\frac{n - \alpha}{n}}\int^{b(t)}_a \rho e(\rho) r^{n-1} \dd r \\
        & = C_\gamma \int^{b(t)}_a \rho e(\rho)r^{n-1} \dd r,
    \end{split}
\end{align*}
where $C_\gamma > 0$, as we have assumed that $M < M^{\varepsilon,b,\alpha}_{\rm c}(\gamma)$.
Now, considering the remaining case $\gamma \in (\frac{2n}{2n - \alpha},\frac{n + \alpha}{n})$,
inspired by \eqref{1.9}, we define
\begin{equation*}
    F(s) = s - B_{n,\gamma,\alpha} M^{\frac{\gamma(2n - \alpha) - 2n}{n(\gamma -1)}}s^\frac{\alpha}{n(\gamma - 1)}.
\end{equation*}
Then
\begin{align*}
    \begin{cases}
        \displaystyle
        F'(s) = 1 - \frac{\alpha}{n(\gamma - 1)}B_{n,\gamma,\alpha}
          M^{\frac{\gamma(2n - \alpha) - 2n}{n(\gamma -1)}}s^\frac{\alpha - n(\gamma - 1)}{n(\gamma - 1)}, \\[3mm]
        \displaystyle
        F''(s) = - \frac{\alpha(\alpha - n(\gamma -1))}{n^2(\gamma - 1)^2}
          B_{n,\gamma,\alpha}M^{\frac{\gamma(2n - \alpha) - 2n}{n(\gamma -1)}}s^\frac{\alpha - 2n(\gamma - 1)}{n(\gamma - 1)}.
    \end{cases}
\end{align*}
which, for $\gamma \in (\frac{2n}{2n - \alpha},\frac{n + \alpha}{n})$,
implies $F''(s) < 0$ for any $s > 0$.
This implies that $F(s)$ is concave for $s>0$. Denote
\begin{equation*}
    s_* = \Big( \frac{\alpha B_{n,\gamma,\alpha}}{n(\gamma - 1)} \Big)^\frac{n(\gamma - 1)}{n(\gamma - 1) - \alpha}
    M^\frac{\gamma(2n - \alpha) - 2n}{n(\gamma - 1) - \alpha},
\end{equation*}
the critical point of $F(s)$ ({\it i.e.}, $F'(s_*) = 0$).
Since $F'(s) < 0$ for $s > 0$, we see that the maximum of $F(s)$ for $s \geq 0$ is given by
\begin{equation*}
    F(s_*) = \frac{\alpha - n(\gamma - 1)}{\alpha} s_* > 0.
\end{equation*}
Thus, using $M < M^{\varepsilon,b,\alpha}_{\rm c}(\gamma)$, \eqref{critical},
and $\frac{\gamma(2n - \alpha) - 2n}{n(\gamma - 1) - \alpha} < 0$
for $\gamma \in (\frac{2n}{2n - \alpha},\frac{n + \alpha}{n})$, we have
\begin{align}\label{1.10}
        F(s_*) > \frac{\alpha - n(\gamma - 1)}{\alpha}
        \Big(\frac{\alpha B_{n,\gamma,\alpha}}{n(\gamma - 1)} \Big)^\frac{n(\gamma - 1)}{n(\gamma - 1) - \alpha}M^{\varepsilon,b,\alpha}_{\rm c}(\gamma)^\frac{\gamma(2n - \alpha) - 2n}{n(\gamma - 1) - \alpha} = E^{\varepsilon,b}_0,
\end{align}
and
\begin{align}\label{1.11}
        s_* = \frac{\alpha}{\alpha - n(\gamma - 1)}F(s_*) > \frac{\alpha}{\alpha - n(\gamma - 1)}E^{\varepsilon,b}_0 > 2E^{\varepsilon,b}_0.
\end{align}
Now, using \eqref{1.7} and \eqref{1.9}--\eqref{1.11}, we obtain
\begin{align}\label{1.12}
&F( \int^{b(t)}_a \rho e(\rho) r^{n-1} \dd r) \leq E^{\varepsilon,b}_0 < F(s_*),\\
&\int^{b}_a \rho_0 e(\rho_0) r^{n-1} \dd r \leq E^{\varepsilon,b}_0 < s_*.\label{1.12a}
\end{align}
Thus, since $t \mapsto \int^{b(t)}_a \rho e(\rho) r^{n-1} \dd r$ is continuous with respect to $t$,
we must have
\begin{align*}
        \int^{b(t)}_a \rho e(\rho) r^{n-1} \dd r < s_*,
\end{align*}
as otherwise we can apply the intermediate value theorem to derive a contradiction to \eqref{1.12}.
This implies
\begin{align}\label{1.15}
        & F(\int^{b(t)}_a \rho e(\rho) r^{n-1} \dd r) \nonumber\\
       & \geq \Big(1 - B_{n,\gamma,\alpha} M^{\frac{\gamma(2n - \alpha) - 2n}{n(\gamma -1)}}s_*^\frac{\alpha - n(\gamma - 1)}{n(\gamma - 1)}\Big)
         \int^{b(t)}_a \rho e(\rho) r^{n-1} \dd r \nonumber\\
        & = \frac{\alpha - n(\gamma-1)}{\alpha}\int^{b(t)}_a \rho e(\rho) r^{n-1} \dd r.
\end{align}
Therefore, using \eqref{1.7}, \eqref{1.9}, and \eqref{1.15}, we derive \eqref{eq2}.

\medskip
5. For Case 4, we first consider $\alpha =0$.
For $\boldsymbol{x},\boldsymbol{y} \in \mathbb{R}^n$ with $|\boldsymbol{x} - \boldsymbol{y}| \geq 1$,
we know that $\log(|\boldsymbol{x}| + |\boldsymbol{y}|) \geq \log(|\boldsymbol{x} - \boldsymbol{y}|) \geq 0$.
Then, if $|\boldsymbol{y}| \geq |\boldsymbol{x}|$,
    \[
    \log(|\boldsymbol{x}| + |\boldsymbol{y}|) = \frac{1}{2} \log ((|\boldsymbol{x}| + |\boldsymbol{y}|)^2) \leq \frac{1}{2} \log (4 + 4 |\boldsymbol{y}|^2) = \log 2 + \frac{1}{2} \log (1 +  |\boldsymbol{y}|^2).
    \]
If $|\boldsymbol{x}| \geq |\boldsymbol{y}|$, then
    \[
    \log(|\boldsymbol{x}| + |\boldsymbol{y}|) \leq \log 2 + \frac{1}{2} \log (1 +  |\boldsymbol{x}|^2).
    \]
Thus, we obtain that, for $\boldsymbol{x}, \boldsymbol{y} \in \mathbb{R}^n$ such that $|\boldsymbol{x} - \boldsymbol{y}|\geq 1$,
\begin{equation}\label{logtingaling}
\log(|\boldsymbol{x}| + |\boldsymbol{y}|)
\leq \log 4 + \frac{1}{2} \log (1 +  |\boldsymbol{x}|^2) + \frac{1}{2} \log (1 +  |\boldsymbol{y}|^2).
\end{equation}
Then, by \eqref{logtingaling} and the H\"older inequality, we have
\begin{align}\label{NewApproach}
& \bigg| \int^{b(t)}_a \rho (\Phi_0 * \rho) r^{n-1} \dd r \bigg|\nonumber \\
& = \frac{1}{\omega_n} \bigg|\int_{\mathbb{R}^n} \int_{\mathbb{R}^n} \rho(\boldsymbol{x}) \log (|\boldsymbol{x} - \boldsymbol{y}|) \rho(\boldsymbol{y}) \dd \boldsymbol{x} \dd \boldsymbol{y} \bigg|\nonumber \\
& \leq \frac{1}{\omega_n} \int_{\mathbb{R}^n} \rho(\boldsymbol{y})
  \Big( \int_{B_1(\boldsymbol{y})} \rho(\boldsymbol{x}) \log (|\boldsymbol{x} - \boldsymbol{y}|^{-1}) \dd \boldsymbol{x} + \int_{B_1(\boldsymbol{y})^c} \rho(\boldsymbol{x}) \log (|\boldsymbol{x} - \boldsymbol{y}|) \dd \boldsymbol{x} \Big) \dd \boldsymbol{y} \nonumber \\
& \leq \frac{1}{\omega_n} \int_{\mathbb{R}^n} \rho(\boldsymbol{y})
\bigg ( \| \rho \|_{L^\gamma(B_1(\boldsymbol{y}))} \|\log (|\cdot|) \|_{L^\frac{\gamma}{\gamma - 1}(B_1)} \nonumber \\
& \qquad \qquad \qquad\quad\,\, + \int_{|\boldsymbol{x} - \boldsymbol{y}| \geq 1} \rho(\boldsymbol{x})
 \Big(\log 4 + \frac{1}{2} \log (1 + |\boldsymbol{x}|^2) + \frac{1}{2} \log (1 +  |\boldsymbol{y}|^2) \Big) \dd \boldsymbol{x} \bigg)\dd \boldsymbol{y}\nonumber \\
& \leq C(M) \Big (1 + \| \rho \|_{L^\gamma(\Omega_t)} + \int_{\Omega_t} \rho(\boldsymbol{x}) \log(1+|\boldsymbol{x}|^2) \dd \boldsymbol{x}\Big) \nonumber \\
& \leq \frac{1}{2} \int^{b(t)}_a \rho e(\rho) r^{n-1} \dd r + C(M) \Big(1 + \int_a^{b(t)} \rho(r) \log(1+r^2) r^{n-1} \dd r \Big),
\end{align}
where we have applied the Young inequality in the final line.

\smallskip
6. For $\alpha \in (-1,0)$, we see that, for $\boldsymbol{x},\boldsymbol{y} \in \mathbb{R}^n$,
\begin{equation}\label{minus}
 |\Phi_\alpha(\boldsymbol{x} - \boldsymbol{y})|
 \leq \frac{2^{-\alpha}}{-\alpha} \big(k_\alpha(1 + |\boldsymbol{x}|^2) + k_\alpha(1 + |\boldsymbol{y}|^2) \big).
\end{equation}
Thus, employing a similar method to \eqref{NewApproach}, we use \eqref{minus} to obtain
\begin{equation}\label{moments}
        \left | \int^{b(t)}_a \rho (\Phi_\alpha * \rho) r^{n-1} \dd r \right | \leq \frac{2^{1 - \alpha} M}{- \alpha} \int^{b(t)}_a \rho k_\alpha(1 + r^2) r^{n-1} \dd r.
\end{equation}
Considering the latter integral in \eqref{NewApproach} and \eqref{moments},
but in generality for $\alpha \in (-1,0]$, from \eqref{Eul} and integration by parts, we obtain
\begin{align}\label{logtime}
        & \frac{\d}{\d t} \int^{b(t)}_a \rho k_\alpha(1 + r^2) r^{n-1} \dd r\nonumber\\
        & = \left (\rho u k_\alpha(1 + r^2) r \right )(t,b(t)) + \int^{b(t)}_a \rho_t k_\alpha(1 + r^2) r^{n-1} \dd r\nonumber\\
        & = \left (\rho u k_\alpha(1 + r^2) r \right )(t,b(t))
        - \int^{b(t)}_a k_\alpha(1 + r^2) ( \rho u r^{n-1})_r \dd r\nonumber\\
        & = \int^{b(t)}_a (k_\alpha(1 + r^2))_r  \rho u r^{n-1} \dd r\nonumber\\
        & \leq \int^{b(t)}_a \rho u r^{n-1} \dd r \nonumber\\
        & \leq \frac{1}{2} \int^{b(t)}_a  \rho(u^2 + 1) r^{n-1} \dd r\nonumber\\
        & = \frac{1}{2} \int^{b(t)}_a  \rho u^2 r^{n-1} \dd r + \frac{M}{2 \omega_n}.
\end{align}
Integrating \eqref{logtime} in $t$, we see that, for $\alpha \in (-1,0]$,
\begin{equation}\label{logtimeint}
    \int^{b(t)}_a \rho k_\alpha(1 + r^2) r^{n-1} \dd r \leq \int^{b}_a \rho_0 k_\alpha(1 + r^2) r^{n-1} \dd r + \frac{Mt}{2\omega_n} + \frac{1}{2} \int^t_0 \int^{b(s)}_a \rho u^2 r^{n-1} \dd r \dd s.
\end{equation}
Hence, combining \eqref{NewApproach}, \eqref{moments}, and \eqref{logtimeint}, we obtain
\begin{align}\label{toit}
&\bigg| \int^{b(t)}_a \rho (\Phi_\alpha * \rho) r^{n-1} \dd r \bigg|
 + \int^{b(t)}_a \rho k_\alpha(1 + r^2) r^{n-1} \dd r \nonumber\\
& \leq \frac{1}{2} \int^{b(t)}_a \rho e(\rho) r^{n-1} \dd r
 + C(M,T) \Big(1 + \int^{b}_a \rho_0 k_\alpha(1 + r^2) r^{n-1} \dd r + \int^t_0 \int^{b(s)}_a \rho u^2 r^{n-1} \dd r \dd s \Big).
\end{align}
Combining \eqref{1.7} and \eqref{toit}, we obtain that, for $\alpha \in (-1,0]$,
\begin{align}\label{pregron}
        & \frac{1}{2} \int^{b(t)}_a \rho \big(u^2 + e(\rho) + k_\alpha(1 + r^2) \big)r^{n-1} \dd r
        + \varepsilon \int^t_0\int^{b(s)}_a \rho \big(u_r^2 + \frac{u^2}{r^2} \big) r^{n-1} \dd r \dd s\nonumber\\
        &\quad\, + \varepsilon \int^t_0 (\rho u^2)(s,b(s)) b(s)^{n-2} \dd s\nonumber\\
        &\leq \int^b_a \rho_0 \Big(\frac{1}{2}u_0^2 + e(\rho_0) + \frac{\kappa}{2} (\Phi_\alpha * \rho_0) \Big)r^{n-1} \dd r\nonumber\\
        &\quad\, + C(M,T) \bigg(1 + \int^{b}_a \rho_0 k_\alpha(1 + r^2) r^{n-1} \dd r + \int^t_0 \int^{b(s)}_a \rho u^2 r^{n-1} \dd r \dd s\bigg).
\end{align}
Applying the Gr\"{o}nwall inequality on \eqref{pregron} and then using \eqref{epinitial2}, we conclude
\begin{align*}
 &\frac{1}{2} \int^{b(t)}_a \rho \big(u^2 + e(\rho) + k_\alpha(1 + r^2) \big)r^{n-1} \dd r
 + \varepsilon \int^t_0\int^{b(s)}_a \rho \big(u_r^2 + \frac{u^2}{r^2} \big) r^{n-1} \dd r \dd s \\
       &\quad\, + \varepsilon \int^t_0 (\rho u^2)(s,b(s)) b(s)^{n-2} \dd s \\
       & \leq C(M,T) \bigg( 1 + \int^b_a \rho_0 \Big(\frac{1}{2}u_0^2 + e(\rho_0) + \frac{\kappa}{2} (\Phi_\alpha *\rho_0) + k_\alpha(1 + r^2) \Big)r^{n-1} \dd r \bigg) \\
    &  \leq C(M,E^{\varepsilon,b}_0,T).
\end{align*}
\end{proof}

\begin{remark}\label{riesztolog}
 Adding $\frac{\kappa M^2}{2\alpha} = \frac{\kappa}{2} \int^{b(t)}_a \rho (\frac{1}{\alpha} * \rho)r^{n-1} \dd r
 = \frac{\kappa}{2} \int^{b}_a \rho_0 (\frac{1}{\alpha} * \rho_0)r^{n-1} \dd r$ to both sides of \eqref{firstest}
 yields
 \begin{align}\label{exper}
 &\int^{b(t)}_a \rho \Big(\frac{1}{2}u^2 + e(\rho) + \frac{\kappa}{2}
 \big((\frac{1}{\alpha} + \Phi_\alpha \big) * \rho \big) \Big)r^{n-1} \dd r
 + \varepsilon \int^t_0\int^{b(s)}_a \rho \Big(u_r^2 + (n-1)\frac{u^2}{r^2} \Big) r^{n-1} \dd r\dd s \nonumber\\
 &\quad + (n-1)\varepsilon\int^t_0 (\rho u^2)(s,b(s))b(s)^{n-2} \dd s \nonumber\\
 &= \int^b_a \rho_0 \Big(\frac{1}{2}u_0^2 + e(\rho_0) + \frac{\kappa}{2}
  \big((\frac{1}{\alpha} + \Phi_\alpha) * \rho_0 \big) \Big)r^{n-1} \dd r.
    \end{align}
We can formally see that, using \eqref{kernelcon}, as $\alpha \to 0$,
\eqref{exper} converges to
    \begin{align*}
        \begin{split}
            &\int^{b(t)}_a \rho \Big(\frac{1}{2}u^2 + e(\rho) + \frac{\kappa}{2} ( \Phi_0 * \rho) \Big)r^{n-1} \dd r
            + \varepsilon \int^t_0\int^{b(s)}_a \rho \Big(u_r^2 + (n-1)\frac{u^2}{r^2} \Big) r^{n-1} \dd r\dd s \\
            &\quad\, + (n-1)\varepsilon\int^t_0 (\rho u^2)(s,b(s))b(s)^{n-2} \dd s \\
            &= \int^b_a \rho_0 \Big(\frac{1}{2}u_0^2 + e(\rho_0) + \frac{\kappa}{2} (\Phi_0 * \rho_0) \Big)r^{n-1} \dd r.
        \end{split}
    \end{align*}
\end{remark}

\begin{corollary}[Uniform Estimates]\label{uniform}
Under the conditions of {\rm Lemma \ref{Energy}}, the following uniform estimates in $(\varepsilon,b)$
{\rm (}for $\varepsilon$ small enough and $b$ large enough{\rm )} hold{\rm :}

\smallskip
\noindent{Case} {\rm 1:} $\kappa = -1$ with $\gamma > 1$ for $\alpha \in (0,n-2]$ or $\gamma > \frac{2n}{2n - \alpha}$ for $\alpha \in (n-2,n-1)$,
\begin{align}\label{ueq1}
\int_{\mathbb{R}^n} \rho \big(1 + u^2 + \rho^{\gamma - 1} - \Phi_\alpha * \rho\big)
\dd \boldsymbol{x} \leq C(M,E_0);
\end{align}

\noindent{Case} {\rm 2:} $\alpha \in (0,n-1)$ and $\kappa = 1$ with $\gamma > \frac{2n}{2n - \alpha}$,
\begin{align}\label{ueq2}
\int_{\mathbb{R}^n} \rho \big(1 + u^2 + \rho^{\gamma-1} \big) \dd \boldsymbol{x}
   \leq C(M,E_0);
\end{align}

\noindent{Case} {\rm 3:} $\alpha \in (-1,0]$ and $\kappa \in\{-1,1\}$ with $\gamma > 1$,
\begin{align}\label{ueq3}
\int_{\mathbb{R}^n} \rho \big(1 + u^2 + \rho^{\gamma - 1} + k_\alpha(1 + |\boldsymbol{x}|^2) \big)
\dd \boldsymbol{x} \leq C(M,E_0,T),
    \end{align}
where we have used the convention that $\rho$ and $u$ are zero outside of $\Omega_t$.
In particular, the following uniform estimates in $(\varepsilon,b)$
{\rm (}for $\varepsilon$ small enough and $b$ large enough{\rm )} hold{\rm :} For {\rm Cases 1--2},
\begin{align*}
\|(\Phi_\alpha*\rho)(t)\|_{L^{\frac{2n}{\alpha}}(\mathbb{R}^n)}
+\|(\nabla \Phi_\alpha*\rho)(t)\|_{L^\frac{2n}{\alpha + 2}(\mathbb{R}^n)} \leq C(M,E_0);
    \end{align*}
and for {\rm Case 3}, for any $K \Subset \mathbb{R}^n$, there exists $C(M,E_0,T,K)>0$ such that
\begin{align*}
\|(\Phi_\alpha*\rho)(t)\|_{L^\infty(K)}
+\|(\nabla \Phi_\alpha*\rho)(t)\|_{L^\frac{2n}{\alpha + 2}(K)} \leq C(M,E_0,T,K).
\end{align*}
\end{corollary}

\begin{proof}
Estimates \eqref{ueq1}--\eqref{ueq3} follow from Lemma \ref{Energy}
and the fact that $\int^{b(t)}_a \rho r^{n-1} \dd r =  \frac{M}{\omega_n}$, $e(\rho) = \frac{a_0}{\gamma - 1} \rho^{\gamma - 1}$,
and
$E_0^{\v,b} \to E_0$ as $b \to \infty$ and $\varepsilon \to 0$ due to Lemma \ref{appendix5}.

For the second part of the corollary, for Case 2 or $\kappa = -1$ and $\gamma > \frac{2n}{2n - \alpha}$, by the HLS inequality
in Lemma \ref{simplehls} and interpolation, we have
\begin{align*}
&\|(\Phi_\alpha*\rho)(t)\|_{L^{\frac{2n}{\alpha}}(\mathbb{R}^n)}
  \leq \frac{C_{n,\alpha}}{\alpha} \|\rho(t)\|_{L^1(\mathbb{R}^n)}^\frac{\gamma (2n - \alpha) - 2n}{n(\gamma - 1)}
  \|\rho(t)\|_{L^\gamma(\mathbb{R}^n)}^\frac{\alpha \gamma}{n(\gamma - 1)} \leq C(M,E_0), \\
&\|(\nabla \Phi_\alpha*\rho)(t)\|_{L^{\frac{2n}{\alpha + 2}}(\mathbb{R}^n)}
\leq \frac{C_{n,\alpha + 1}}{\alpha} \|\rho(t)\|_{L^1(\mathbb{R}^n)}^\frac{\gamma (2n - \alpha) - 2n}{n(\gamma - 1)}
\|\rho(t)\|_{L^\gamma(\mathbb{R}^n)}^\frac{\alpha \gamma}{n(\gamma - 1)} \leq C(M,E_0),
\end{align*}
where the final inequalities follow from the first part of this corollary.

For Case 1, we use the Riesz composition formula in Theorem \ref{composition}, the HLS inequality in Lemma \ref{simplehls},
and the Tonelli theorem to obtain
\begin{align*}
 \| \Phi_\alpha * \rho \|_{L^\frac{2n}{\alpha}(\mathbb{R}^n)}
 & = \frac{1}{\alpha}\big\| |\cdot|^{-\alpha} * \rho\big\|_{L^\frac{2n}{\alpha}(\mathbb{R}^n)} \\
 & = \frac{1}{\alpha \kappa_{\frac{n-\alpha}{2},\frac{n-\alpha}{2}}}
 \Big\| |\cdot|^{-\frac{n+\alpha}{2}} * \big(|\cdot|^{-\frac{n+\alpha}{2}} * \rho \big)\Big\|_{L^\frac{2n}{\alpha}(\mathbb{R}^n)} \\
 & \leq \frac{C_{n,\frac{n+\alpha}{2}}}{\alpha \kappa_{\frac{n-\alpha}{2},\frac{n-\alpha}{2}}}
 \Big\| |\cdot|^{-\frac{n+\alpha}{2}} * \rho \Big\|_{L^2(\mathbb{R}^n)} \\
 & = \frac{C_{n,\frac{n+\alpha}{2}}}{\alpha \kappa_{\frac{n-\alpha}{2},\frac{n-\alpha}{2}}}
 \int_{\mathbb{R}^n} \rho\,\Big( |\cdot|^{-\frac{n+\alpha}{2}}* \big(|\cdot|^{-\frac{n+\alpha}{2}} * \rho\big) \Big) \dd \boldsymbol{x} \\
 & = - C_{n,\frac{n+\alpha}{2}} \int_{\mathbb{R}^n} \rho ( \Phi_\alpha * \rho) \dd \boldsymbol{x} \\
 & \leq C(M,E_0),
    \end{align*}
and, for $\alpha \in (0,n-2)$,
\begin{align}\label{comp}
\| \nabla \Phi_\alpha * \rho \|_{L^\frac{2n}{\alpha + 2}(\mathbb{R}^n)}
& \leq \frac{1}{\alpha}\big\| |\cdot|^{-(\alpha + 1)} * \rho\big\|_{L^\frac{2n}{\alpha + 2}(\mathbb{R}^n)}\nonumber \\
& = \frac{1}{\alpha \kappa_{\frac{n-\alpha}{2} - 1,\frac{n-\alpha}{2}}}
 \Big\| |\cdot|^{-\left ( 1 + \frac{n+\alpha}{2} \right )} * \big(|\cdot|^{-\frac{n+\alpha}{2}} * \rho \big) \Big\|_{L^\frac{2n}{\alpha + 2}(\mathbb{R}^n)} \nonumber\\
& \leq \frac{C_{n,\frac{n+\alpha}{2} + 1}}{\alpha \kappa_{\frac{n-\alpha}{2} - 1,\frac{n-\alpha}{2}}}
\Big\| |\cdot|^{-\frac{n+\alpha}{2}} * \rho \Big \|_{L^2(\mathbb{R}^n)},
\end{align}
where $\kappa_{\alpha,\beta}$ is defined in \eqref{A.4a}.

\smallskip
For $\alpha = n-2$, it follows directly
that
\begin{align}\label{newtcase}
    \| \Phi_\alpha * \rho \|_{L^\frac{2n}{\alpha}(\mathbb{R}^n)}
    & \leq \Big\| |\cdot|^{-\frac{n+\alpha}{2}} * \rho \Big\|_{L^2(\mathbb{R}^n)}.
\end{align}
Now, applying Theorem \ref{composition} and the Tonelli theorem, we have
\begin{align}\label{lastone}
        \Big\| |\cdot|^{-\frac{n+\alpha}{2}} * \rho \Big\|_{L^2(\mathbb{R}^n)}
        & = \int_{\mathbb{R}^n} \rho\,\Big( |\cdot|^{-\frac{n+\alpha}{2}} * \big(|\cdot|^{-\frac{n+\alpha}{2}} * \rho \big) \Big)
        \dd \boldsymbol{x}\nonumber \\
        & = - \kappa_{\frac{n-\alpha}{2},\frac{n-\alpha}{2}}\int_{\mathbb{R}^n} \rho ( \Phi_\alpha * \rho) \dd \boldsymbol{x} \nonumber\\
        & \leq C(M,E_0).
\end{align}
Tying together \eqref{comp}--\eqref{lastone},
we obtain that, for $\kappa = -1$ and $\alpha \in (0,n-2]$,
\begin{align*}
    \| \nabla \Phi_\alpha * \rho \|_{L^\frac{2n}{\alpha + 2}(\mathbb{R}^n)} \leq C(M,E_0).
\end{align*}

\smallskip
For $\alpha = 0$, for simplicity, we may assume that $K = \overline{B_D(\mathbf{0})}$ for some $D > 0$, where $B_D(\mathbf{0})$ denotes the ball with origin at $\mathbf{0}$ and radius $D$.
Notice that, for any $\boldsymbol{x} \in K$ and $\boldsymbol{y} \in \mathbb{R}^n$ with
$|\boldsymbol{x} - \boldsymbol{y}| \geq 1$, $\log(|\boldsymbol{x}| + |\boldsymbol{y}|) \geq \log(|\boldsymbol{x} - \boldsymbol{y}|) \geq 0$.
Then, if $\boldsymbol{y} \in K$,
\[
\log(|\boldsymbol{x}| + |\boldsymbol{y}|) \leq \log (1 + 2D).
\]
If $\boldsymbol{y} \in \mathbb{R}^n \setminus K$, then
\[
\log(|\boldsymbol{x}| + |\boldsymbol{y}|) = \frac{1}{2} \log ((|\boldsymbol{x}| + |\boldsymbol{y}|)^2) \leq \frac{1}{2} \log (4 + 4 |\boldsymbol{y}|^2) = \log 2 + \frac{1}{2} \log (1 +  |\boldsymbol{y}|^2).
\]
Thus, for $\boldsymbol{y} \in \mathbb{R}^n$, we have
\begin{equation}\label{logtings}
\log(|\boldsymbol{x}| + |\boldsymbol{y}|) \leq \log (1 + 2D) + \log 2 + \frac{1}{2} \log (1 +  |\boldsymbol{y}|^2) = \log (2 + 4D) + \frac{1}{2} \log (1 +  |\boldsymbol{y}|^2).
\end{equation}
Then, for $\boldsymbol{x} \in K$, we apply the H\"older inequality, \eqref{logtings},
and \eqref{eq4} to obtain
\begin{align}\label{logest}
            & |(\Phi_0 * \rho)(\boldsymbol{x})| \nonumber\\
            & \leq \int_{B_1(\boldsymbol{x})} \log(|\boldsymbol{x} - \boldsymbol{y}|^{-1}) \rho(\boldsymbol{y}) \dd \boldsymbol{y} + \int_{B_1(\boldsymbol{x})^c} \log(|\boldsymbol{x} - \boldsymbol{y}|) \rho(\boldsymbol{y}) \dd \boldsymbol{y} \nonumber\\
            & \leq \| \log(|\cdot|) \|_{L^\frac{\gamma}{\gamma - 1}(B_1)} \|\rho\|_{L^\gamma(B_1(\boldsymbol{x}))}
            + \int_{B_1(\boldsymbol{x})^c} \big(\log (2 + 4D) + \frac{1}{2} \log (1 +  |\boldsymbol{y}|^2) \big) \rho(\boldsymbol{y}) \dd \boldsymbol{y}\nonumber\\
            & \leq C \|\rho\|_{L^\gamma(B_1(\boldsymbol{x}))} + \log (2 + 4D) \|\rho\|_{L^1(B_1(\boldsymbol{x})^c)} + \frac{1}{2} \int_{B_1(\boldsymbol{x})^c} \log (1 +  |\boldsymbol{y}|^2)  \rho(\boldsymbol{y}) \dd \boldsymbol{y} \nonumber\\
            & \leq C(K,M,E_0,T).
    \end{align}
Therefore, we have
    \[
    \|\Phi_0 * \rho\|_{L^\infty(K)} \leq C(K,M,E_0,T).
    \]

For $\alpha \in (-1,0)$, we assume again that $K = \overline{B_D}$ for some $D > 0$.
Using \eqref{minus}, we obtain that, for any $\boldsymbol{x} \in K$ and $y \in \mathbb{R}^n$,
    \begin{equation}\label{minustings}
        |\Phi_\alpha(\boldsymbol{x} - \boldsymbol{y})| \leq \frac{2^{-\alpha}}{-\alpha} \left  (k_\alpha(1 + D^2) + k_\alpha(1 + |\boldsymbol{y}|^2) \right ).
    \end{equation}
Thus, similar to \eqref{logest}, using \eqref{minustings} and \eqref{eq4}, we have
    \begin{align*}
        \begin{split}
            |(\Phi_\alpha * \rho)(\boldsymbol{x})|
            & \leq \frac{2^{-\alpha}}{-\alpha} \int_{\mathbb{R}^n} \left  (k_\alpha(1 + D^2) + k_\alpha(1 + |\boldsymbol{y}|^2) \right ) \rho(\boldsymbol{y}) \dd \boldsymbol{y} \\
            & \leq C(K,M,E_0,T).
        \end{split}
    \end{align*}
Then we obtain
    \[
    \|\Phi_\alpha * \rho\|_{L^\infty(K)} \leq C(K,M,E_0,T).
    \]

\smallskip
For $\alpha \in (-1,0]$, take $1 < p < \min\{\frac{2n}{\alpha + 2},\gamma \}$.
Applying the HLS inequality in Lemma \ref{simplehls}, interpolation, and \eqref{eq4}, we have
 \begin{align*}
& \|\nabla \Phi_\alpha * \rho \|_{L^\frac{2n}{\alpha + 2}(K)} \\
& \leq \| (|\cdot|^{-(\alpha + 1)} \mathds{1}_{B_1}) * \rho \|_{L^\frac{2n}{\alpha + 2}(K)} + \| (|\cdot|^{-(\alpha + 1)} \mathds{1}_{B_1^c}) * \rho \|_{L^\frac{2n}{\alpha + 2}(K)} \\
& \leq \| (|\cdot|^{-(1 + \frac{\alpha}{2} + n(1-\frac{1}{p}))}) * \rho \|_{L^\frac{2n}{\alpha + 2}(K)} + \| \mathds{1}_{B_1^c} * \rho \|_{L^\frac{2n}{\alpha + 2}(K)} \\
& \leq C_{n,1 + \frac{\alpha}{2}+ n(1-\frac{1}{p})} \|\rho\|_{L^p(\mathbb{R}^n)} + |K|^{\frac{\alpha + 2}{2n}} \|\rho\|_{L^1(\mathbb{R}^n)} \\
& \leq C_{n,1 + \frac{\alpha}{2} + n(1-\frac{1}{p})} \|\rho\|_{L^1(\mathbb{R}^n)}^{\frac{\gamma - p}{p(\gamma - 1)}} \|\rho\|_{L^\gamma(\mathbb{R}^n)}^{\frac{\gamma(p-1)}{p(\gamma - 1)}} + |K|^{\frac{\alpha + 2}{2n}} \|\rho\|_{L^1(\mathbb{R}^n)} \\
& \leq C(K,M,E_0,T).
  \end{align*}
\end{proof}

\begin{remark}
It is still open for the repulsive case $\kappa = -1$ and $\alpha \in (n-2,n-1)$
without any restriction on $\gamma$,
owing to the fact that the $L^\frac{2n}{\alpha + 2}$--norm
of the derivative of the potential term cannot be bounded by the potential energy of the system.
\end{remark}

\subsection{Expanding the domain}
We need to ensure that, as $b \to \infty$, our domain $\Omega^T$ converges to the whole space
$\mathbb{R}^n$.

\begin{lemma}[Expanding of domain $\Omega^T$]\label{expand}
Given $T>0$ and $\v\in(0,\v_0]$, there exists a positive constant $C_1(M,E_0,T,\v)>0$ such that,
if $b\geq \max\{C_1(M,E_0,T,\v),\mathfrak{B}(\v)\}$ {\rm (}as defined in {\rm Lemma \ref{appendix5}}{\rm )},
	\begin{align}\label{2.60}
		b(t)\geq \frac{b}{2}\qquad\,\, \mbox{for $t\in[0,T]$}.
	\end{align}
\end{lemma}

This theorem can be achieved by following the same argument as in \cite[Lemma 3.4]{Chen2021}.
Therefore, as $b \to \infty$, $\lim_{b \to \infty} \inf_{t \in [0,T]} b(t) = \infty$, as we desire.

\subsection{BD-type entropy}
We now need to make uniform estimates on the derivatives of the density in order
to apply the Aubin-Lions lemma to prove the convergence
of $(\rho^{\v,b}, \rho^{\v,b}u^{\v,b})$ as $b \to \infty$
to global weak solutions $(\rho^\v, \rho^\v u^\v)$ of \eqref{0.2}.
We do so by proving a BD entropy bound, taking advantage of the fact that
our viscosity terms
satisfy $\lambda(\rho) = \rho \mu'(\rho) - \mu(\rho)$.
We have two different BD-type estimates: The first is to follow the same approach
as in \cite{Chen2021} for $\alpha \in (-1,n-2]$ due to the inability to
integrate by parts rigorously for $\alpha > n-2$,
as $\Phi_\alpha * \rho$ is only twice differentiable for $\alpha \in (-1,n-2]$.
This approach fails to provide a suitable estimate in the attractive case $\kappa = 1$, due to the lack of the control of the boundary term $a^{n-1} (\rho(\Phi_\alpha * \rho)_r)(t,a)$
as we have no {\it a priori} of the density at the lower boundary $r = a$ except for \eqref{away}.
Moreover, this is insufficient, as we require this to be a uniform bound in $b$,
which is not the case.
Although, for $\alpha = n-2$, we have no issue, as we can show that the boundary term disappears.
However, in the repulsive case $\kappa = -1$, we can show that the boundary term has favourable sign
and can thus be controlled.

\begin{lemma}[BD-Type Entropy Estimate I]\label{oldBD}
Under the conditions of {\rm Lemma {\rm \ref{Energy}}},
for $\kappa = -1$ with $\alpha \in (-1,n-2]$ or $\kappa = 1$ with $\alpha = n-2$, any given $T > 0$,
$\v\in(0,\v_0]$, and $b \geq \max\{C_1(M,E_0,T,\varepsilon), \mathfrak{B}(\varepsilon)\}$
{\rm (}as defined in {\rm Lemmas \ref{expand}} and {\rm \ref{appendix5}}{\rm )}, the following estimate holds{\rm :}
\begin{align*}
&\frac{\varepsilon^2}{4}\int_a^{b(t)} \big|(\sqrt{\rho(t,r)})_r\big|^2\,r^{n-1} \dd r
	+\frac{4 a_0\varepsilon}{\gamma} \int_0^t\int_a^{b(s)} |(\rho^{\frac{\gamma}{2}})_r|^2\,r^{n-1}\dd r\dd s\\
	&\quad+\frac{1}{n}p(\rho(t,b(t)))b(t)^n
	+\frac{1}{n\varepsilon}\int_0^t p(\rho(s,b(s)))p'(\rho(s,b(s))) b(s)^n\dd s \\
	&\leq C(E_0,M,T) \qquad\,\,\mbox{for all $t \in [0,T]$}.
\end{align*}
\end{lemma}

\begin{proof} We divide the proof into three steps.

\medskip
1. We first consider \eqref{Eul} in Lagrangian coordinates.
Taking the derivative of the mass equation $\eqref{lag}_1$ with respect to $x$, we have
    \begin{align*}
            \rho_{x \tau} = -\big(\rho^2(r^{n-1}u)_x\big)_x.
    \end{align*}
Plugging it into $\eqref{lag}_2$ leads to
    \begin{align}\label{1.16}
        u_\tau + r^{n-1}p_x = -\varepsilon r^{n-1} \rho_{x \tau} - \varepsilon (n-1)r^{n-2}\rho_x u - \kappa \rho r^{n-1}(\Phi_\alpha * \rho)_x.
    \end{align}
    Furthermore, \eqref{1.16} can be simplified to
    \begin{align}\label{1.17}
        (u + \varepsilon r^{n-1}\rho_x)_\tau + r^{n-1}p_x = - \kappa \rho r^{n-1}(\Phi_\alpha * \rho)_x.
    \end{align}
Now, multiplying \eqref{1.17} by $u + \varepsilon r^{n-1}\rho_x$ and integrating with respect to $x$
over $[0,\frac{M}{\omega_n}]$, we obtain
\begin{align}\label{1.18}
	& \frac{1}{2}\frac{\d}{\d\tau}\int_0^{\frac{M}{\omega_n}}(u+\v r^{n-1}\r_x)^2 \dd x
	+\v\int_0^{\frac{M}{\omega_n}}r^{2n-2}p(\rho)_x\rho_x\dd x+\int_0^{\frac{M}{\omega_n}}r^{n-1}p(\rho)_x u\dd x\nonumber\\
	& = - \kappa \int^{\frac{M}{\omega_n}}_0 u\rho r^{n-1}(\Phi_\alpha * \rho)_x \dd x - \kappa \varepsilon \int^{\frac{M}{\omega_n}}_0 \rho \rho_x r^{2(n-1)} (\Phi_\alpha * \rho)_x \dd x.
\end{align}
Using \eqref{1.4}, we can rewrite \eqref{1.18} in the form:
\begin{align}\label{1.19}
        & \frac{1}{2}\frac{\d}{\d\tau}\int_0^{\frac{M}{\omega_n}}(u+\v r^{n-1}\r_x)^2 \dd x
	+\v\int_0^{\frac{M}{\omega_n}}r^{2n-2}p(\rho)_x\rho_x\dd x+\int_0^{\frac{M}{\omega_n}}r^{n-1}p(\rho)_x u\dd x\nonumber\\
	& = - \frac{\kappa}{2} \frac{\d}{\d \tau} \int^{\frac{M}{\omega_n}}_0(\Phi_\alpha * \rho) \dd x - \kappa \varepsilon \int^{\frac{M}{\omega_n}}_0 \rho \rho_x r^{2(n-1)} (\Phi_\alpha * \rho)_x \dd x.
\end{align}
Using $\eqref{lag}_1$ and \eqref{boundary}, we have
\begin{align}\label{1.20}
    \rho_\tau(\tau,\frac{M}{\omega_n}) = -(\rho^2(r^{n-1}u)_x)(\tau,\frac{M}{\omega_n}) = -\frac{a_0 \rho^\gamma}{\varepsilon}(\tau,\frac{M}{\omega_n}).
\end{align}
Now, applying $\eqref{lag}_1$, \eqref{deriv}, and \eqref{1.20}, we obtain
\begin{align}\label{1.21}
        & \int_0^{\frac{M}{\omega_n}}p(\rho)_xu\,r^{n-1}\dd x\nonumber\\
	& =-\int_0^{\frac{M}{\omega_n}}p(\rho)(r^{n-1}u)_x \dd x
	+( p(\rho)u r^{n-1})(\tau,\frac{M}{\omega_n})\nonumber\\
        & =\frac{\d}{\d\tau}\int_0^{\frac{M}{\omega_n}}e(\rho)\dd x
	+( p(\rho)r^{n-1})(\tau,\frac{M}{\omega_n})\, r_\tau(\tau,\frac{M}{\omega_n})\nonumber\\
	& =\frac{\d}{\d\tau}\int_0^{\frac{M}{\omega_n}}e(\rho)\dd x
	+\Big(\frac{1}{n}p(\rho(\tau,\frac{M}{\omega_n}))b(\tau)^n\Big)_\tau
	-\frac{1}{n}p'(\rho(\tau,\frac{M}{\omega_n}))\rho_{\tau}(\tau,\frac{M}{\omega_n})b(\tau)^n\nonumber\\
	& =\frac{\d}{\d\tau}\int_0^{\frac{M}{\omega_n}}e(\rho) \dd x
	+\Big(\frac{1}{n}p(\rho(\tau,\frac{M}{\omega_n}))b(\tau)^n\Big)_\tau
    + \frac{1}{n\v}p(\rho(\tau,\frac{M}{\omega_n}))p'(\rho(\tau,\frac{M}{\omega_n}))b(\tau)^n.
\end{align}
Substituting \eqref{1.21} into \eqref{1.19}, we have
\begin{align}\label{1.22}
	&\frac{\d}{\d\tau}\Big\{\int_0^{\frac{M}{\omega_n}}\frac{1}{2}(u+\v r^{n-1}\r_x)^2\,\dd x
	+\int_0^{\frac{M}{\omega_n}}e(\rho)\,\dd x + \frac{\kappa}{2}\int^{\frac{M}{\omega_n}}_0\Phi_\alpha * \rho \dd x\Big\}
       \nonumber\\
	&\quad\, +\v \int_0^{\frac{M}{\omega_n}}p'(\r)\r_x^2\,r^{2n-2} \dd x
	+\Big(\frac{1}{n}p(\rho(\tau,\frac{M}{\omega_n}))b(\tau)^n\Big)_\tau\nonumber\\
    &\quad\, +\frac{1}{n\v}p(\rho(\tau,\frac{M}{\omega_n}))p'(\rho(\tau,\frac{M}{\omega_n})) b(\tau)^n \nonumber\\
	&= -  \kappa \varepsilon \int^{\frac{M}{\omega_n}}_0 \rho \rho_x r^{2(n-1)} (\Phi_\alpha * \rho)_x \dd x\nonumber\\
    &= - \kappa \varepsilon \left [ \rho^2 r^{2(n-1)} (\Phi_\alpha * \rho)_x \right ]^\frac{M}{\omega_n}_0 + \kappa \varepsilon \int^{\frac{M}{\omega_n}}_0 \rho (\rho r^{2(n-1)} (\Phi_\alpha * \rho)_x)_x \dd x.
\end{align}

\smallskip
2. Consider the last integral of \eqref{1.22}. Using the HLS inequality in Lemma \ref{simplehls},
we obtain the estimate for $\alpha \in (-1,n-2]$:
\begin{align*}
    \begin{split}
        & \int^{\frac{M}{\omega_n}}_0 \rho \big(\rho r^{2(n-1)} (\Phi_\alpha * \rho)_x\big)_x \dd x \\
        & = \frac{1}{\omega_n} \int_{\Omega_t} (\rho \Delta (\Phi_\alpha * \rho))(t,\boldsymbol{y}) \dd \boldsymbol{y} \\
        & \leq \frac{1}{\omega_n} \| \rho \|_{L^{\frac{2n}{2n - (\alpha + 2)}}(\Omega_t)} \| \Delta (\Phi_\alpha * \rho) \|_{L^\frac{2n}{\alpha + 2}(\mathbb{R}^n)} \\
        & = \frac{\alpha + 1}{\omega_n} \| \rho \|_{L^{\frac{2n}{2n - (\alpha + 2)}}(\Omega_t)}
         \big\| \, |\cdot |^{-(\alpha + 2)} * \rho\big\|_{L^\frac{2n}{\alpha + 2}(\mathbb{R}^n)} \\
        & \leq \frac{C_{\alpha + 2,n}(\alpha + 1)}{\omega_n} \| \rho \|_{L^{\frac{2n}{2n - (\alpha + 2)}}(\Omega_t)}^2.
    \end{split}
\end{align*}

\smallskip
\noindent{Case} {\rm 1:} $\kappa = -1$.
We can show that the last term on the RHS of \eqref{1.22} is negative as
 \begin{equation*}
 \Delta \Phi_\alpha =
 \begin{cases}
     \frac{(n-2) - \alpha}{|\,\cdot\,|^{\alpha + 2}}\quad & \mbox{for $\alpha \in (-1,n-2)$}, \\[2mm]
    \omega_n \delta_0 & \mbox{for $\alpha = n-2$},
 \end{cases}
 \end{equation*}
in the sense of distributions, where $\delta_0$ is the Dirac mass centred at 0.
Then we have
\begin{align}\label{deriva}
    \begin{split}
        \rho (\rho r^{2(n-1)} (\Phi_\alpha * \rho)_x)_x = \Delta (\Phi_\alpha * \rho) & =
        \begin{cases}
            (n-2 - \alpha)(|\cdot|^{-(\alpha + 2)} * \rho) &\mbox{for $\alpha \in (-1,n-2)$}, \\
            \omega_n \rho & \mbox{for $\alpha = n-2$}
        \end{cases} \\[2mm]
        & \geq 0.
    \end{split}
\end{align}
Thus, we obtain
\begin{align*}
\begin{split}
	&\frac{\d}{\d\tau}\Big\{\int_0^{\frac{M}{\omega_n}}\frac{1}{2}(u+\v r^{n-1}\r_x)^2\,\dd x
	+\int_0^{\frac{M}{\omega_n}}e(\rho)\,\dd x + \frac{\kappa}{2}\int^{\frac{M}{\omega_n}}_0 \Phi_\alpha * \rho \dd        x\Big\}\\
	&\quad +\v \int_0^{\frac{M}{\omega_n}}p'(\r)\r_x^2\,r^{2n-2} \dd x
	+\Big(\frac{1}{n}p(\rho(\tau,\frac{M}{\omega_n}))b(\tau)^n\Big)_\tau\\
	&\quad +\frac{1}{n\v}p(\rho(\tau,\frac{M}{\omega_n}))p'(\rho(\tau,\frac{M}{\omega_n})) b(\tau)^n \\
        &\leq - \kappa \varepsilon \left [ \rho^2 r^{2(n-1)} (\Phi_\alpha * \rho)_x \right ]^\frac{M}{\omega_n}_0.
\end{split}
\end{align*}

\medskip
\noindent{Case} {\rm 2:} $\kappa = 1$ and $\gamma \geq \frac{2n}{2n - \alpha - 2}$.
Then we can control the integral by the internal energy:
\begin{align*}
    \begin{split}
        \| \rho \|_{L^{\frac{2n}{2n -\alpha - 2)}}(\Omega_t)}^2 \leq C \int^{b(s)}_a \rho\big(1 + e(\rho)\big)r^{n-1} \dd r.
    \end{split}
\end{align*}

\smallskip
\noindent{Case} {\rm 3:} $\kappa = 1$ and $\gamma \in(\frac{2n}{2n - \alpha}, \frac{2n}{2n - \alpha -2})$.
Applying the interpolation yields
\begin{align*}
        \| \rho \|_{L^{\frac{2n}{2n -\alpha-2}}(\Omega_t)}
        = \| \rho \|_{L^{\frac{n\gamma}{n-2}}(\Omega_t)}^{\theta} \| \rho \|_{L^{\gamma}(\Omega_t)}^{1 - \theta} \qquad \text{ with } \theta
        = \frac{2n + (\alpha + 2 - 2n)\gamma}{4}.
\end{align*}
This can be controlled as in \cite{Chen2021}.

\medskip
3. Now, looking at the second to the last term
of \eqref{1.22}, we have
\begin{equation*}
    \big[ \rho^2 r^{2(n-1)} (\Phi_\alpha * \rho)_x \big]^\frac{M}{\omega_n}_0
    = \big[ \rho r^{n-1} (\Phi_\alpha * \rho)_r \big]^{b(t)}_a.
\end{equation*}
From \eqref{deriva}, we see that $r \mapsto r^{n-1}(\Phi_\alpha * \rho)_r$ is increasing in $r$.
Since $\lim_{r \searrow 0} r^{n-1}(\Phi_\alpha * \rho)_r = 0$,
then $a^{n-1}(\rho(\Phi_\alpha * \rho)_r)(t,a) \geq 0$, owing to the fact that $\rho(t,a) > 0$.

In the case that $\kappa = - 1$, we can ignore this term as it has a negative sign.

In the case that $\alpha = n - 2$, we obtain
\begin{align*}
    \begin{split}
        (r^{n-1}(\Phi_{n-2} * \rho)_r)_r = r^{n-1} \Delta (\Phi_{n-2} * \rho) = \omega_n r^{n-1} \rho(t,r) = 0
        \qquad\mbox{for all $r \in (0,a)$}.
    \end{split}
\end{align*}
Hence, $r \mapsto r^{n-1}(\Phi_{n-2} * \rho)_r$ is constant in $r \in (0,a)$.
Since $\lim_{r \searrow 0} r^{n-1}(\Phi_{n-2} * \rho)_r = 0$, then
$a^{n-1}(\rho(\Phi_{n-2} * \rho)_r)(t,a) = 0$.

\smallskip
4. We now focus on the other boundary term for $\alpha \in (-1,n-2)$ to obtain
\begin{align}\label{1.24}
\bigg| \int_{\mathbb{S}^{n-1}} \frac{b(s)\boldsymbol{e_1} - \eta \boldsymbol{y}}{|b(s)\boldsymbol{e_1} - \eta\boldsymbol{y}|^{\alpha + 2}} \cdot \boldsymbol{e_1}
\dd \sigma(\boldsymbol{y}) \bigg| \leq \omega_n |b(s) - \eta|^{-(\alpha + 1)}.
\end{align}
Employing \eqref{1.24}, \eqref{2.60}, \eqref{continuity.8}, and $a = \frac{1}{b}$,
we obtain that, for $\eta \in (a,b(s))$,
\begin{align}\label{1.25}
 \bigg| \int_{\mathbb{S}^{n-1}} \frac{b(s)\boldsymbol{e_1} - \eta\boldsymbol{y}}{|b(s)\boldsymbol{e_1} - \eta\boldsymbol{y}|^{\alpha + 2}} \cdot \boldsymbol{e_1}
 \dd \sigma(\boldsymbol{y}) \bigg|
 & \leq C \min \{|b(s) - \eta|^{-(\alpha + 1)}, \,(\eta b(s))^{-\frac{\alpha + 1}{2}} \} \nonumber\\
& \leq C \min \{|b(s) - \eta|^{-(\alpha + 1)},\, 1 \} \leq C.
\end{align}
From Lemma 3.4 in \cite{Chen2021}, we see that, for $b \geq \max\{C_1(M,E_0,T,\varepsilon), \mathfrak{B}(\varepsilon)\}$,
\begin{align}\label{1.26}
    \bigg| \int^t_0 u(s,b(s)) \dd s \bigg| \leq \frac{1}{4}b.
\end{align}
Using \eqref{ineq}, \eqref{1.25}, and \eqref{1.26}, we obtain
\begin{align*}
    \begin{split}
        & \bigg| \int^t_0 (\rho r^{n-1}(\Phi_\alpha * \rho)_r)(s,b(s)) \dd s \bigg| \\
        & \leq \int^t_0 \rho_0(b) \Big( \int^{b(s)}_a \Big| \int_{\mathbb{S}^{n-1}} \frac{b(s)\boldsymbol{e_1} - \eta\boldsymbol{y}}{|b(s)\boldsymbol{e_1} - \eta\boldsymbol{y}|^{\alpha + 2}} \cdot \boldsymbol{e_1} \dd \sigma(\boldsymbol{y}) \Big| \rho(s,\eta) \eta^{n-1} \dd\eta \Big) b(s)^{n-1} \dd s \\
        & \leq C \int^t_0 \rho_0(b) \Big( \int^{b(s)}_a  \rho(s,\eta) \eta^{n-1} \dd\eta \Big) b(s)^{n-1} \dd s \\
        & \leq C(M) \int^t_0 \rho_0(b) b(s)^{n-1} \dd s \\
        & = C(M) \int^t_0 \rho_0(b) \Big( b + \int^s_0 u(\tau,b(\tau)) \dd \tau \Big)^{n-1} \dd s \\
        & \leq C(M) \int^t_0 \rho_0(b) b^{n-1} \dd s \\
        & \leq C(M,T) b^{-(n-\beta)} b^{n-1} \\
        & \leq C(M,T),
    \end{split}
\end{align*}
where we have used that $\rho_0(b) \cong b^{-(n-\beta)}$, with $\beta = \min\{ \frac{1}{2},(1 - \frac{1}{\gamma})n \}$. We can also obtain this bound for $\alpha = n-2$ from \cite{Chen2021}. Hence, similar to \cite{Chen2021}, we see that,
for $\kappa = -1$ with $\alpha \in(-1,n-2]$ or $\kappa = 1$ with $\alpha = n-2$,
\begin{align*}
\begin{split}
	&\frac{\varepsilon^2}{4}\int_a^{b(t)} \big|(\sqrt{\rho(t,r)})_r\big|^2\,r^{n-1} \dd r
	+\frac{4 a_0\varepsilon}{\gamma} \int_0^t\int_a^{b(s)} \big|(\rho^{\frac{\gamma}{2}})_r\big|^2\,r^{n-1}\dd r\dd s\\
	&\quad+\frac{1}{n}p(\rho(t,b(t)))b(t)^n
	+\frac{1}{n\varepsilon}\int_0^t p(\rho(s,b(s)))p'(\rho(s,b(s))) b(s)^n\dd s \\[1mm]
	&\leq C(E_0,M,T).
 \end{split}
\end{align*}
\end{proof}

\medskip
\begin{remark}\label{3.10a}
In {\rm Step 2} above, when $\kappa = 1$, it follows from \eqref{deriva}
that the boundary term given
by $\kappa \varepsilon a^{n-1}(\rho(\Phi_\alpha * \rho)_r)(t,a)$ has a bad sign.
When $\alpha \in (-1,n-2)$,  we have
\begin{align}\label{boundar}
        a^{n-1}(\Phi_\alpha * \rho)_r(t,a) & = \int^a_0 (r^{n-1}(\Phi_\alpha * \rho)_r)_r \dd r
        = \frac{(n-2) - \alpha}{\omega_n} \int_{B_a} (|\cdot|^{-(\alpha + 2)} * \rho)(\boldsymbol{y}) \dd \boldsymbol{y}>0.
\end{align}
For $\gamma < \frac{n}{(n-2) - \alpha}$, by the HLS inequality from {\rm Lemma \ref{simplehls}} and \eqref{boundar}, we have
\begin{align*}
        a^{n-1}(\Phi_\alpha * \rho)_r(t,a) & \leq C \| \rho \|_{L^\gamma(\Omega_t)} \leq C(E_0).
\end{align*}
For $\gamma \geq \frac{n}{(n-2) - \alpha}$, using the above and the interpolation, we obtain
\begin{align*}
    \begin{split}
        a^{n-1}(\Phi_\alpha * \rho)_r(t,a) & \leq C(M,E_0).
    \end{split}
\end{align*}
Thus, with such estimates, in order to control $\kappa \varepsilon a^{n-1}(\rho(\Phi_\alpha * \rho)_r)(t,a)$,
it suffices to control $\rho(t,a)$.
However, such an {\it a priori} estimate on the density at $r=a$ has not yet been available.
\end{remark}

\smallskip
We now prove the second BD-type estimate result. Here we are able to provide a result
for both the repulsive and attractive case $\kappa \in \{-1,1\}$ and for the more singular Riesz potential
$\alpha \in (0,n-1)$. However, the estimate given by the Gr\"{o}nwall inequality is rough,
compared to that of the case where integration by parts can be applied, so that
we obtain a different range of $\gamma$ for which it can be applied.
This is due to the non-locality of the Riesz potential, while the potential is a local operator
in the already proved Coulomb case $\alpha = n-2$.

\begin{lemma}[BD-Type Entropy Estimate II]\label{newbd}
Under the conditions of {\rm Lemma {\rm \ref{Energy}}}, for $\alpha \in (-1,n-1)$,
$\gamma \geq \frac{3n}{3n - 2(1+\alpha)}$, any given $T > 0$, $\v\in(0,\v_0]$,
and $b \geq \max\{C_1(M,E_0,T,\varepsilon), \mathfrak{B}(\varepsilon)\}$
{\rm (}as defined in {\rm Lemmas \ref{expand}} and {\rm \ref{appendix5}}{\rm )},
the following uniform estimate holds for $t \in [0,T]${\rm :}
\begin{align*}
\begin{split}
	&\frac{\varepsilon^2}{4}\int_a^{b(t)} \big|(\sqrt{\rho(t,r)})_r\big|^2\,r^{n-1} \dd r
	+\frac{4 a_0\varepsilon}{\gamma} \int_0^t\int_a^{b(s)} \big|(\rho^{\frac{\gamma}{2}})_r\big|^2\,r^{n-1}\dd r\dd s\\
	&\quad+\frac{1}{n}p(\rho(t,b(t)))b(t)^n
	+\frac{1}{n\varepsilon}\int_0^t p(\rho(s,b(s)))p'(\rho(s,b(s))) b(s)^n\dd s \\
	&\leq C(E_0,M,T).
 \end{split}
\end{align*}
\end{lemma}

\begin{proof}
We use calculations \eqref{1.16}--\eqref{1.21} done in Lemma \ref{oldBD}.
Substituting \eqref{1.21} into \eqref{1.19}, we have
\begin{align}\label{1.22b}
	&\frac{\d}{\d\tau}\bigg\{\int_0^{\frac{M}{\omega_n}}\frac{1}{2}(u+\v r^{n-1}\r_x)^2\,\dd x
	+\int_0^{\frac{M}{\omega_n}}e(\rho)\,\dd x \bigg\}
      +\v \int_0^{\frac{M}{\omega_n}}p'(\r)\r_x^2\,r^{2n-2} \dd x\nonumber\\
	& \quad +\Big(\frac{1}{n}p(\rho(\tau,\frac{M}{\omega_n}))b(\tau)^n\Big)_\tau +\frac{1}{n\v}p(\rho(\tau,\frac{M}{\omega_n}))p'(\rho(\tau,\frac{M}{\omega_n})) b(\tau)^n \nonumber\\
	& = - \kappa \int^{\frac{M}{\omega_n}}_0 (u + \varepsilon r^{n-1} \rho_x)\rho r^{n-1}(\Phi_\alpha * \rho)_x \dd x.
\end{align}

In the case that $\gamma \geq \frac{3n}{3n - 2(1+\alpha)}$, we apply the H\"older inequality
and then the HLS inequality from Lemma \ref{simplehls} to obtain
\begin{align}\label{whole}
        & \int^{b(t)}_a |(\Phi_\alpha * \rho)_r(t,r)|^2 \rho(t,r) r^{n-1} \dd r\nonumber\\
        & \leq \frac{1}{\omega_n} \int_{\Omega_t} |\nabla(\Phi_\alpha * \rho)(t,\boldsymbol{y})|^2 \rho(t,\boldsymbol{y}) \dd\boldsymbol{y} \nonumber\\
        & \leq C \| \nabla(\Phi_\alpha * \rho) \|^2_{L^{\frac{3n}{1+\alpha}}(\Omega_t)} \| \rho \|_{L^{\frac{3n}{3n - 2(1+\alpha)}}(\Omega_t)} \nonumber\\
        & \leq C \| |\cdot|^{-(\alpha + 1)}* \rho \|^2_{L^{\frac{3n}{1+\alpha}}(\Omega_t)} \| \rho \|_{L^{\frac{3n}{3n - 2(1+\alpha)}}(\Omega_t)} \nonumber\\
        & \leq C \| \rho \|^3_{L^{\frac{3n}{3n - 2(1+\alpha)}}(\Omega_t)}\nonumber\\
        & \leq C \| \rho \|^{3\theta}_{L^\gamma(\Omega_t)} \| \rho \|^{3(1 - \theta)}_{L^1(\Omega_t)} \nonumber\\
        & \leq C(M,E_0),
\end{align}
where $\theta = \frac{2\gamma(1 + \alpha) + 3n}{3n\gamma}$. Applying the Cauchy-Schwartz inequality to \eqref{1.22b}
and using \eqref{whole}, we have
\begin{align}\label{1.27}
    \begin{split}
        &\frac{\d}{\d\tau}\bigg\{\int_0^{\frac{M}{\omega_n}}\frac{1}{2}(u+\v r^{n-1}\r_x)^2 \dd x
	+\int_0^{\frac{M}{\omega_n}}e(\rho)\,\dd x \bigg\}
    +\v \int_0^{\frac{M}{\omega_n}}p'(\r)\r_x^2\,r^{2n-2} \dd x \\
	& \quad\, +\Big(\frac{1}{n}p(\rho(\tau,\frac{M}{\omega_n}))b(\tau)^n\Big)_\tau
      +\frac{1}{n\v}p(\rho(\tau,\frac{M}{\omega_n}))p'(\rho(\tau,\frac{M}{\omega_n})) b(\tau)^n \\
	& \leq C(M,E_0) + \int^{\frac{M}{\omega_n}}_0 \frac{1}{2}\big(u + \varepsilon r^{n-1} \rho_x\big)^2 \dd x.\\
    \end{split}
\end{align}
Applying the Gr\"{o}nwall type inequality to \eqref{1.27} yields
\begin{align*}
    \begin{split}
        & \int_0^{\frac{M}{\omega_n}}\frac{1}{2}(u+\v r^{n-1}\r_x)^2 \dd x \\
        &\quad+ \int_0^\tau e^{\tau -s} \bigg \{ \Big(\int_0^{\frac{M}{\omega_n}}e(\rho)\,\dd x \Big)_s
         + \v \int_0^{\frac{M}{\omega_n}}p'(\r)\r_x^2\,r^{2n-2} \dd x
         +\Big(\frac{1}{n}p(\rho(s,\frac{M}{\omega_n}))b(s)^n\Big)_s \\
         &\qquad\qquad \qquad\,\,\,
          + \frac{1}{n\v}p(\rho(s,\frac{M}{\omega_n}))p'(\rho(s,\frac{M}{\omega_n})) b(s)^n \bigg \} \dd s \\
        & \leq C(M,E_0) \int^\tau_0 e^{\tau -s} \dd s + e^\tau \int_0^{\frac{M}{\omega_n}}\frac{1}{2}(u_0+\v r^{n-1}\r_{0,x})^2 \dd x.
    \end{split}
\end{align*}
Integrating by parts, we have
\begin{align}\label{1.29}
        & \int_0^{\frac{M}{\omega_n}}\frac{1}{2}(u+\v r^{n-1}\r_x)^2 \dd x + \int_0^{\frac{M}{\omega_n}}e(\rho)\,\dd x
           + \frac{1}{n}p(\rho(\tau,\frac{M}{\omega_n}))b(\tau)^n\nonumber\\
        &+ \int_0^\tau e^{\tau -s} \bigg \{ \int_0^{\frac{M}{\omega_n}}e(\rho)\,\dd x
          + \v \int_0^{\frac{M}{\omega_n}}p'(\r)\r_x^2\,r^{2n-2} \dd x + \frac{1}{n}p(\rho(s,\frac{M}{\omega_n}))b(s)^n\nonumber\\
        & \qquad\qquad\qquad + \frac{1}{n\v}p(\rho(s,\frac{M}{\omega_n}))p'(\rho(s,\frac{M}{\omega_n})) b(s)^n \Bigg \} \dd s \nonumber\\
        & \leq e^\tau \bigg \{ \int_0^{\frac{M}{\omega_n}}\frac{1}{2}(u_0+\v r^{n-1}\r_{0,x})^2 \dd x + \int_0^{\frac{M}{\omega_n}}e(\rho_0)\,\dd x \bigg \}
        + C(M,E_0) (e^\tau - 1) + \frac{e^\tau}{n}p(\rho_0(\frac{M}{\omega_n}))b^n.
\end{align}
Using the positivity of the integrands in the integral of the second line in \eqref{1.29}, we conclude the following inequality:
\begin{align*}
 \begin{split}
        & \int_0^{\frac{M}{\omega_n}}\frac{1}{2}(u+\v r^{n-1}\r_x)^2 \dd x + \int_0^{\frac{M}{\omega_n}}e(\rho)\,\dd x
          + \frac{1}{n}p(\rho(\tau,\frac{M}{\omega_n}))b(\tau)^n \\
        &\quad\,+ \int_0^\tau \bigg \{ \v \int_0^{\frac{M}{\omega_n}}p'(\r)\r_x^2\,r^{2n-2} \dd x
         + \frac{1}{n\v}p(\rho(s,\frac{M}{\omega_n}))p'(\rho(s,\frac{M}{\omega_n})) b(s)^n \bigg \} \dd s \\
        & \leq e^\tau \bigg \{ \int_0^{\frac{M}{\omega_n}}\frac{1}{2}(u_0+\v r^{n-1}\r_{0,x})^2 \dd x + \int_0^{\frac{M}{\omega_n}}e(\rho_0)\,\dd x \bigg \}
        + C(M,E_0) (e^\tau - 1) + \frac{e^\tau}{n}p(\rho_0(\frac{M}{\omega_n}))b^n.
    \end{split}
\end{align*}
\end{proof}

\subsection{Higher integrability}
In order to apply the compensated compactness framework in Chen-Perepelitsa \cite{Chen_2010} to prove
the convergence of weak solutions $(\rho^\v,\rho^\v u^\v)$ of \eqref{0.2} to
finite-energy solutions of \eqref{0.0}--\eqref{0.1} as $\v \to 0$,
we need higher integrability uniform estimates of the density and the velocity.
In the following lemma, we require that $\gamma > \frac{1}{n-1- \alpha}$.
Indeed, this is only required so that we can obtain a pointwise estimate
for the derivative of the convolution of the kernel with the density.

\begin{lemma}[Higher Integrability of the Density]\label{higher}
Let $(\rho,u)$ be the smooth solution of \eqref{Eul}. Then, under the assumptions of {\rm Theorem {\rm \ref{existence}}},
$T>0$, $\v\in(0,\v_0]$, and $b \geq \max\{C_1(M,E_0,T,\varepsilon), \mathfrak{B}(\varepsilon)\}$
{\rm (}as defined in {\rm Lemmas \ref{expand}} and {\rm \ref{appendix5}}{\rm )},
    \begin{align*}
        \int^T_0 \int_K \rho^{\gamma + 1}(t,r) \dd r \dd t \leq C(K,M,E_0,T)
    \end{align*}
for any $K \Subset (a,b(t))$ and any $t \in [0,T]$.
\end{lemma}

\begin{proof}
We follow a similar proof to that of \cite{Chen2021}. Given $K \Subset (a,b(t))$ for any $t \in [0,T]$,
there exists $a < d < D$ such that $K \Subset (d,D) \Subset (a,b(t))$.
Let $w \in C^\infty_0(\mathbb{R})$ such that $\supp (w) \subset (d,D)$ and $w(r) = 1$ for all $r \in K$.
Next, we multiply $\eqref{Eul}_2$ by $w$ to obtain the equation:
    \begin{align}\label{5.1}
            & (\rho u w)_t + \big((\rho u^2 + p(\rho))w\big)_r + \frac{n-1}{r}\rho u^2w + \rho(\Phi_\alpha * \rho)_rw \nonumber\\
            & = \varepsilon\Big( \rho \big (u_r + \frac{n-1}{r}u \big ) w \Big)_r
            - \varepsilon\frac{n-1}{r}\rho_ru w
            + \Big( \rho u^2 + p(\rho) - \varepsilon \rho \big (u_r + \frac{n-1}{r}u \big ) \Big ) \Big)w_r.
    \end{align}
We integrate \eqref{5.1} over $[d,r)$ and then multiply by $\rho w$ to obtain
\begin{align}\label{5.2}
            \r p(\rho) w^2
            & = \v \r^2\big(u_r+\frac{n-1}{r} u\big) w^2 -\v\r w\int_d^r \frac{n-1}{z} u\rho_z w\dd z
              -\Big(\r w\int_d^r\r u w \dd z\Big)_t\nonumber\\
	    & \quad -\Big(\r  uw\int_d^r\r u w \dd z\Big)_r+\r uw_r \int_d^r \r u w \dd z
          -\frac{n-1}{r}  \r u w \int_d^r \r u w \dd z\nonumber\\
	    & \quad -\r w\int_d^r \frac{n-1}{z}\r u^2 w \dd z +\r w\int_d^r\Big(\r u^2+p(\rho)
        - \v\rho(u_z+\frac{n-1}{z}u)\Big)w_z \dd z\nonumber\\
	    & \quad - \kappa \rho w \int^r_d \rho (\Phi_\alpha * \rho)_z w \dd z =: \sum_{i=1}^9 I_i.
    \end{align}
As in \cite{Chen2021}, we have
    \begin{equation}\label{5.3}
        \bigg| \int^T_0 \int^D_d \Big( \sum_{i=1}^8 I_i \Big) \dd r \dd t \bigg| \leq C(d,M,E_0,T).
    \end{equation}
When
$\alpha \in (n-2,n-1)$, employing \eqref{continuity.8} and \eqref{continuity.5} with the Fubini theorem yields
\begin{align}
            & \left | \int^T_0 \int^D_d I_9 \dd r \dd t \right | \nonumber\\
            & \leq C \int^T_0 \int^D_d (\rho w)(t,r) \int^r_d (\rho w)(t,z) \bigg( \int_{(a,b(t)) \cap (z - \frac{d}{2},z + \frac{d}{2})} \omega(z,\eta) \rho(t,\eta) \eta^{n-1} \dd \eta \nonumber\\
            & \qquad\qquad\qquad\qquad\qquad\qquad\qquad\quad\quad\,\, + \int_{(a,b(t)) \cap (z - \frac{d}{2},z + \frac{d}{2})^c} \omega(z,\eta) \rho(t,\eta) \eta^{n-1} \dd \eta \bigg) \dd z \dd r \dd t \nonumber\\
            & \leq C \int^T_0 \int^D_d (\rho w)(t,r) \int^r_d (\rho w)(t,z)
            \bigg( \int_{(a,b(t)) \cap (z - \frac{d}{2},z + \frac{d}{2})} \frac{|z-\eta|^{n-2-\alpha}}{(r\eta)^{\frac{n-1}{2}}} \rho(t,\eta) \eta^{n-1} \dd \eta\nonumber\\
            & \qquad\qquad\qquad\qquad\qquad\qquad\qquad\quad\quad\,\,
              + \int_{(a,b(t)) \cap (z - \frac{d}{2},z + \frac{d}{2})^c} |z-\eta|^{-(\alpha + 1)} \rho(t,\eta) \eta^{n-1} \dd \eta \bigg) \dd z \dd r \dd t  \nonumber\\
            & \leq C(d) \int^T_0 \int^D_d (\rho w)(t,r) \int^r_d (\rho w)(t,z) \int_a^{b(t)}
              \big(1 + |z-\eta|^{n-2-\alpha}\big) \rho(t,\eta) \eta^{n-1} \dd \eta  \dd z \dd r \dd t.  \nonumber\\
            & \leq C(d,M,T)
            + C(d,D) \int^T_0 \int^D_d (\rho w)(t,r)
              \Big(\int_a^{b(t)}\int^D_d \rho(t,\eta) |z-\eta|^{n-2-\alpha} (\rho w)(t,z) \dd z \dd \eta\Big)\dd r \dd t.\label{5.4}
\end{align}
For $\gamma > \frac{1}{n-1 - \alpha}$, by Theorem \ref{pointwise}, we obtain
    \begin{align}\label{high}
            & \bigg| \int^T_0 \int^D_d I_9 \dd r \dd t \bigg| \nonumber\\[1mm]
            & \leq C(d,M,T) \nonumber\\
            & \quad + C(d,D) \int^T_0 \int^D_d (\rho w)(t,r) \Big(\int_a^{b(t)} \rho(t,\eta) \sup_{s \in \mathbb{R}_+}
            \Big\{ \int^D_d |s-\eta|^{n-2-\alpha} (\rho w)(t,z) \dd z \Big\} \dd \eta\Big) \dd r \dd t \nonumber\\[1mm]
            & = C(d,M,T)  \nonumber\\
            & \quad + C(d,D) \int^T_0 \int^D_d (\rho w)(t,r)\Big( \int_a^{b(t)} \rho(t,\eta)
              \sup_{s \in [d,D]} \Big\{ \int^D_d |s-\eta|^{n-2-\alpha} (\rho w)(t,z) \dd z \Big\} \dd \eta\Big) \dd r \dd t  \nonumber\\
            & \leq C(d,M,T) + C(d,D) \|\rho\|_{L^\gamma(d,D)} \int^T_0 \int^D_d (\rho w)(t,r)\Big( \int_a^{b(t)} \rho(t,\eta) \dd \eta\Big)
            \dd r \dd t.  \nonumber\\
            & \leq C(d,D,M,E_0,T).
    \end{align}
When
$\alpha \in (-1,n-2)$, by a similar approach to \eqref{5.4} and using \eqref{continuity.8}, we have
    \begin{align}\label{low}
        & \bigg| \int^T_0 \int^D_d I_9 \dd r \dd t \bigg|  \nonumber\\
        & \leq C \int^T_0 \int^D_d (\rho w)(t,r) \int^r_d (\rho w)(t,z) \bigg( \int_{(a,b(t)) \cap (z - \frac{d}{2},z + \frac{d}{2})} \omega(z,\eta) \rho(t,\eta) \eta^{n-1} \dd \eta \nonumber\\
        & \qquad\qquad\qquad\qquad\qquad\qquad\qquad\quad\quad\,\,
          + \int_{(a,b(t)) \cap (z - \frac{d}{2},z + \frac{d}{2})^c} \omega(z,\eta) \rho(t,\eta) \eta^{n-1} \dd \eta \bigg) \dd z \dd r \dd t \nonumber\\
        & \leq {\small C(d,M,T) + C \int^T_0 \int^D_d (\rho w)(t,r)
          \Big(\int^r_d (\rho w)(t,z) \int_{(a,b(t)) \cap (z - \frac{d}{2},z + \frac{d}{2})} (r\eta)^{-\frac{\alpha + 1}{2}} \rho(t,\eta) \eta^{n-1} \dd \eta \dd z\Big) \dd r \dd t } \nonumber\\
        & \leq C(d,M,T) + C(d) \int^T_0 \int^D_d (\rho w)(t,r)
         \Big( \int^r_d (\rho w)(t,z) \int_a^{b(t)} \rho(t,\eta) \eta^{n-1} \dd \eta  \dd z\Big) \dd r \dd t \nonumber\\
        & \leq C(d,M,T).
    \end{align}

\medskip
When $\alpha = n - 2$, using \eqref{continuity.8}, we obtain
    \begin{align}\label{mid}
         \bigg| \int^T_0 \int^D_d I_9 \dd r \dd t \bigg|
        & = \int^T_0 \int^D_d (\rho w)(t,r)
          \Big(\int^r_d (\rho w)(t,z) \int_a^{r} \frac{\omega_n}{r^{n-1}} \rho(t,\eta) \eta^{n-1} \dd \eta  \dd z\Big)
          \dd r \dd t \nonumber\\[2mm]
        & \leq C(d,M,T).
    \end{align}
Therefore, integrating \eqref{5.2} over $[0,T] \times [d,D]$ and using \eqref{5.3}--\eqref{mid},
we conclude the desired result.
\end{proof}

\smallskip
For the next result, we need to make use of entropy pairs.
As defined in Lax \cite{Lax}, a pair of functions $(\eta(\rho,m),q(\rho,m))$
is called an entropy pair
of the $1$-D isentropic Euler system if they satisfy
\begin{equation*}
    \partial_t \eta(\rho(t,r),m(t,r)) + \partial_r q(\rho(t,r),m(t,r)) = 0
\end{equation*}
for any smooth solution $(\rho,m)$ of the Euler system. Moreover, $\eta$ is called a weak entropy if
\begin{equation*}
    \eta|_{\rho = 0}=0 \qquad \text{ for all fixed $u = \frac{m}{\rho}$}.
\end{equation*}

From \cite{Lions_1994}, it has been shown that any weak entropy pair $(\eta,q)$
can be represented by
\begin{align}\label{5.10}
	\begin{cases}
            \displaystyle
		\eta^\psi(\rho,m) = \eta^\psi(\rho,\rho u) = \int_\mathbb{R} \chi(\rho;s-u) \psi(s) \dd s,\\[3mm]
            \displaystyle
		  q^\psi(\r,m) = q^\psi(\r, \r u) = \int_\mathbb{R}(\theta s + (1 - \theta)u) \chi(\rho;s-u) \psi(s) \dd s,
	\end{cases}
\end{align}
where $\chi(\rho;s-u) = (\rho^{2\theta} - (s-u)^2)_+^{\mathfrak{b}}$
with $\mathfrak{b} = \frac{3 - \gamma}{2(\gamma-1)}$, $\theta = \frac{\gamma - 1}{2}$,
and $(\cdot)_+ = \max \{\cdot\,, 0\}$.
Choosing $\psi(s) = \frac{1}{2}s|s|$ as in \cite{Chen2021},
we obtain the corresponding entropy pair $(\eta^{\#}(\rho, m), q^{\#}(\rho,m))$ that is
represented by
\begin{align}\label{1.32}
	\begin{cases}
            \displaystyle
		\eta^{\#}(\rho,m) = \eta^{\#}(\rho,\rho u)
        = \frac{1}{2} \rho \int_{-1}^1 (u+\rho^{\theta} s) |u+\r^{\theta}s| (1-s^2)_+^{\mathfrak{b}} \dd s,\\[3mm]
            \displaystyle
		q^{\#}(\r,m) = q^{\#}(\r, \r u)
        = \frac{1}{2} \r \int_{-1}^1 (u+\theta\r^{\theta}s)                (u+\r^{\theta} s)|u+\r^{\theta}s| (1-s^2)_+^{\mathfrak{b}} \dd s.
	\end{cases}
\end{align}
Using \eqref{1.32}, it can be shown that there exists $C_\gamma > 0$ such that the following inequalities hold:
\begin{equation}\label{1.33}
    \begin{split}
    & |\eta^{\#}(\rho,\rho u)| \leq C_\gamma(\rho|u|^2 + \rho^\gamma), \quad q^\#(\rho, \rho u) \geq C_\gamma^{-1}(\rho|u|^3 + \rho^{\gamma + \theta}), \\
    & |\eta^{\#}_m(\rho,\rho u)| \leq C_\gamma(|u| + \rho^\theta),
    \quad\quad |\eta^{\#}_\rho(\rho,\rho u)| \leq C_\gamma(|u|^2 + \rho^{2\theta}).
    \end{split}
\end{equation}
As in \cite{Chen_2010,Chen_2015}, we have
\begin{align}\label{5.11}
	\r u \partial_\r \eta^{\#}+\r u^2 \partial_m\eta^{\#}-q^{\#}
	&=\frac{\theta}{2} \rho^{1+\theta}\int_{-1}^1  (u-\r^{\theta}s)s |u+\r^{\theta}s|\, (1-s^2)_+^{\mathfrak{b}} \dd s\leq 0.
\end{align}

\begin{lemma}[Higher Integrability of the Velocity]
Let $(\r,u)$ be the smooth solution of \eqref{Eul}. Then,
under the assumption of {\rm Lemma {\rm \ref{higher}}},
	\begin{align*}
		\int_0^T\int_d^D \big(\rho |u|^3 + \r^{\gamma+\theta}\big)(t,r)\,\dd r \dd t\leq C(d,D,M,E_0,T)
	\end{align*}
for any $(d,D)\Subset[a,b(t)]$ for any $t \in[0,T]$.
\end{lemma}

\begin{proof}
Firstly, multiplying $\eqref{Eul}_1$ by $r^{n-1} \eta^\#_\rho$ and $\eqref{Eul}_2$ by $r^{n-1} \eta^\#_m$, we have
\begin{align}
& (\eta^\# r^{n-1})_t + (q^\# r^{n-1})_r + (n-1)\big(-q^\# + \rho u \eta^\#_\rho + \rho u^2 \eta_m^\#\big) r^{n-2}\nonumber\\
& = \v \eta^\#_m \big( (\rho u)_r + (n-1)\rho ( \frac{u}{r} )_r\big) r^{n-1}
- \kappa \eta^\#_m \rho (\Phi_\alpha * \rho)_r r^{n-1}.\label{5.12}
\end{align}
Integrating \eqref{5.12} over $[r,b(t))$ and using \eqref{0.19} and \eqref{5.11}, we obtain
    \begin{align}
	     q^{\#}(t,r)r^{n-1} & \leq -\v \int_r^{b(t)}  \eta^{\#}_m(t,z)(\rho u_z)_z\, z^{n-1}\dd z
	   -(n-1)\v\int_r^{b(t)}\eta^{\#}_m(t,z) \left (\frac{u}{z} \right )_z \rho \, z^{n-1}\dd  z\nonumber\\
	   &\quad\, +\Big( \int_r^{b(t)}\eta^{\#}(t,z)\,z^{n-1}\dd z\Big)_t
	   +\big(q^{\#}-u\eta^{\#}\big)(t,b(t))b(t)^{n-1}\nonumber\\
	   &\quad\, +\kappa \int_r^{b(t)} \eta^\#_m(t,z) \rho (\Phi_\alpha * \rho)_z z^{n-1} \dd z =: \sum^5_{i=1} I_i.\label{5.13}
    \end{align}
Now, as in \cite[Lemma 3.7]{Chen2021}, we have
\begin{equation*}
\bigg| \int^T_0 \int^D_d \Big( \sum^4_{i = 1} I_i \Big) \dd r \dd t \bigg| \leq C(d,D,M,E_0,T).
\end{equation*}
To estimate the potential term $I_5$, we use \eqref{1.33}, the Young inequality, and once again
the estimates in \eqref{continuity.8}, and apply the same method as for \eqref{5.4} and \eqref{high}--\eqref{low}
to obtain
    \begin{align}
    & \bigg| \int^T_0 \int^D_d \int^{b(t)}_r \kappa \eta^{\#}_m \rho (\Phi_\alpha * \rho)_z z^{n-1} \dd z \dd r\dd t\bigg|\nonumber \\
    & \leq C(d,D,M,E_0,T) \int^T_0 \int^D_d \int^{b(t)}_r
      \big((|u| + \rho^\theta)\rho\big)(t,z) z^{n-1} \dd z \dd r \dd t\nonumber\\
   & \leq C(d,D,M,E_0,T) \int^T_0 \int^D_d \int^{b(t)}_r \big(\rho|u| + \rho^\frac{\gamma}{2}\rho^\frac{1}{2}\big)(t,z)
        z^{n-1} \dd z \dd r \dd t \nonumber\\
  & \leq C(d,D,M,E_0,T) \int^T_0 \int^D_d \int^{b(t)}_r\big(\rho u^2 + \rho + \rho^\gamma\big)(t,z) z^{n-1} \dd z \dd r \dd t\nonumber\\[1mm]
  & \leq C(d,D,M,E_0,T).\label{5.15}
    \end{align}
Combining $\eqref{1.32}_2$ with \eqref{5.13}--\eqref{5.15}, we obtain the desired result.
\end{proof}

\medskip
\section[Global Solutions of the Compressible Navier-Stokes-Riesz Equations]{Global Solutions to the Compressible Navier-Stokes-Riesz Equations}\label{S4}

We now show the convergence of our solutions $(\rho^{\v,b},\rho^{\v,b}u^{\v,b})$ of the free boundary
problem \eqref{Eul}--\eqref{0.20} to global weak solutions $(\rho^\v,m^\v)$ of CNSREs \eqref{0.2} as $b \to \infty$.
When we take the limit, the cavitation (vacuum states) may form, which should be taken into account carefully.
For the following, we extend $(\rho^{\v,b},u^{\v,b})$ to be zero on $([0,T] \times (0,\infty)) \setminus \Omega^T$.

\begin{lemma}\label{lem5.2}
Under the assumptions of {\rm Theorem {\rm \ref{existence}}}, for fixed $\v\in(0,\v_0]$
{\rm (}as defined in {\rm Lemma \ref{appendix5}}{\rm )}, there exists a non-negative function $\rho^{\v}(t,r)$ a.e. such that, as $b\rightarrow\infty$ $($up to a sub-sequence$)$,
	\begin{equation} \label{2.66}
		(\sqrt{\rho^{\v,b}}, \rho^{\v,b})\longrightarrow (\sqrt{\rho^{\v}}, \rho^{\v})  \quad  \mbox{ {\it a.e.} and strongly in} \  C(0,T; L^q_{\rm loc})
	\end{equation}
	for any $q\in [1,\infty)$, where $L^q_{\rm loc}$ denotes  $L^q(K)$ for $K\Subset (0,\infty)$.
\end{lemma}

The proof is the same as for \cite[Lemma 4.1]{Chen2021}.

\begin{corollary}
	Under the assumptions of {\rm Lemma {\rm \ref{lem5.2}}}, the pressure function sequence $p(\rho^{\v,b})$ is uniformly bounded in $L^\infty(0,T; L^q_{\rm loc})$ for all $q\in[1,\infty]$
	and, as $b\rightarrow\infty$ $($up to a sub-sequence$)$,
	\begin{align*}
		p(\rho^{\v,b})\longrightarrow p(\rho^{\v}) \qquad  \mbox{strongly in $L^q(0,T; L^q_{\rm loc})$ for all $q\in[1,\infty)$}.
	\end{align*}
\end{corollary}

\begin{lemma}\label{lem5.5}
	Under the assumptions of {\rm Lemma {\rm \ref{lem5.2}}}, as $b\rightarrow\infty$ $($up to a sub-sequence$)$,
	the momentum function sequence $m^{\v,b}:=\rho^{\v,b} u^{\v,b}$ converges strongly in $L^2(0,T; L^q_{\rm loc})$ to
	some function $m^{\v}(t,r)$ for all $q\in[1,\infty)$, which particularly implies that
	\begin{align*}
		m^{\v,b}=\rho^{\v,b} u^{\v,b}\longrightarrow m^{\v}(t,r) \qquad  \mbox{a.e. in $[0,T]\times(0,\infty)$}.
	\end{align*}
\end{lemma}

\begin{proof}
From Lemmas \ref{Energy}, \ref{oldBD}, and \ref{newbd}, and the Sobolev embedding theorem,
we see that $\sqrt{\rho^{\varepsilon,b}}$ is uniformly bounded
in $L^\infty(0,T; H^{1}_{\rm loc})\hookrightarrow L^\infty(0,T; L^{\infty}_{\rm loc})$ for $b>0$,
so that
\[
\rho^{\varepsilon,b}u^{\varepsilon,b}
= \sqrt{\rho^{\varepsilon,b}}\big(\sqrt{\rho^{\varepsilon,b}}u^{\varepsilon,b} \big)
\]
is uniformly bounded in $L^\infty(0,T; L^2_{\rm loc})$ for $b>0$.
Since
\begin{align*}
\big(\rho^{\varepsilon,b}u^{\varepsilon,b} \big)_r
= \rho^{\varepsilon,b}_ru^{\varepsilon,b} + \rho^{\varepsilon,b}u^{\varepsilon,b}_r
= 2 \big(\sqrt{\rho^{\varepsilon,b}}\big)_r \big( \sqrt{\rho^{\varepsilon,b}}u^{\varepsilon,b} \big)
+ \sqrt{\rho^{\varepsilon,b}} \big( \sqrt{\rho^{\varepsilon,b}}u^{\varepsilon,b}_r \big),
\end{align*}
it follows from Lemmas \ref{Energy}, \ref{oldBD}, and \ref{newbd}
that $(\rho^{\varepsilon,b}u^{\varepsilon,b})_r$ is uniformly
bounded in $L^2(0,T; L^1_{\rm loc})$ for $b > 0$. This implies
that
\begin{align}\label{4.2}
\rho^{\varepsilon,b}u^{\varepsilon,b} \in L^2(0,T; W^{1,1}_{\rm loc})
 \qquad\mbox{uniformly for $b > 0$}.
\end{align}
From Lemma \ref{Energy} and the HLS inequality from Lemma \ref{simplehls}, we have
    \begin{align}\label{4.3}
	\begin{cases}
		\displaystyle \partial_r\big((\sqrt{\rho^{\v,b}} u^{\v,b})^2\big)\in L^{\infty}(0,T;W_{\rm loc}^{-1,1}),\\[2mm]
		\displaystyle \frac{n-1}{r} \big(\sqrt{\rho^{\v,b}} u^{\v,b}\big)^2 \in L^{\infty}(0,T; L_{\rm loc}^{1}),\\[2mm]
		\displaystyle \partial_r p(\rho^{\v,b}) \in L^{2}(0,T;H_{\rm loc}^{-1}),\\[2mm]
	\displaystyle \rho^{\varepsilon,b} \big(\Phi_\alpha * \rho^{\varepsilon,b}\big)_r \in L^\infty(0,T; L^\infty_{\rm loc}),
	\end{cases}
\end{align}
and
\begin{align}\label{4.4}
    \sqrt{\rho^{\v,b}}\,\big(\sqrt{\rho^{\v,b}} u^{\v,b}_r\big)
	+\frac{n-1}{r} \sqrt{\rho^{\v,b}}\,\big(\sqrt{\rho^{\v,b}}u^{\v,b}\big) \in L^2(0,T; L^2_{\rm loc})
\end{align}
uniformly for $b>0$.
Hence, it follows from \eqref{4.4} that
\begin{align}\label{4.5}
    \partial_r \big( \rho^{\varepsilon,b} (u^{\varepsilon,b}_r + \frac{n-1}{r}u^{\varepsilon,b})\big)
    \in L^2(0,T; H^{-1}_{\rm loc})
\end{align}
uniformly for $b>0$.
Using Lemmas \ref{Energy}, \ref{oldBD}, and \ref{newbd}, we have
\begin{align}\label{4.6}
    \frac{n-1}{r} u^{\v,b} \partial_r \rho^{\v,b} =\frac{2(n-1)}{r} (\sqrt{\rho^{\v,b}})_r\,\big(\sqrt{\rho^{\v,b}}u^{\v,b}\big)\in L^2(0,T; L^1_{\rm loc})
\end{align}
uniformly for $b > 0$.
Then, substituting \eqref{4.3} and \eqref{4.5}--\eqref{4.6} into $\eqref{Eul}_2$
and using the Sobolev embedding theorem, we see that
\begin{align*}
    \partial_t(\rho^{\v,b} u^{\v,b})\in L^{2}(0,T; W^{-1,q}_{\rm loc}),
\end{align*}
uniformly for $b > 0$ and $q \in [1,\frac{n}{n-1})$.
Using the Aubin-Lions lemma together with \eqref{4.2}, we obtain that
$\rho^{\v,b} u^{\v,b}$ is compact in $L^{2}(0,T;L^q_{\rm loc})$ for all $q \in [1,\infty)$.
\end{proof}

\begin{lemma}\label{lem5.6}
	Under the assumptions of {\rm Lemma {\rm \ref{lem5.2}}}, the limit function $m^\v(t,r)$ in {\rm Lemma \ref{lem5.5}} satisfies that
	$m^{\v}(t,r)=0$ a.e. on $\{(t,r)\,:\,\rho^{\v}(t,r)=0\}$. Furthermore, there exists
	a function $u^{\v}(t,r)$ such that $m^{\v}(t,r)=\rho^{\v}(t,r) u^{\v}(t,r)$ a.e., and $\,u^\v(t,r)=0$ a.e.
	on $\{(t,r)\,:\,\rho^{\v}(t,r)=0\}$. Moreover, as $b\rightarrow\infty$ $($up to a sub-sequence$)$,
	\begin{align*}
		& m^{\v,b} \longrightarrow m^{\v}=\rho^{\v} u^{\v} \qquad  \mbox{strongly in $L^2(0,T; L^q_{\rm loc})$ for $q\in[1,\infty)$},\\
		& \frac{m^{\v,b}}{\sqrt{\rho^{\v,b}}}\longrightarrow \sqrt{\rho^{\v}} u^{\v}=:\frac{m^\v}{\sqrt{\rho^\v}}
		\qquad  \mbox{strongly in $L^2(0,T; L^2_{\rm loc})$}.
	\end{align*}
\end{lemma}

The proof of Lemma \ref {lem5.6} is the same as \cite[Lemma 4.4]{Chen2021}.

\begin{lemma}\label{mass}
Under the assumptions of {\rm Lemma {\rm \ref{lem5.2}}}, let $0\leq t_1<t_2\leq T$, and let $\zeta(t,\boldsymbol{x})\in C^1_{\rm c}([0,T]\times\mathbb{R}^n)$
be any smooth function with compact support.
Then
\begin{align}\label{3.1}
	\int_{\mathbb{R}^n} \rho^{\v}(t_2,\boldsymbol{x}) \zeta(t_2,\boldsymbol{x})\, \dd\boldsymbol{x}
	=\int_{\mathbb{R}^n} \rho^{\v}(t_1,\boldsymbol{x}) \zeta(t_1,\boldsymbol{x})\, \dd\boldsymbol{x}
	+\int_{t_1}^{t_2} \int_{\mathbb{R}^n} \big(\rho^{\v} \zeta_t + \M^{\v}\cdot\nabla\zeta\big)\,\dd\boldsymbol{x}\dd t.
\end{align}
Moreover, the total mass is conserved:
\begin{equation}\label{3.2}
	\int_{\mathbb{R}^n} \rho^\v(t,\boldsymbol{x})\, \dd\boldsymbol{x}=\int_{\mathbb{R}^n} \rho_0^\v(\boldsymbol{x})\, \dd\boldsymbol{x}=M\qquad\, \mbox{for $t\geq0$}.
\end{equation}
\end{lemma}

The proof of Lemma \ref {mass} is  the same as  \cite[Lemma 4.9]{Chen2021}.

We can now prove an {\it a priori} estimate of the energy terms.

\begin{lemma}\label{simple}
    Under the assumptions of {\rm Lemma {\rm \ref{lem5.2}}}, the following estimates hold{\rm :}
\begin{enumerate}
\item[\rm (i)] When $\alpha \in (0,n - 1)$,
    \begin{align*}
        \begin{split}
            & \int_0^\infty  \Big(\frac12 \Big|\frac{m^{\v}}{\sqrt{\rho^\varepsilon}} \Big|^2
            + \rho^{\v}e(\rho^\v) \Big)(t,r)\,r^{n-1}\dd r
+ \varepsilon (n-1) \int^t_0\int^\infty_0 \Big|\frac{m^{\v}}{\sqrt{\rho^\varepsilon}}\Big|^2 r^{n-3} \dd r\dd s \\[1mm]
            & \leq C(M,E_0).
        \end{split}
    \end{align*}

\item[\rm (ii)] When $\alpha \in (-1,0]$,
    \begin{align*}
        \begin{split}
            & \int_0^\infty \Big(\frac12 \Big|\frac{m^{\v}}{\sqrt{\rho^\varepsilon}} \Big|^2
            + \rho^{\v}\big( e(\rho^\v) + k_\alpha(1 + r^2)\big) \Big)(t,r)\,r^{n-1}\dd r + \varepsilon (n-1) \int^t_0\int^\infty_0 \Big|\frac{m^{\v}}{\sqrt{\rho^\varepsilon}} \Big|^2 r^{n-3} \dd r\dd s \\[1mm]
            & \leq C(M,E_0,T).
            \end{split}
    \end{align*}
\end{enumerate}
\end{lemma}

\begin{proof}
It follows directly from Lemma \ref{Energy} that, for $\alpha \in (0,n - 1)$,
\begin{align*}
        \begin{split}
            & \int_a^{b(t)}  \Big(\frac12 \Big|\frac{m^{\v,b}}{\sqrt{\rho^{\v,b}}} \Big|^2
            + \rho^{\v,b}e(\rho^{\v,b}) \Big)(t,r)\,r^{n-1}\dd r
            + \varepsilon (n-1) \int^t_0\int_a^{b(s)} \Big|\frac{m^{\v,b}}{\sqrt{\rho^{\v,b}}} \Big|^2 r^{n-3} \dd r\dd s \\[1mm]
            & \leq C(n,M,E_0,T).
        \end{split}
    \end{align*}
Now, applying the Fatou lemma, and Lemmas \ref{expand}, \ref{lem5.2}, and \ref{lem5.6},
we obtain the result when $\alpha \in (0,n - 1)$. The approach is the same when $\alpha \in (-1,0]$.
\end{proof}

Notice that our estimates do not involve the potential term;
this is because the convergence of the potential has not yet been proved.
To do this, we require a few more steps.
Since the mass is conserved in the limit, and the local convergence of the density has been verified,
there is no blow-up of mass at the origin or a loss of mass at infinity; this allows us to prove some global convergence results.

\begin{lemma}\label{Lq}
Under the assumptions of {\rm Lemma {\rm \ref{lem5.2}}}, as $b\rightarrow\infty$ $($up to a sub-sequence$)$,
	\begin{equation*}
		\rho^{\v,b} \longrightarrow \rho^{\v}  \quad
        \text{ strongly in $C(0,T; L^q(\mathbb{R}^n))$ for any $q \in [1,\gamma)$}.
	\end{equation*}
That is, the convergence of $\rho^{\v,b}$ is global sub-sequentially in $L^q$.
\end{lemma}

\begin{proof}
For all $\alpha \in (-1,n-1)$, it follows from Lemma \ref{lem5.2} and \eqref{3.2} that
 \begin{align}\label{truth}
           \lim_{k \to \infty} \inf_{t \in [0,T]} \int_\frac{1}{k}^k \rho^\v(t,r) r^{n-1} \dd r = \frac{M}{\omega_n},
 \end{align}
as if not, there exists $\varepsilon > 0$ and $t_k \in [0,T]$ such that
\begin{align}\label{contra}
    \int_\frac{1}{k}^k \rho^\v(t_k,r) r^{n-1} \dd r < \frac{M}{\omega_n} - \varepsilon
    \qquad \mbox{for all $k \in \mathbb{N}$}.
\end{align}
Then there exists a subsequence of $(t_k)_k$ given by $(t_{k_l})_l$
and  $t_* \in [0,T]$ such that $t_{k_l} \to t_*$ as $l \to \infty$.
From \eqref{3.2}, we see that, by the monotone convergence theorem, there exists $L \in \mathbb{N}$ such that
\begin{align*}
    \int_\frac{1}{L}^L \rho^\v(t_*,r) r^{n-1} \dd r > \frac{M}{\omega_n} - \varepsilon.
\end{align*}
Note, for $k_l \geq L$, using \eqref{contra}, we have
\begin{align}\label{contraband}
    \int_\frac{1}{L}^L \rho^\v(t_{k_l},r) r^{n-1} \dd r \leq \int_\frac{1}{k_l}^{k_l} \rho^\v(t_{k_l},r) r^{n-1} \dd r < \frac{M}{\omega_n} - \varepsilon.
\end{align}
By Lemma \ref{lem5.2}, $\rho^\v \in C(0,T; L^1_{\rm loc})$ so that
\begin{align}\label{contaconv}
    \int_\frac{1}{L}^L \rho^\v(t_{k_l},r) r^{n-1} \dd r \longrightarrow \int_\frac{1}{L}^L \rho^\v(t_*,r) r^{n-1} \dd r
    \qquad \mbox{as $k\to\infty$}.
\end{align}
Thus, combining \eqref{contraband} with \eqref{contaconv}, we obtain
\begin{align*}
    \int_\frac{1}{L}^L \rho^\v(t_*,r) r^{n-1} \dd r \leq \frac{M}{\omega_n} - \varepsilon,
\end{align*}
which is a contradiction. Thus, \eqref{truth} is verified.

With this, there exists $k = k(\delta) \in \mathbb{N}$ such that
    \begin{align}\label{2.98}
           \inf_{t \in [0,T]} \int_\frac{1}{k}^k \rho^\v(t,r) r^{n-1} \dd r > \frac{M}{\omega_n} - \frac{\delta}{4},
    \end{align}
which, by \eqref{3.2} and \eqref{2.98}, implies
\begin{align}\label{2.99}
\sup_{t \in [0,T]} \Big(\int_{k}^\infty + \int_0^\frac{1}{k} \Big) \rho^\v(t,r) r^{n-1} \dd r
< \frac{\delta}{4}.
\end{align}
By \eqref{2.66}, there exists $B(\delta) > 0$ such that, for $b \geq B(\delta)$ (up to a sub-sequence),
\begin{align}\label{2.100}
        \begin{split}
           \sup_{t \in [0,T]} \int_\frac{1}{k}^k |\rho^\v - \rho^{\v,b}|(t,r) r^{n-1} \dd r < \frac{\delta}{4}.
        \end{split}
\end{align}
Then it follows from \eqref{2.98}--\eqref{2.100} that, for $b \geq B(\delta)$,
 \begin{align}\label{2.101}
        \begin{split}
           & \inf_{t \in [0,T]} \int_\frac{1}{k}^k \rho^{\v,b}(t,r) r^{n-1} \dd r > \frac{M}{\omega_n} - \frac{\delta}{2}, \\
           & \sup_{t \in [0,T]} \Big(\int_{k}^\infty + \int_0^\frac{1}{k} \Big)
           \rho^{\v,b}(t,r) r^{n-1} \dd r < \frac{\delta}{2}.
        \end{split}
\end{align}
Thus, we obtain that, for $b>B(\delta)$,
\begin{equation*}
\sup_{t \in [0,T]} \int_0^\infty |\rho^\v - \rho^{\v,b}|(t,r) r^{n-1} \dd r < \delta.
\end{equation*}
Now, for $q \in [1,\gamma)$, using Lemmas \ref{Energy} and \ref{simple},
and performing interpolation, we have
    \begin{align*}
        \begin{split}
        & \| \rho^\v - \rho^{\v,b} \|_{C(0,T;L^q(\mathbb{R}^n))} \\
        & \leq \| \rho^\v - \rho^{\v,b} \|^\frac{\gamma - q}{q(\gamma - 1)}_{C(0,T;L^1(\mathbb{R}^n))} \| \rho^\v - \rho^{\v,b} \|^\frac{\gamma(q - 1)}{q(\gamma - 1)}_{L^\infty(0,T;L^\gamma(\mathbb{R}^n))} \\
        & \leq C(n,M,E_0,T,q) \| \rho^\v - \rho^{\v,b} \|^\frac{\gamma - q}{q(\gamma - 1)}_{C(0,T;L^1(\mathbb{R}^n))} \longrightarrow 0,
        \end{split}
    \end{align*}
    as $b \to \infty$ (up to a sub-sequence).
\end{proof}

\begin{lemma}\label{lem5.1}
Under the assumptions of {\rm Lemma {\rm \ref{lem5.2}}}, the following statements hold
as $b\rightarrow\infty$ {\rm(}up to a sub-sequence{\rm)}{\rm :}
\begin{align}
&\Phi_\alpha*\rho^{\v,b} \longrightharpoonup \Phi_\alpha*\rho^\v \quad
	\mbox{weak-$\ast$ in $L^\infty(0,T; W^{1,\frac{2n}{\alpha +2}}_{\rm loc}(\mathbb{R}^n))$ and weakly in $L^2(0,T; W^{1,\frac{2n}{\alpha +2}}_{\rm loc}(\mathbb{R}^n))$},\label{2.67}\\
	&(\Phi_\alpha*\rho^{\v,b})_r(t,r)r^{n-1} \longrightarrow (\Phi_\alpha*\rho^\v)_r(t,r)r^{n-1} \quad
	\mbox{ in $C_{\rm loc}([0,T]\times[0,\infty))$}, \label{2.68} \\[4mm]
    &(\Phi_\alpha*\rho^{\v,b})(t,r)r^{n-1} \longrightarrow (\Phi_\alpha*\rho^\v)(t,r)r^{n-1} \quad
	\mbox{ in $C_{\rm loc}([0,T]\times[0,\infty))$},\label{2.69} \\[3mm]
    & \int_0^\infty \big|\rho^{\v,b} (\Phi_\alpha*\rho^{\v,b}) - \rho^{\v} (\Phi_\alpha*\rho^{\v})\big|(t,r)\,
    r^{n-1} \dd r \longrightarrow 0. \label{2.71}
\end{align}
Moreover, when $\alpha \in (0,n-1)$,
\begin{align}\label{2.70}
	\|(\Phi_\alpha*\rho^\v)(t)\|_{L^{\frac{2n}{\alpha}}(\mathbb{R}^n)}
    +\|(\nabla \Phi_\alpha*\rho^\v)(t)\|_{L^\frac{2n}{\alpha + 2}(\mathbb{R}^n)}\leq C(M,E_0)\qquad\,\mbox{for $t\geq0$};
\end{align}
when $\alpha \in (-1,0]$, for any $K \Subset \mathbb{R}^n$,
\begin{align}\label{2.70-2}
	\|(\Phi_\alpha*\rho^\v)(t)\|_{L^\infty(K)}
    +\|(\nabla \Phi_\alpha*\rho^\v)(t)\|_{L^\frac{2n}{\alpha + 2}(K)}
    \leq C(M,E_0,T,K)\qquad\,\mbox{for $t\geq0$}.
\end{align}
\end{lemma}

\begin{proof}  We divide the proof into five steps.

\medskip
 1. We first see that \eqref{2.67} and \eqref{2.70}--\eqref{2.70-2} follow
from Corollary \ref{uniform} and the weak-*/weak compactness directly.

\smallskip
2. To show that the weak limit of $\Phi_\alpha * \rho^{\v,b}$, denoted by $\Phi_\alpha^\varepsilon$,
is indeed $\Phi_\alpha * \rho^\v$, we test by a smooth compactly supported
function $\varphi \in C^\infty_{\rm c}([0,T] \times \mathbb{R}^n)$ so that
it can be shown by shifting the convolution from $\rho^{\v,b}$ to $\varphi$ that, as $b \to \infty$,
    \[
    \int_{\mathbb{R}^n} \varphi (\Phi_\alpha * \rho^{\varepsilon,b}) \dd \boldsymbol{x} = \int_{\mathbb{R}^n} \rho^{\varepsilon,b} (\Phi_\alpha * \varphi) \dd \boldsymbol{x} \longrightarrow \int_{\mathbb{R}^n} \rho^{\varepsilon} (\Phi_\alpha * \varphi) \dd \boldsymbol{x} = \int_{\mathbb{R}^n} \varphi (\Phi_\alpha * \rho^{\varepsilon}) \dd \boldsymbol{x}.
    \]
 Using \eqref{2.67}, we have
    \[
    \int_{\mathbb{R}^n} \varphi (\Phi_\alpha * \rho^{\varepsilon,b}) \dd \boldsymbol{x} \longrightarrow \int_{\mathbb{R}^n} \varphi\,\Phi_\alpha^\varepsilon \dd \boldsymbol{x}.
    \]
Then, by the fundamental theorem of calculus of variations, $\Phi_\alpha^\varepsilon = \Phi_\alpha * \rho^{\varepsilon}$
{\it a.e.}. The result follows.

\smallskip
3. We now prove \eqref{2.68}.
When $\alpha \in (-1,n-1)$,
for $D > 0$, $(t,r) \in [0,T] \times [0,D]$, and $b$ sufficiently large,
and taking
$0 < \sigma < {\hat{\sigma}}<\infty$,
we have
    \begin{align}\label{2.72}
            &\big|(\Phi_\alpha*\rho^{\v,b})_r r^{n-1} - (\Phi_\alpha*\rho^\v)_r r^{n-1}\big|(t,r) \nonumber\\
            &= r^{n-1} \Big| \int^\infty_0 \omega(r,\eta) (\rho^{\v,b} - \rho^\v)(t,\eta) \eta^{n-1} \dd\eta \Big| \nonumber\\
            & \leq r^{n-1} \Big| \int^\sigma_0 \omega(r,\eta) (\rho^{\v,b} - \rho^\v)(t,\eta) \eta^{n-1} \dd \eta \Big|
            + r^{n-1} \Big| \int^{\hat{\sigma}}_\sigma \omega(r,\eta) (\rho^{\v,b} - \rho^\v)(t,\eta) \eta^{n-1} \dd \eta \Big| \nonumber\\
            & \quad + r^{n-1} \Big| \int^\infty_{\hat{\sigma}} \omega(r,\eta) (\rho^{\v,b} - \rho^\v)(t,\eta) \eta^{n-1} \dd \eta \Big|\nonumber\\
            & =: I_1 + I_2 + I_3.
    \end{align}
For $I_1$, applying the H\"{o}lder inequality, we have
    \begin{align}\label{2.73}
            & I_1 \leq r^{n-1} \Big(\int^\sigma_0 \big( (\rho^{\v,b})^\gamma + (\rho^b)^\gamma\big)\eta^{n-1} \dd \eta \Big)
            \Big| \int^\sigma_0 \omega^{\frac{\gamma}{\gamma - 1}}(r,\eta) \eta^{n-1}
            \dd \eta \Big|^{1 - \frac{1}{\gamma}} \nonumber\\
            & \quad \leq C(M,E_0) r^{n-1}
            \Big| \int^\sigma_0 \omega^{\frac{\gamma}{\gamma - 1}}(r,\eta) \eta^{n-1}
            \dd \eta \Big|^{1 - \frac{1}{\gamma}}.
    \end{align}
When $\alpha \in (n-2,n-1)$, using \eqref{continuity.8}, we have
\begin{align}\label{2.74}
            & r^{n-1} \bigg| \int_0^\sigma\omega^{\frac{\gamma}{\gamma - 1}}(r,\eta) \eta^{n-1}
            \dd \eta \bigg|^{1 - \frac{1}{\gamma}} \nonumber\\
            & \leq C r^{n-1} \sigma^{(n-1)(1 - \frac{1}{\gamma})}
            \bigg( \int_{(0,\sigma)\cap(\frac{r}{2},\infty)}\omega^{\frac{\gamma}{\gamma - 1}}(r,\eta) \dd \eta + \int_{(0,\sigma)\cap(0,\frac{r}{2})}\omega^{\frac{\gamma}{\gamma - 1}}(r,\eta) \dd \eta
            \bigg)^{1 - \frac{1}{\gamma}} \nonumber\\
            & \leq C r^{n-1} \sigma^{(n-1)(1 - \frac{1}{\gamma})} \bigg ( \int_{(0,\sigma)\cap(\frac{r}{2},\infty)}\Big( \frac{|r-\eta|^{n-2 - \alpha}}{(r\eta)^{\frac{n-1}{2}}} \Big)^{\frac{\gamma}{\gamma - 1}} \dd \eta
            + \int_{(0,\sigma)\cap(0,\frac{r}{2})} |r-\eta|^{-(\alpha +1)\frac{\gamma}{\gamma - 1}}
            \dd \eta \bigg )^{1 - \frac{1}{\gamma}}\nonumber\\
            & \leq C r^{n-1} \sigma^{(n-1)(1 - \frac{1}{\gamma})}
            \bigg( r^{-(n-1)\frac{\gamma}{\gamma - 1}} \int_{(0,\sigma)\cap(\frac{r}{2},\infty)}|r-\eta|^{(n-2- \alpha)\frac{\gamma}{\gamma - 1}} \dd \eta + r^{1 -(\alpha +1)\frac{\gamma}{\gamma - 1}} \bigg)^{1 - \frac{1}{\gamma}}\nonumber\\
            & \leq C \sigma^{(n - 1)(1 - \frac{1}{\gamma })} + C r^{n - 1 - \alpha - \frac{1}{\gamma}} \sigma^{(n-1)(1 - \frac{1}{\gamma})} \nonumber\\
            & \leq C(D) \sigma^{(n-1)(1 - \frac{1}{\gamma})},
\end{align}
 where the last inequality follows from the fact that $n - 1 - \alpha - \frac{1}{\gamma} > 0$
 due to the condition for \eqref{2.68},
 and the second to the last inequality follows from the fact that there exists $C>0$ such that,
 for small $\sigma >0$,
    \begin{equation*}
        \int_{(0,\sigma)\cap(\frac{r}{2},\infty)}|r-\eta|^{(n-2 - \alpha)\frac{\gamma}{\gamma - 1}} \dd \eta \leq C\qquad\,\mbox{uniformly in $r$},
    \end{equation*}
due to the fact that $(n-2 - \alpha)\frac{\gamma}{\gamma - 1} > -1$ as a consequence of the condition in \eqref{2.68}.

When $\alpha \in (-1,n-2]$, applying \eqref{continuity.8} and taking the same approach to \eqref{2.74},
we obtain that, for $\sigma$ small enough,
\begin{align}\label{2.74n}
    \begin{split}
        r^{n-1} \left | \int_0^\sigma\omega^{\frac{\gamma}{\gamma - 1}}(r,\eta) \eta^{n-1} \dd \eta \right |^{1 - \frac{1}{\gamma}} \leq C(D) \sigma^{n(1 - \frac{1}{\gamma})} \leq C(D) \sigma^{(n - 1)(1 - \frac{1}{\gamma})}.
    \end{split}
\end{align}
Thus, combining \eqref{2.73}--\eqref{2.74n} together, we have
\begin{align}\label{2.76}
        \begin{split}
            \sup_{[0,T] \times [0,D]} I_1
            \leq C(D,M,E_0) \sigma^{(n - 1)(1 - \frac{1}{\gamma})}\big(1 + \sigma^{1 - \frac{1}{\gamma}}\big) \to 0
            \qquad \text{ as $\sigma \to 0$}.
        \end{split}
    \end{align}

\smallskip
Considering $I_2$, for $\alpha \in (-1,n-2]$, we use \eqref{continuity.8} to obtain
\begin{align}\label{2.77}
            & I_2 \leq C r^{n-1} \int^{\hat{\sigma}}_\sigma (r\eta)^{-\frac{\alpha + 1}{2}} |\rho^{\v,b} - \rho^\v|(t,\eta) \eta^{n-1} \dd \eta \nonumber\\
            & \leq C({\hat{\sigma}}) r^{n - 2 - \alpha} \sigma^{-\frac{\alpha + 1}{2}} \int^{\hat{\sigma}}_\sigma  |\rho^{\v,b} - \rho^\v|(t,\eta) \dd \eta \nonumber\\
            & \leq C(D,{\hat{\sigma}},\sigma) \|\rho^{\v,b} - \rho^\v\|_{L^1(\sigma,{\hat{\sigma}})}.
\end{align}
For $\alpha \in (n-2,n-1)$, using the HLS inequality from Lemma \ref{simplehls} and \eqref{continuity.5} yields
    \begin{align}\label{2.78}
            & I_2 \leq C r^{n-1} \int^{\hat{\sigma}}_\sigma \frac{|r - \eta|^{n-2-\alpha}}{(r\eta)^{\frac{n-1}{2}}} |\rho^{\v,b} - \rho^\v|(t,\eta) \eta^{n-1} \dd \eta \nonumber\\
            & \leq C({\hat{\sigma}}) \Big( \frac{r}{\sigma} \Big)^{\frac{n-1}{2}}
              \big\| |\cdot|^{n-2-\alpha} * ((\rho^{\v,b} - \rho^\v)\mathds{1}_{(\sigma,{\hat{\sigma}})})\big\|_{C[0,D]}
                 \nonumber\\
            & \leq C(D,{\hat{\sigma}},\sigma) \|\rho^{\v,b} - \rho^\v\|_{L^\gamma(\sigma,{\hat{\sigma}})}.
    \end{align}
Combining \eqref{2.77} with \eqref{2.78}, we obtain via Lemma {\rm \ref{lem5.2}} that
    \begin{align}\label{2.79}
        \sup_{[0,T] \times [0,D]} I_2
        \leq C(D,{\hat{\sigma}},\sigma)\sup_{[0,T]} \big(\|\rho^{\v,b} - \rho^\v\|_{L^1(\sigma,{\hat{\sigma}})}
        + \|\rho^{\v,b} - \rho^\v\|_{L^\gamma(\sigma,{\hat{\sigma}})}\big) \to 0
        \qquad\text{as $b \to \infty$}.
    \end{align}

\smallskip
Considering $I_3$, since $\omega(r,\eta) \leq C |r-\eta|^{-(\alpha + 1)}$, we have
    \begin{align}
           \sup_{[0,T] \times [0,D]} I_3 &\leq C |D - {\hat{\sigma}}|^{-(\alpha + 1)} \sup_{[0,T]} \int^\infty_{\hat{\sigma}} |\rho^{\v,b} - \rho^\v|(t,\eta) \eta^{n-1} \dd \eta \nonumber\\
            & \leq C(M) |D - {\hat{\sigma}}|^{-(\alpha + 1)} \to 0 \qquad\text{ as ${\hat{\sigma}} \to \infty$}.\label{2.80}
    \end{align}
    Thus, combining \eqref{2.72} and \eqref{2.76} with \eqref{2.79}--\eqref{2.80}, we conclude     \begin{align*}
        \sup_{[0,T] \times [0,D]} |(\Phi_\alpha*\rho^{\v,b})_r r^{n-1} - (\Phi_\alpha*\rho^\v)_r r^{n-1}|(t,r) \to 0 \qquad\text{ as $b \to \infty$}.
    \end{align*}
    Thus, \eqref{2.68} is proved.

\medskip
4.  Regarding \eqref{2.69},  when $\alpha \in (0,n - 1)$,
using the same method as for the calculations for $\omega$
in Lemma \ref{potentially} with \eqref{continuity.8} and \eqref{continuity.4}--\eqref{continuity.5}, we have
    \begin{align}\label{2.81}
        K(r,\eta) \leq C \min\big\{|r - \eta|^{-\alpha},\,(r \eta)^{- \frac{\alpha}{2}} \big\}.
    \end{align}
That is, $K(r,\eta)$ has the same bound as $\omega$ from \eqref{continuity.8}
but with $\alpha + 1$ replaced by $\alpha$.
Then the proof of \eqref{2.69} is verbatim the proof of \eqref{2.68}
corresponding to the case that $\alpha \in (0,n-2)$.

\smallskip
When $\alpha \in (-1,0]$, taking ${\hat{\sigma}} >1$, we define $J_{\hat{\sigma}}(\boldsymbol{x}) := \frac{1}{{\hat{\sigma}}^n} J(\frac{\boldsymbol{x}}{\hat{\sigma}})$,
where $J$ is a standard spherically symmetric mollifier kernel
with $J \in C^\infty_{\rm c}(\mathbb{R}^n)$,
$\supp (J) \subset B_1$, and $\int_{\mathbb{R}^n} J(\boldsymbol{x}) \dd \boldsymbol{x} = 1$.
Using \eqref{Eul} and integrating by parts, we have
    \begin{align*}
        \begin{split}
        & \frac{\d}{\d t} \int^{b(t)}_a \rho^{\v,b} (\mathds{1}_{B_{2{\hat{\sigma}}}^c} * J_{\hat{\sigma}})
        k_\alpha(1 + r^2) r^{n-1} \dd r \\
        & = \int^{b(t)}_a \rho^{\v,b} u^{\v,b} \left ( (\mathds{1}_{B_{2{\hat{\sigma}}}^c} * J_{\hat{\sigma}}) (k_\alpha(1 + r^2))_r + (\mathds{1}_{B_{2{\hat{\sigma}}}^c} * J_{\hat{\sigma}})_r k_\alpha(1+ r^2) \right ) r^{n-1} \dd r.
        \end{split}
    \end{align*}
Integrating over time, we obtain
        \begin{align}\label{FTC1}
        \begin{split}
        &\int^{b(t)}_a \rho^{\v,b} (\mathds{1}_{B_{2{\hat{\sigma}}}^c} * J_{\hat{\sigma}}) k_\alpha(1 + r^2) r^{n-1} \dd r\\
        &= \int^{b}_a \rho^{\v,b}_0 (\mathds{1}_{B_{2{\hat{\sigma}}}^c} * J_{\hat{\sigma}}) k_\alpha(1 + r^2) r^{n-1} \dd r\\
        & \quad + \int^t_0 \int^{b(s)}_a \left ( (\mathds{1}_{B_{2{\hat{\sigma}}}^c} * J_{\hat{\sigma}}) (k_\alpha(1 + r^2))_r + (\mathds{1}_{B_{2{\hat{\sigma}}}^c} * J_{\hat{\sigma}})_r k_\alpha(1+ r^2) \right ) \rho^{\v,b}u^{\v,b} r^{n-1} \dd r \dd s.
        \end{split}
    \end{align}
As in \eqref{logest}, for $\boldsymbol{x} \in K = \overline{B_D}$ and $q \in (1,\gamma)$,
we use \eqref{logtings} and \eqref{minustings} to obtain
     \begin{align*}
        \begin{split}
            & \big|(\Phi_\alpha * \rho^{\v,b})(\boldsymbol{x}) - (\Phi_\alpha * \rho^{\v})(\boldsymbol{x})\big|\\
            & \leq \int_{B_{3{\hat{\sigma}}}(\boldsymbol{x})} k_\alpha(|\boldsymbol{x} - \boldsymbol{y}|) |\rho^{\v,b} - \rho^{\v}|(\boldsymbol{y}) \dd \boldsymbol{y} + \int_{B_{3{\hat{\sigma}}}(\boldsymbol{x})^c} k_\alpha(|\boldsymbol{x} - \boldsymbol{y}|) |\rho^{\v,b} - \rho^{\v}|(\boldsymbol{y}) \dd \boldsymbol{y} \\
            & \leq \| k_\alpha(|\cdot|) \|_{L^\frac{q}{q - 1}(B_{3{\hat{\sigma}}})} \|\rho^{\v,b} - \rho^{\v}\|_{L^q(B_{3{\hat{\sigma}}}(\boldsymbol{x}))} \\
            & \quad\, + C \int_{B_{3{\hat{\sigma}}}(\boldsymbol{x})^c} \left (k_\alpha(1 + D^2) + k_\alpha (1 +  |\boldsymbol{y}|^2) \right ) |\rho^{\v,b} - \rho^{\v}|(\boldsymbol{y}) \dd \boldsymbol{y} \\
            & \leq C({\hat{\sigma}}) \|\rho^{\v,b} - \rho^{\v}\|_{L^q(\mathbb{R}^n)} + C(D) \|\rho^{\v,b} - \rho^{\v}\|_{L^1(\mathbb{R}^n)} \\
            & \quad\, + C \int_{B_{3{\hat{\sigma}}}(\boldsymbol{x})^c} k_\alpha (1 +  |\boldsymbol{y}|^2)  \rho^{\v}(\boldsymbol{y}) \dd \boldsymbol{y} + C \int_{B_{3{\hat{\sigma}}}(\boldsymbol{x})^c} k_\alpha (1 +  |\boldsymbol{y}|^2)  \rho^{\v,b}(\boldsymbol{y}) \dd \boldsymbol{y} \\
            & =: I_1 + I_2 + I_3 + I_4.
        \end{split}
    \end{align*}
By Lemma \ref{Lq}, for fixed ${\hat{\sigma}}$, we have
    \begin{equation}\label{alpha1}
        \sup_{[0,T] \times [0,D]} |I_1 + I_2| \longrightarrow 0 \qquad \mbox{as $b \to \infty$}.
    \end{equation}

For $I_3$, by Lemma \ref{simple} and the Lebesgue dominated convergence theorem,
we obtain
\begin{equation*}
\sup_{[0,T] \times [0,D]} |I_3| \longrightarrow 0
\qquad\mbox{as ${\hat{\sigma}} \to \infty$, uniformly in $b$}.
\end{equation*}

For $I_4$, by \eqref{FTC1}, we have
    \begin{align*}
        \begin{split}
            |I_4| & = C \int^{b(t)}_{3{\hat{\sigma}}} k_\alpha (1 + r^2)  \rho^{\v,b}(r) r^{n-1} \dd r \\
            & \leq C \int^{b(t)}_0 \rho^{\v,b} (\mathds{1}_{B_{2{\hat{\sigma}}}^c} * J_{\hat{\sigma}})
               k_\alpha(1 + r^2) r^{n-1} \dd r \\
            & = C \int^{b(t)}_0 \rho^{\v,b}_0 (\mathds{1}_{B_{2{\hat{\sigma}}}^c} * J_{\hat{\sigma}}) k_\alpha(1 + r^2) r^{n-1} \dd r \\
            & \quad + C \int^t_0 \int^{b(s)}_0 \left ( (\mathds{1}_{B_{2{\hat{\sigma}}}^c} * J_{\hat{\sigma}}) (k_\alpha(1 + r^2))_r + (\mathds{1}_{B_{2{\hat{\sigma}}}^c} * J_{\hat{\sigma}})_r k_\alpha(1 + r^2) \right ) \rho^{\v,b} u^{\v,b} r^{n-1} \dd r \dd s \\
            & \leq C \int^{\b(t)}_{\hat{\sigma}} \rho^{\v,b}_0 k_\alpha(1 + r^2) r^{n-1} \dd r + \frac{C}{{\hat{\sigma}}^{\alpha + 1}} \int^t_0 \int^{b(s)}_{\hat{\sigma}} \rho^{\v,b} u^{\v,b} r^{n-1} \dd r \dd s \\
            & \quad + \frac{C k_\alpha (1 + {\hat{\sigma}}^2)}{{\hat{\sigma}}} \int^t_0 \int^{3{\hat{\sigma}}}_{\hat{\sigma}} \rho^{\v,b} u^{\v,b} r^{n-1} \dd r \dd s \\
            & \leq C \int^{\infty}_{\hat{\sigma}} \rho^{\v,b}_0 k_\alpha(1 + r^2) r^{n-1} \dd r + \frac{C k_\alpha (1 + {\hat{\sigma}}^2)}{{\hat{\sigma}}} \int^t_0 \int^{b(s)}_a \rho^{\v,b} u^{\v,b} r^{n-1} \dd r \dd s \\
            & =: I_{4,1} + I_{4,2}.
        \end{split}
    \end{align*}
By Lemma \ref{kernel2} and the Lebesgue dominated convergence theorem, we obtain
\begin{equation}\label{alpha3}
\sup_{[0,T] \times [0,D]} |I_{4,1}| \longrightarrow 0
 \qquad \mbox{as ${\hat{\sigma}} \to \infty$, uniformly in $b$}.
\end{equation}
By Lemma \ref{Energy} and the Cauchy-Schwartz inequality, we have
\begin{align}\label{alpha4}
 \sup_{[0,T] \times [0,D]} |I_{4,2}|
&\leq \frac{C k_\alpha (1 + {\hat{\sigma}}^2)}{{\hat{\sigma}}}
\int^t_0 \int^{b(s)}_a \big(\rho^{\v,b} + \rho^{\v,b} (u^{\v,b})^2\big) r^{n-1} \dd r \dd s \nonumber\\
& \leq \frac{C(M,E_0,T) k_\alpha (1 + {\hat{\sigma}}^2)}{{\hat{\sigma}}} \longrightarrow 0
            \qquad \mbox{as ${\hat{\sigma}} \to \infty$, uniformly in $b$.}
\end{align}
Thus, combining \eqref{alpha1}--\eqref{alpha4} together yields that, for $\alpha \in (-1,0]$,
    \begin{equation*}
        (\Phi_\alpha*\rho^{\v,b})(t,r)r^{n-1} \longrightarrow (\Phi_\alpha*\rho^\v)(t,r)r^{n-1} \qquad
	\text{ in $C_{\rm loc}([0,T]\times[0,\infty))$ as $b \to \infty$}.     \end{equation*}

\smallskip
5. For \eqref{2.71}, it follows from Lemma \ref{Lq} and \eqref{2.69}
that, when $\alpha \in (-1,n - 1)$ and ${\hat{\sigma}} > 0$,     \begin{align}\label{missed}
            \int_0^{\hat{\sigma}}
             \big|\rho^{\v,b} (\Phi_\alpha*\rho^{\v,b}) - \rho^{\v} (\Phi_\alpha*\rho^{\v})\big|(t,r)\,
             r^{n-1} \dd r \longrightarrow 0
            \qquad \mbox{as $b\to \infty$}.
    \end{align}
When $\alpha \in (0,n-1)$, using \eqref{2.81}, we have
    \begin{align}\label{2.82}
            & \int_{\hat{\sigma}}^\infty\big|\rho^{\v,b} (\Phi_\alpha*\rho^{\v,b})\big|(t,r)\, r^{n-1} \dd r \nonumber\\
            & \leq C \int_{\hat{\sigma}}^\infty \rho^{\v,b}(t,r)
            \bigg ( \Big(\int_0^{r-1}+ \int_{r+1}^\infty \Big)
            \rho^{\v,b}(t,\eta) \eta^{n-1} \dd \eta
            + \int_{r-1}^{r+1} \frac{1}{(r \eta)^{\frac{\alpha}{2}}} \rho^{\v,b}(t,\eta) \eta^{n-1} \dd \eta \bigg ) r^{n-1} \dd r \nonumber\\
            & \leq C \int_{\hat{\sigma}}^\infty \rho^{\v,b}(t,r)
            \bigg ( \Big(\int_0^{r-1}+\int_{r+1}^\infty\Big) \rho^{\v,b}(t,\eta) \eta^{n-1} \dd \eta
            + \frac{1}{({\hat{\sigma}}-1)^{\alpha}} \int_{r-1}^{r+1} \rho^{\v,b}(t,\eta) \eta^{n-1} \dd \eta \bigg ) r^{n-1} \dd r\nonumber\\
            & \leq C(M) \int_{\hat{\sigma}}^\infty \rho^{\v,b}(t,r) r^{n-1} \dd r.
    \end{align}
We can also replace $\rho^{\v,b}$ by $\rho^{\v}$ in \eqref{2.82}.
Thus, for $\delta > 0$, with  \eqref{2.99}, \eqref{2.101}, and \eqref{2.82},
we obtain that, for $b$ and ${\hat{\sigma}}$ sufficiently large,
    \begin{align}\label{2.102}
            \int_{\hat{\sigma}}^\infty\big|\rho^{\v,b} (\Phi_\alpha*\rho^{\v,b})
              - \rho^{\v} (\Phi_\alpha*\rho^{\v})\big|(t,r)\, r^{n-1} \dd r \leq C(M) \delta.
    \end{align}
When $\alpha \in (-1,0]$, using \eqref{logtingaling}, \eqref{minus}, and Lemma \ref{simple}, we have
    \begin{align*}
        \begin{split}
            & \int_{\hat{\sigma}}^\infty |\rho^{\v,b} (\Phi_\alpha*\rho^{\v,b})|(t,r)\, r^{n-1} \dd r \\
            & = \frac{1}{\omega_n} \Big| \int_{B_{\hat{\sigma}}^c} \int_{\mathbb{R}^n} \rho^{\v,b}(\boldsymbol{x}) k_\alpha (|\boldsymbol{x} - \boldsymbol{y}|) \rho^{\v,b}(\boldsymbol{y}) \dd \boldsymbol{x} \dd \boldsymbol{y} \Big| \\
            & \leq \frac{1}{\omega_n} \int_{B_{\hat{\sigma}}^c} \rho^{\v,b}(\boldsymbol{y}) \Big( \int_{B_1(\boldsymbol{y})} \rho^{\v,b}(\boldsymbol{x}) |k_\alpha (|\boldsymbol{x} - \boldsymbol{y}|)| \dd \boldsymbol{x} + \int_{B_1(\boldsymbol{y})^c} \rho^{\v,b}(\boldsymbol{x}) k_\alpha (|\boldsymbol{x} - \boldsymbol{y}|) \dd \boldsymbol{x} \Big) \dd \boldsymbol{y} \\
            & \leq \frac{1}{\omega_n} \int_{B_{\hat{\sigma}}^c} \rho^{\v,b}(\boldsymbol{y}) \Big ( \| \rho^{\v,b} \|_{L^\gamma(B_1(\boldsymbol{y}))} \|k_\alpha (|\cdot|) \|_{L^\frac{\gamma}{\gamma - 1}(B_1)} \\
            & \qquad \qquad\qquad\qquad\,\, + C \int_{|\boldsymbol{x} - \boldsymbol{y}| \geq 1} \rho^{\v,b}(\boldsymbol{x}) \left (1 + k_\alpha (1 +  |\boldsymbol{x}|^2) + k_\alpha (1 +  |\boldsymbol{y}|^2) \right ) \dd \boldsymbol{x} \Big ) \dd \boldsymbol{y} \\
            & \leq C(M,E_0,T) \int_{B_{\hat{\sigma}}^c} \rho^{\v,b}(\boldsymbol{x}) \left ( 1 + k_\alpha(1+|\boldsymbol{x}|^2) \right ) \dd \boldsymbol{x} \\
            & \leq C(M,E_0,T) \Big( \|\rho^{\v,b} - \rho^{\v,b}\|_{L^1(B_{\hat{\sigma}}^c)} + \int_{B_{\hat{\sigma}}^c} \rho^{\v}(\boldsymbol{x}) \dd \boldsymbol{x} + \int_{B_{\hat{\sigma}}^c} \rho^{\v,b}(\boldsymbol{x})  k_\alpha(1+|\boldsymbol{x}|^2) \dd \boldsymbol{x} \Big).
        \end{split}
    \end{align*}
Notice that, by \eqref{alpha3}--\eqref{alpha4},
    \begin{equation}\label{alpha5}
        \sup_{[0,T]} \Big|\int_{B_{\hat{\sigma}}^c} \rho^{\v,b}(\boldsymbol{x}) k_\alpha(1+|\boldsymbol{x}|^2) \dd \boldsymbol{x} \Big| \longrightarrow 0
       \qquad\mbox{as $b \to \infty$ \,\,  uniformly in $b$}.
    \end{equation}
Thus, for $\delta > 0$, by Lemmas \ref{simple}--\ref{Lq}, the Lebesgue dominated convergence theorem, and \eqref{alpha5}, we obtain that, for $b$ and ${\hat{\sigma}}$ sufficiently large,
    \begin{equation}\label{bigsig}
        \int_{\hat{\sigma}}^\infty
        \big|\rho^{\v,b} (\Phi_\alpha*\rho^{\v,b}) - \rho^{\v} (\Phi_\alpha*\rho^{\v})\big|(t,r)\, r^{n-1} \dd r
        \leq C(M,E_0,T) \delta.
    \end{equation}
Therefore, combining \eqref{missed} and \eqref{2.102} with \eqref{bigsig} leads
to \eqref{2.71} for $\alpha \in (-1,n-1)$.
\end{proof}

\begin{lemma}
Under the assumptions of {\rm Lemma {\rm \ref{lem5.2}}}, the following holds{\rm :}
\begin{align*}
& \int_0^\infty  \Big(\frac12 \Big|\frac{m^{\v}}{\sqrt{\rho^\varepsilon}} \Big|^2
 + \rho^{\v}e(\rho^\v) + \frac{\kappa}{2} \rho^\v (\Phi_\alpha * \rho^\v)\Big)(t,r)\,r^{n-1}\dd r \\
& \leq \int_0^\infty\Big(\frac12 \Big|\frac{m^{\v}_0}{\sqrt{\rho^\varepsilon_0}}\Big|^2
+ \rho^{\v}_0 e(\rho^{\v}_0) + \frac{\kappa}{2} \rho^\v_0 (\Phi_\alpha * \rho^\v_0)\Big)(r)\,r^{n-1}\dd r.
\end{align*}
\end{lemma}

\begin{proof}
It follows from Lemma \ref{Energy} that
\begin{align*}
& \int_0^{b(t)}\Big(\frac12 \Big|\frac{m^{\v,b}}{\sqrt{\rho^{\varepsilon,b}}}\Big|^2
+ \rho^{\v,b} e(\rho^{\v,b})
+ \frac{\kappa}{2} \rho^{\varepsilon,b} \big(\Phi_\alpha * \rho^{\v,b}\big)\Big)(t,r)\,r^{n-1}\dd r \\
& \leq \int_0^{b(t)} \Big(\frac12 \Big|\frac{m^{\v,b}_0}{\sqrt{\rho^{\varepsilon,b}_0}}\Big|^2
+ \rho^{\v,b}_0 e(\rho^{\v,b}_0)
+ \frac{\kappa}{2} \rho^{\varepsilon,b}_0\big(\Phi_\alpha * \rho^{\v,b}_0\big) \Big)(r)\,r^{n-1}\dd r.
\end{align*}
Using the construction of the approximation of $(\rho^\v_0,m^\v_0)$ and Lemma \ref{expand}, we have
\begin{align*}
& \int_0^{b(t)} \Big(\frac12 \Big|\frac{m^{\v,b}_0}{\sqrt{\rho^{\varepsilon,b}_0}}\Big|^2
+ \rho^{\v,b}_0 e(\rho^{\v,b}_0) + \frac{\kappa}{2} \rho^{\varepsilon,b}_0 (\Phi_\alpha * \rho^{\v,b}_0)\Big)(r)\,
r^{n-1}\dd r \\
& \longrightarrow \int_0^\infty
\Big(\frac12 \Big|\frac{m^{\v}_0}{\sqrt{\rho^\varepsilon_0}}\Big|^2
+ \rho^{\v}_0 e(\rho^{\v}_0) + \frac{\kappa}{2} \rho^{\varepsilon}_0 (\Phi_\alpha * \rho^\v_0)\Big)(r)\,r^{n-1}\dd r \qquad\text{as } b \to \infty.
\end{align*}
Now, by the Fatou lemma, Lemma \ref{expand}, and \eqref{2.71}, we obtain
 \begin{align*}
 & \int_0^\infty  \Big(\frac12 \Big|\frac{m^{\v}}{\sqrt{\rho^\varepsilon}}\Big|^2
 + \rho^{\v} e(\rho^{\v}) + \frac{\kappa}{2} \rho^{\varepsilon} (\Phi_\alpha * \rho^\v)\Big)(t,r)\,r^{n-1}\dd r \\
 & \leq \liminf_{b \to \infty} \int_0^{b(t)}
 \Big(\frac12 \Big|\frac{m^{\v,b}}{\sqrt{\rho^{\varepsilon,b}}} \Big|^2
 + \rho^{\v,b} e(\rho^{\v,b})
 + \frac{\kappa}{2} \rho^{\varepsilon,b} \big(\Phi_\alpha * \rho^{\v,b}\big)\Big)(t,r)\,r^{n-1}\dd r.
    \end{align*}
This completes the proof.
\end{proof}

\begin{lemma}
Under the assumptions of {\rm Lemma {\rm \ref{lem5.2}}}, let $\bp(t,\boldsymbol{x})\in \left(C^2_0([0,T]\times \mathbb{R}^n)\right)^n$ be any smooth function with compact support
so that $\bp(T,\boldsymbol{x})=0$. Then
\begin{align*}
	&\int_{\mathbb{R}_+^{n+1}} \Big\{\M^{\v} \cdot\partial_t\bp +\frac{\M^{\v}}{\sqrt{\rho^{\v}}} \cdot \big(\frac{\M^{\v}}{\sqrt{\rho^{\v}}}\cdot \nabla\big)\bp
	+p(\rho^{\v}) \mbox{\rm div} \bp -\rho^\v \nabla(\Phi_\alpha * \rho^\v)\cdot \bp\Big\}\,\dd\boldsymbol{x} \dd t\\
	&\quad +\int_{\mathbb{R}^n} \M_0^\v\cdot \bp(0,\boldsymbol{x})\,\dd\boldsymbol{x}
	\\
	&=-\v\int_{\mathbb{R}_+^{n+1}}
	\Big\{\frac{1}{2}\M^{\v}\cdot \big(\Delta \bp+\nabla\mbox{\rm div}\,\bp \big)
	+ \frac{\M^{\v}}{\sqrt{\rho^{\v}}} \cdot \big(\nabla\sqrt{\rho^{\v}}\cdot \nabla\big)\bp
    + \nabla\sqrt{\rho^{\v}}  \cdot \big(\frac{\M^{\v}}{\sqrt{\rho^{\v}}}\cdot \nabla\big)\bp\Big\}\,\dd\boldsymbol{x}\dd t\nonumber\\
	&=\sqrt{\v}\int_{\mathbb{R}_+^{n+1}}
	\sqrt{\rho^{\v}} \Big\{V^{\v}  \frac{\boldsymbol{x}\otimes\boldsymbol{x}}{r^2}
	+\frac{\sqrt{\v}}{r}\frac{m^\v}{\sqrt{\rho^\v}}\big(I_{n\times n}-\frac{\boldsymbol{x}\otimes\boldsymbol{x}}{r^2}\big)\Big\}: \nabla\bp\, \dd\boldsymbol{x}\dd t,
\end{align*}
where $V^{\v}(t,r)\in L^2(0,T; L^2(\mathbb{R}^n))$ is a function such that
$$
\displaystyle\int_0^T\int_{\mathbb{R}^n} |V^{\v}(t,\boldsymbol{x})|^2\, \dd\boldsymbol{x}\dd t\leq C(E_0,M)
\quad \mbox{for some $C(E_0,M)>0$ independent of $T>0$}.
$$
\end{lemma}

The proof of this lemma is the same as the argument for that of \cite{Chen2021}.

\begin{lemma}[$H_{\rm loc}^{-1}$--Compactness]
Under the assumptions of {\rm Lemma {\rm \ref{lem5.2}}}, let $(\eta, q)$ be a weak entropy pair defined in \eqref{5.10}
for any compactly supported smooth function $\psi(s)$ in $\mathbb{R}$. Then
\begin{align*}
	\partial_t\eta(\rho^\v,m^\v)+\partial_rq(\rho^\v,m^\v) \qquad   \mbox{is compact in $ H^{-1}_{\rm loc}(\mathbb{R}^2_+)$}.
\end{align*}
\end{lemma}

\begin{proof}
We multiply $\eqref{Eul}_1$ by $\eta_\rho^{\v,b} := \eta_\rho(\rho^{\v,b},m^{\v,b})$
and $\eqref{Eul}_2$ by $\eta_m^{\v,b} := \eta_m(\rho^{\v,b},m^{\v,b})$.
Then, by the construction of the entropy pair, we obtain
\begin{align}\label{entropyting}
\partial_t \eta^{\v,b} + \partial_r q^{\v,b} = & - \frac{n-1}{r}m^{\v,b}(\eta_\rho^{\v,b} + u^{\v,b} \eta_m^{\v,b}) - \kappa \eta^{\v,b}_m \rho^{\v,b} (\Phi_\alpha * \rho^{\v,b})_r \nonumber\\
& + \v \eta_m^{\v,b} \Big\{ \Big( \rho^{\v,b}(u^{\v,b}_r + \frac{n-1}{r} u^{\v,b})\Big)_r
  - \frac{n-1}{r} \rho_r^{\v,b} u^{\v,b} \Big\}.
\end{align}
Multiplying \eqref{entropyting} by $\phi(t,r) \in C^\infty_0(\mathbb{R}^2_+)$ and integrating it
over $\mathbb{R}^2_+$, and then integrating by parts, we obtain
 \begin{align*}
        \displaystyle
        & \int_{\mathbb{R}^2_+} \big(\partial_t \eta^{\v,b} + \partial_r q^{\v,b}\big)\phi \dd r\dd t \\
        \displaystyle
        & = - \int_{\mathbb{R}^2_+} \frac{n-1}{r}m^{\v,b}
            \big(\eta_\rho^{\v,b} + u^{\v,b} \eta_m^{\v,b}\big) \phi \dd r \dd t
            - \v \int_{\mathbb{R}^2_+} \big(\eta_m^{\v,b})_r  \rho^{\v,b}(u^{\v,b}_r
              + \frac{n-1}{r} u^{\v,b}\big) \phi \dd r \dd t \\
        \displaystyle
        &\quad  - \v \int_{\mathbb{R}^2_+} \eta_m^{\v,b}
        \rho^{\v,b}\big(u^{\v,b}_r + \frac{n-1}{r} u^{\v,b}\big) \phi_r \dd r \dd t
        - \v \int_{\mathbb{R}^2_+} \eta_m^{\v,b} \frac{n-1}{r} \rho_r^{\v,b} u^{\v,b}\phi\dd r \dd t \\
        \displaystyle
        & \quad - \kappa \int_{\mathbb{R}^2_+} \eta^{\v,b}_m \rho^{\v,b}
         \big(\Phi_\alpha * \rho^{\v,b}\big)_r \phi \dd r \dd t:= \sum^5_{i=1} I^{\v,b}_i.
 \end{align*}
Notice that the terms, $I^{\v,b}_i$ for $i=1,\cdots,4$, are the same as in \cite[Lemma 4.12]{Chen2021},
so that they can be handled in the same ways.
However, $I^{\v,b}_5$ is different and, as such, requires another approach.

Since $\phi$ is compactly supported in $\mathbb{R}^2_+$,
there exist $(d,D) \Subset (0,\infty)$ and $T>0$
such that $\supp (\phi) \subset (0,T) \times (d,D)$ so that
    \begin{align}\label{3.14}
            |I^{\v,b}_5| \leq C\int^T_0 \int_d^D
            \Big| \eta^{\v,b}_m \rho^{\v,b} (\Phi_\alpha * \rho^{\v,b})_r \Big| \dd r \dd t
            \leq C_\psi \int^T_0 \int_d^D \rho^{\v,b}(r) \int^{b(t)}_a \big|\omega(r,\eta)\big|
            \rho^{\v,b}(\eta) \eta^{n-1} \dd \eta \dd r \dd t.
    \end{align}
We need to prove the following estimate:
\begin{align}\label{3.15}
\int^T_0 \int_d^D \big|\eta^{\v,b}_m \rho^{\v,b} (\Phi_\alpha * \rho^{\v,b})_r \big|\dd r \dd t
\leq C_\psi(d,D,M,E_0,T).
\end{align}
When $\alpha \in (n-2,n-1)$, combining \eqref{3.14} with \eqref{continuity.8}
and using the Fubini theorem, we have
\begin{align}\label{3.16}
& \int^T_0 \int_d^D \left | \eta^{\v,b}_m \rho^{\v,b} (\Phi_\alpha * \rho^{\v,b})_r \right | \dd r \dd t
   \nonumber\\
& \leq C_\psi \int^T_0 \int_d^D \rho^{\v,b}(r) \int_{(a,b(t))\cap(r-\frac{d}{2},r + \frac{d}{2})}|\omega(r,\eta)| \rho^{\v,b}(\eta) \eta^{n-1} \dd \eta \dd r \dd t \nonumber\\
&\quad\,  + C_\psi \int^T_0 \int_d^D \rho^{\v,b}(r) \int_{(a,b(t))\cap(r-\frac{d}{2},r + \frac{d}{2})^c}|\omega(r,\eta)| \rho^{\v,b}(\eta) \eta^{n-1} \dd \eta \dd r \dd t\nonumber\\
& \leq C_\psi \int^T_0 \int_d^D \rho^{\v,b}(r) \int_{(a,b(t))\cap(r-\frac{d}{2},r + \frac{d}{2})} \frac{|r - \eta|^{n-2-\alpha}}{(r \eta)^{\frac{n-1}{2}}} \rho^{\v,b}(\eta) \eta^{n-1} \dd \eta \dd r \dd t
   \nonumber\\
&\quad\,  + C_\psi(d) \int^T_0 \int_d^D \rho^{\v,b}(r) \int_{(a,b(t))\cap(r-\frac{d}{2},r + \frac{d}{2})^c} \rho^{\v,b}(\eta) \eta^{n-1} \dd \eta \dd r \dd t\nonumber\\
& \leq C_\psi(d,M,T) + C_\psi \int^T_0 \int_d^D \rho^{\v,b}(r) \int_{(a,b(t))\cap(r-\frac{d}{2},r + \frac{d}{2})} \frac{|r - \eta|^{n-2-\alpha}}{d^{n-1}} \rho^{\v,b}(\eta) \eta^{n-1} \dd \eta \dd r \dd t
   \nonumber\\
& \leq C_\psi(d,M,T) +  C_\psi(d) \int^T_0 \int^{b(t)}_a \rho^{\v,b}(\eta) \eta^{n-1} \int_d^D |r - \eta|^{n-2-\alpha} \rho^{\v,b}(r) \dd r \dd \eta \dd t.
\end{align}
Thus, for $\gamma > \frac{1}{(n-1) - \alpha}$, using Theorem \ref{pointwise} and \eqref{3.16}, we obtain
\begin{align*}
& \int^T_0 \int_d^D \big| \eta^{\v,b}_m \rho^{\v,b} (\Phi_\alpha * \rho^{\v,b})_r \big|\dd r \dd t \\
& \leq C_\psi \int^T_0 \int^{b(t)}_a \rho^{\v,b}(\eta) \eta^{n-1} \sup_{z \in \mathbb{R}_+}
            \Big\{ \int_d^D |r - z|^{n-2-\alpha} \rho^{\v,b}(r) \dd r \Big\} \dd \eta \dd t \\
& = C_\psi \int^T_0 \int^{b(t)}_a \rho^{\v,b}(\eta) \eta^{n-1} \sup_{z \in [d,D]}
            \Big\{ \int_d^D |r - z|^{n-2-\alpha} \rho^{\v,b}(r) \dd r \Big\} \dd \eta \dd t \\
& \leq C_\psi(K) \| \rho^{\v,b} \|_{L^\gamma(d,D)} \int^T_0 \int^{b(t)}_a \rho^{\v,b}(\eta) \eta^{n-1} \dd \eta \dd t \\
& \leq C_\psi(d,D,M,E_0,T).
\end{align*}
When $\alpha \in (-1,n-2]$, using the approach to \eqref{3.16} but instead with the alternate estimate for $\omega$ as in \eqref{continuity.8} for $\alpha \in (-1,n-2]$,
we obtain \eqref{3.15}.
\end{proof}

Thus, to sum up, we have the following corollary.
\begin{corollary}
Under the assumptions of {\rm Lemma {\rm \ref{lem5.2}}}, for any $T>0$,
there exists a global spherically symmetric solution $(\rho^\v, m^\v)=(\rho^\v, \rho^\v u^\v)$
of CNSREs \eqref{0.2} in the sense of {\rm Definition \ref{CNSREs}} satisfying
\begin{align*}
\begin{split}
&\rho^{\v}(t,r)\geq0 \,\,\,\, a.e.,
\end{split}\\
\begin{split}
& u^\v(t,r)=0, \,\,\big(\frac{m^\v}{\sqrt{\rho^\v}}\big)(t,r)=\sqrt{\rho^\v}(t,r)u^\v(t,r)=0 \quad\, a.e.\,\, \mbox{on $\{(t,r)\,:\, \rho^\v(t,r)=0\}$},
\end{split}\\
\begin{split}
& \int^\infty_0 \rho^{\v}(t,r) r^{n-1} \dd r = M \qquad \mbox{for all $t\in[0,T]$},
\end{split} \\
\begin{split}
&\int_0^\infty  \Big(\frac12\rho^{\v} |u^{\v}|^2+ \rho^{\v}e(\rho^\v) - \frac{\kappa}{2}\rho^\v (\Phi_\alpha * \rho^\v)\Big)(t,r)\,r^{n-1}\dd r +\v\int_0^t\int_0^\infty  (\rho^{\v} |u^{\v}|^2)(s,r)\, r^{n-3} \dd r \dd s \\
&\quad\leq C(M,E_0) \qquad \mbox{for all $t\in[0,T]$},
\end{split}\\[1mm]
\begin{split}
& \int_0^\infty  \bigg(\frac12 \Big|\frac{m^{\v}}{\sqrt{\rho^\varepsilon}} \Big|^2
+ \rho^{\v}e(\rho^\v) + \frac{\kappa}{2} \rho^\v (\Phi_\alpha * \rho^\v)\bigg)(t,r)\,r^{n-1}\dd r \\
& \quad \leq \int_0^\infty  \bigg(\frac12 \Big|\frac{m^{\v}_0}{\sqrt{\rho^\varepsilon_0}} \Big|^2 + \rho^{\v}_0 e(\rho^{\v}_0)
+ \frac{\kappa}{2} \rho^\v_0 (\Phi_\alpha * \rho^\v_0)\bigg)(r)\,r^{n-1}\dd r \qquad \mbox{for all $t\in[0,T]$},
\end{split} \\
\begin{split}
&\v^2 \int_0^\infty \big|(\sqrt{\rho^{\v}(t,r)})_r\big|^2\,r^{n-1}\dd r
+\v\int_0^T\int_{0}^\infty \big|\big((\rho^{\v}(s,r))^{\frac{\gamma}{2}}\big)_r\big|^2\,r^{n-1}\dd r\dd s\\
&\quad\leq C(M,E_0,T) \qquad \mbox{for all $t\in[0,T]$},
\end{split}\\[2mm]
\begin{split}
\int_0^T\int_d^D  \big(\rho^{\v}|u^{\v}|^3+(\rho^{\v})^{\gamma+\theta}+(\rho^{\v})^{\gamma+1}\big)(t,r)\, r^{n-1}\dd r \dd t
\leq C(d,D,M,E_0,T) \qquad \mbox{for all $t\in[0,T]$},
\end{split}
\end{align*}
where $[d,D]\Subset (0,\infty)$. Let $(\eta, q)$ be a weak entropy pair defined in \eqref{5.10}
for any smooth compact supported function $\psi(s)$ on $\mathbb{R}$. Then, for $\v\in (0, \v_0]$,
\begin{align*}
	\partial_t\eta(\rho^\v,m^\v)+\partial_rq(\rho^\v,m^\v) \qquad   \mbox{is compact in $ H^{-1}_{\rm loc}(\mathbb{R}^2_+)$}.
\end{align*}
\end{corollary}

\smallskip
As in \cite{Chen2021}, similar to the compensated compactness argument in \cite{Chen_2010},
we can derive that there exists a vector function $(\rho,m)(t,r)$ such that
\begin{align*}
(\rho^\v,m^{\v})\longrightarrow (\rho,m) \qquad  \mbox{{\it a.e.} $\,(t,r)\in \mathbb{R}^2_+$\, as $\v\rightarrow0^+$ (up to a sub-sequence)}.
\end{align*}
Notice that
$m(t,r)=0$ {\it a.e.} on $\{(t,r) \,:\, \rho(t,r)=0\}$.
We can define the  limit velocity $u(t,r)$ by setting $u(t,r):=\frac{m(t,r)}{\rho(t,r)}$ {\it a.e.} on $\{(t,r)\,:\,\rho(t,r)\neq0\}$
and $u(t,r):=0$ {\it a.e.} on $\{(t,r)\,: \rho(t,r)=0\,\,\, \mbox{or $r=0$}\}$.
Then
\begin{align}
m(t,r)=\rho(t,r) u(t,r).\nonumber
\end{align}

We can also define $(\frac{m}{\sqrt{\rho}})(t,r):=\sqrt{\rho(t,r)} u(t,r)$, which is $0$ {\it a.e.} on $\{(t,r) \ :\ \rho(t,r)=0\}$.
Moreover, we obtain that, as $\v\rightarrow0^+$,
\begin{equation*}
\frac{m^{\v}}{\sqrt{\rho^{\v}}}\equiv\sqrt{\rho^{\v}} u^{\v}\longrightarrow \frac{m}{\sqrt{\rho}}\equiv\sqrt{\rho} u \qquad
\mbox{strongly in $L^2([0,T]\times[d,D], r^{n-1}\dd r \dd t)$}
\end{equation*}
for any given $T$ and $[d,D]\Subset (0,\infty)$, and
\begin{align}\label{3.18}
(\rho^\v,m^{\v})\longrightarrow (\rho,m) \quad \mbox{in $L^{p_1}_{\rm loc}(\mathbb{R}^2_+)\times L^{p_2}_{\rm loc}(\mathbb{R}^2_+)$}
\end{align}
for $p_1\in[1,\gamma+1)$ and $p_2\in[1,\frac{3(\gamma+1)}{\gamma+3})$. \\\

Integrating \eqref{3.1} over $[0,T]$, letting $\varepsilon \to 0^+$ (up to a sub-sequence),
and using \eqref{3.18}, we obtain

\begin{align}\label{3.19}
	\int^T_0 \int_{\mathbb{R}^n} \rho(t,\boldsymbol{x}) \zeta(t,\boldsymbol{x})\, \dd\boldsymbol{x} \dd t
	= \int^T_0 \int_{\mathbb{R}^n} \rho_0(\boldsymbol{x}) \zeta(0,\boldsymbol{x})\, \dd\boldsymbol{x} \dd t
	+ \int^T_0 \int_{0}^{t} \int_{\mathbb{R}^n} \big(\rho \zeta_t + \M\cdot\nabla\zeta\big)\,\dd\boldsymbol{x}\dd s \dd t
\end{align}
for any $\zeta \in C^1_c([0,T] \times \mathbb{R}^n)$. Thus,
in the same way as for Lemma \ref{mass}, using \eqref{3.19}, we can obtain
\begin{align*}
\int^T_0 \int_{\mathbb{R}^n} \rho(t,\boldsymbol{x}) \zeta(t,\boldsymbol{x})\, \dd\boldsymbol{x} \dd t = MT.
\end{align*}
Then, in the same way as in the proof of Lemma \ref{lem5.1}, we have

\begin{lemma}
Under the assumptions of {\rm Theorem \ref{existence}}, the following statements hold
as $\v \rightarrow 0$ {\rm(}up to a sub-sequence{\rm):}
\begin{align*}
	&\Phi_\alpha*\rho^{\v} \longrightharpoonup \Phi_\alpha*\rho \quad
	\mbox{weak-$\ast$ in $L^\infty(0,T; W^{1,\frac{2n}{\alpha +2}}_{\rm loc}(\mathbb{R}^n))$ and weakly in $L^2(0,T; W^{1,\frac{2n}{\alpha +2}}_{\rm loc}(\mathbb{R}^n))$},\\[4mm]
	&(\Phi_\alpha*\rho^{\v})_r(t,r)r^{n-1} \longrightarrow (\Phi_\alpha*\rho)_r(t,r)r^{n-1} \quad
	\mbox{ in $C_{\rm loc}([0,T]\times[0,\infty))$}, \\[4mm]
    &(\Phi_\alpha*\rho^{\v})(t,r)r^{n-1} \longrightarrow (\Phi_\alpha*\rho)(t,r)r^{n-1} \quad
	\mbox{ in $C_{\rm loc}([0,T]\times[0,\infty))$}, \\[3mm]
& \int_0^\infty \big|\rho^{\v} (\Phi_\alpha*\rho^{\v})-\rho (\Phi_\alpha*\rho)\big|(t,r)\, r^{n-1}
  \dd r \longrightarrow 0.
\end{align*}
 In addition, for $\alpha \in (0,n-1)$,
    \begin{align*}
	\|\Phi_\alpha*\rho(t)\|_{L^{\frac{2n}{\alpha}}(\mathbb{R}^n)}+\|\nabla  \Phi_\alpha*\rho(t)\|_{L^\frac{2n}{\alpha + 2}(\mathbb{R}^n)}\leq C(M,E_0)\,\qquad\mbox{for $t\geq0$};
    \end{align*}
and, for $\alpha \in (-1,0]$, for any $K \Subset \mathbb{R}^n$,
    \begin{align*}
	\|\Phi_\alpha*\rho(t)\|_{L^\infty(K)}+\|\nabla \Phi_\alpha*\rho(t)\|_{L^\frac{2n}{\alpha + 2}(K)}\leq C(M,E_0,T,K)\qquad\,\mbox{for $t\geq0$}.
    \end{align*}
\end{lemma}

Then, similar to \cite{Chen2021}, we can conclude that
$(\rho,m)(t,r)$ define a global finite-energy
solution of CEREs \eqref{0.0}
in the sense of Definition \ref{definition}, which leads to Theorem \ref{existence}.

\section{Nonlinear Stability of Steady States for the Compressible Euler-Riesz Equations}\label{sec:stab}

In this section, we consider the attractive case $(\kappa = 1)$, $\alpha \in (0,n-1)$,
and $\gamma > \frac{n + \alpha}{n}$, and
assume that the initial data $(\rho_0,\mathcal{M}_0)(\boldsymbol{x})$
have finite initial mass and energy.
For $0<M<\infty$, we recall the definition for a set of functions $\varrho${\rm:}
\begin{align*}
X_M := \bigg \{ \varrho \in L^1_+(\mathbb{R}^n) \cap L^\frac{n + \alpha}{n}(\mathbb{R}^n) \,\, : \, \int_{\mathbb{R}^n} \varrho(\boldsymbol{x})\,\dd \boldsymbol{x} = M \bigg \}.
\end{align*}
For $\varrho\in X_M,$ we recall the definition of the energy functional $\mathcal{G}$:
\begin{align*}
\mathcal{G}(\varrho):=\int_{\mathbb{R}^n} \Big((\varrho e(\varrho))(\boldsymbol{x})+\frac{1}{2}\varrho(\boldsymbol{x})(\Phi_\alpha\ast \varrho)(\boldsymbol{x})\Big)\,\dd\boldsymbol{x}.
\end{align*}
We first focus on
proving the existence, uniqueness, regularity, and other properties of global minimizers of the free energy $\mathcal{G}$ on the set $X_M$. We include some details for the sake of completeness and clarity of exposition.

We
notice that one can show that the global minimizers
of the functional $\mathcal{G}$ in $X_M$, if they exist, are actually  radially symmetric. This is a very classical result that goes back to \cite{Auchmuty_1971} by using symmetric decreasing rearrangement via Lemma \ref{rearrange}. The following result on the existence of global minimizers is proven in \cite[Theorem 5]{Carrillo_2018}.

\begin{theorem}\label{lmul}
For $\alpha\in(0,n)$  with $\gamma > \frac{n + \alpha}{n}$, there exists a global
minimizer $\tilde{\rho} \in X_M$ to $\mathcal{G}$.
Moreover, any global minimizer is radially decreasing, compactly supported, and bounded,
which is a stationary solution in the sense of {\rm Definition \ref{steady}}.
Furthermore, any global minimizer satisfies the Euler-Lagrange conditions{\rm :}
There exists a constant {\rm (}Lagrangian multiplier{\rm )} $\lambda$ such that
$$
\begin{cases}
(\tilde{\rho}e(\tilde{\rho}))_{\tilde{\rho}} +\Phi_\alpha\ast\tilde{\rho}
=\lambda &
\mbox{ for $\boldsymbol{x}\in \Gamma$},\\[1mm]
(\Phi_\alpha\ast\tilde{\rho})(\boldsymbol{x})\geq\lambda
& \mbox{ for $\boldsymbol{x}\in \R^n\setminus\Gamma$},
\end{cases}
$$
where
$\Gamma=\{\boldsymbol{x}\in \R^n\,:\, \tilde{\rho}(\boldsymbol{x})>0\}$.
\end{theorem}

We need some parts of the core of the existence result of global minimizers in order to derive our stability estimates for CEREs. We sketch partially the proof making emphasis only on the relevant elements. The existence part of the proof follows the classical concentration compactness argument introduced by Lions \cite{Lions_1984} in the following result:

\begin{theorem}[\text{\cite[Theorem II.1]{Lions_1984}}]\label{minimizer}
    Suppose that
    \begin{equation}\label{potential}
        \Psi \in L^p_{\rm w}(\mathbb{R}^n) \text{{\rm (}weak $L^p$ space{\rm )}}\, \text{ for } p \in (1,\infty), \text{ with } \Psi \geq 0 \text{ a.e.},
    \end{equation}
    and
    \begin{equation}\label{convex}
        \begin{cases}
        \Pi:[0,\infty) \to [0,\infty) \text{ strictly convex,} \\
        \lim_{t \to 0^+} \Pi(t)t^{-1} = 0, \\
        \lim_{t \to \infty} \Pi(t)t^{-q} = \infty,
        \end{cases}
    \end{equation}
    where $q = \frac{p + 1}{p}$. Define
    \begin{equation*}
    \mathcal{A}_M := \bigg \{ \varrho \in L^1_+(\mathbb{R}^n) \cap L^q(\mathbb{R}^n)
     \, :  \int_{\mathbb{R}^n} \varrho(\boldsymbol{x})\, \dd \boldsymbol{x} = M \bigg \},
    \end{equation*}
    and the energy functional as follows{\rm :}
    \begin{equation}\label{free}
    \mathcal{F}(\varrho)=\int_{\mathbb{R}^n}\Big(\Pi(\varrho(\boldsymbol{x})) - \frac{1}{2}\varrho(\boldsymbol{x})(\Psi \ast \varrho)(\boldsymbol{x})\Big)\,\dd\boldsymbol{x}.
\end{equation}
Define the minimum of $\mathcal{F}$ over $\mathcal{A}_M$ by
\begin{equation*}
    F_M := \inf_{\varrho \in \mathcal{A}_M} \mathcal{F}(\varrho).
\end{equation*}
Then, for any minimizing sequence $\{\rho^j\}
\subset \mathcal{A}_M$,
there exists $\boldsymbol{y}_j \in \mathbb{R}^n$ and a minimizer
$\tilde{\rho} \in \mathcal{A}_M$ to $\mathcal{F}$ {\rm (}i.e., $\mathcal{F}(\tilde{\rho}) = F_M${\rm )}
with $\{T^{\boldsymbol{y}_j} \rho^j\}$ compact in $L^1(\mathbb{R}^n) \cap L^q(\mathbb{R}^n)$ such that
$T^{\boldsymbol{y}_j} \rho^j \to \tilde{\rho}$ in $L^1(\mathbb{R}^n) \cap L^q(\mathbb{R}^n)$ if and only if
\begin{equation}\label{inequality}
    F_M < F_m + F_{M - m} \qquad \text{ for all $m \in (0,M)$}.
\end{equation}
\end{theorem}

We can now use this theorem in our problem since $- \Phi_\alpha \in L^\frac{n}{\alpha}_{\rm w}(\mathbb{R}^n)$
and satisfies \eqref{potential}.
Thus, we take $p = \frac{n}{\alpha}$, $q = \frac{p + 1}{p} = \frac{n + \alpha}{n}$, and the internal energy function $\Pi(\tau):= \tau e(\tau) = \frac{a_0}{\gamma - 1} \tau^\gamma$ that satisfies \eqref{convex}.
So it suffices to prove \eqref{inequality} for applying Theorem \ref{minimizer}.
In order to verify condition \eqref{inequality}, we also employ a corollary from \cite{Lions_1984}.

\begin{corollary}[\text{\cite[Corollary II.1]{Lions_1984}}]\label{corollary}
Suppose that \eqref{potential} and \eqref{convex} are satisfied.
If, for some $\beta \in (1,n)$,
 \begin{equation}\label{Psi}
   \Psi(t \xi) \geq t^{-\beta} \Psi(\xi)
   \qquad \mbox{for all $t \geq 1$ and $\xi \in \mathbb{R}^n$ a.e.},
\end{equation}
then \eqref{inequality} is satisfied if and only if $F_M < 0$.
\end{corollary}

Denoting by
\begin{equation*}
g_M = \inf_{\varrho \in X_M} \mathcal{G}(\varrho),
\end{equation*}
the result in \cite[Theorem 5]{Carrillo_2018} shows that $g_M<0, \text{ for all } M>0.$ This is obtained via a simple scaling argument on a suitably normalized characteristic of a ball by using the homogeneities of the entropy and the interaction parts of functional $\mathcal{G}$ to express this as a function of the radius of the ball with competing homogeneities. This shows \eqref{inequality} for the free energy functional $\mathcal{G}$.

Clearly, $\Phi_\alpha$ satisfies \eqref{Psi}. In fact, with an equality for $\beta = \alpha$ and for all $\xi \in \mathbb{R}^n\setminus \{0\}$ that is for $\xi \in \mathbb{R}^n$ {\it a.e.}, since $g_M < 0$, Corollary \ref{corollary} implies that
there exists a minimizer $\tilde{\rho} \in X_M$ to $\mathcal{G}$ over $X_M$. The rest of the properties of the global minimizers of $\mathcal{G}$ over $X_M$ are referred to \cite{Carrillo_2018}. This finishes the proof of Theorem \ref{lmul}.

We now prove the existence of a radial minimizer of functional \eqref{free} over $\mathcal{A}_M$
under conditions \eqref{potential} and \eqref{convex}.

\begin{lemma}
Suppose that \eqref{potential}--\eqref{convex} are satisfied, with $\Psi$ radially symmetric and radially decreasing.
Then there exists a radial minimizer $\tilde{\rho} \in \mathcal{A}_M$ of $\mathcal{F}$ over $\mathcal{A}_M$.
Moreover, any minimizer of $\mathcal{F}$ over $\mathcal{A}_M$ is the translation of a radially symmetric and decreasing minimizer.
\end{lemma}

\begin{proof}
Suppose that $\tilde{\rho}$ is a minimizer of $\mathcal{F}$ over $\mathcal{A}_M$. Then
\[
\int_{\mathbb{R}^n} (\tilde{\rho}^* e(\tilde{\rho}^*))(\boldsymbol{x})\,\dd\boldsymbol{x} = \int_{\mathbb{R}^n} (\tilde{\rho} e(\tilde{\rho}))(\boldsymbol{x})\,\dd\boldsymbol{x}.
\]
where $\tilde{\rho}^*$ is the symmetric decreasing rearrangement of $\tilde{\rho}$ defined in \eqref{A.2a} in Appendix A.
\noindent
If $\Psi$ is spherically symmetric and decreasing in $r$ so that $\Psi^* = \Psi$, by Lemma \ref{rearrange}, we have
\begin{equation*}
- \int_{\mathbb{R}^n} \tilde{\rho}^*(\boldsymbol{x}) (\Psi * \tilde{\rho}^*)(\boldsymbol{x}) \dd \boldsymbol{x} \leq - \int_{\mathbb{R}^n} \tilde{\rho}(\boldsymbol{x}) (\Psi * \tilde{\rho})(\boldsymbol{x}) \dd \boldsymbol{x},
\end{equation*}
with the equality only if $\tilde{\rho} = \rho^*(\cdot - y)$ for some $y \in \mathbb{R}^n$.
Thus, $\mathcal{F}(\rho^*) \leq \mathcal{F}(\tilde{\rho})$.
Since $\tilde{\rho}$ is a minimizer of $\mathcal{F}$ over $\mathcal{A}_M$,
the equality in fact holds, $\tilde{\rho} := \rho^*$ is a radial minimizer,
and $\tilde{\rho}$ is the translation of a radial minimizer.
\end{proof}

We finally discuss further properties of the global minimizers and their uniqueness.
In Carrillo-Hoffmann-Mainini-Volzone \cite{Carrillo_2018}, they proved that,
for the Riesz potential ($\alpha \in [0,n)$) and in the case that $n \geq 1$ and $\frac{n + \alpha}{n} < \gamma < \gamma^*$
with
\begin{equation}\label{gamma*}
\gamma^* :=
\begin{cases}
    \frac{\alpha - (n - 2)}{\alpha - (n-1)} \quad & \text{for } \alpha \in (n-1,n),\\
    \infty & \text{for } \alpha \in [0,n-1],
\end{cases}
\end{equation}
all the global minimizers of $\mathcal{G}$ in $X_M$ are steady states of \eqref{0.0}.
Moreover, if $\frac{n + \alpha}{n} < \gamma \leq 2$, then any global minimizers $\varrho \in X_m$
are in the space $\varrho \in \mathcal{W}^{1,\infty}(\mathbb{R}^n)$.
Furthermore, for $n=1$, $\alpha \in (0,1)$, and $\gamma > \frac{n + \alpha}{n}$,
the global minimizers of $\mathcal{G}$ in $X_M$ and the steady states of \eqref{0.0} in $X_M$ are unique (up to translation). Similar results hold for $\alpha=0$ as proven in \cite{Carrillo_2015}, so this case is included in \eqref{gamma*}.

Calvez-Carrillo-Hoffmann \cite{Calvez_2021} proved that, for $n \geq 2$, $\alpha \in [n-2,n)$, and $\gamma > \frac{n + \alpha}{n}$,
the steady states of \eqref{0.0} in $X_M$ are unique (up to translation).
Thus, for $n \geq 2$, $\alpha \in [n-2,n)$, and $\frac{n + \alpha}{n} < \gamma < \gamma^*$,
there is a unique minimizer of $\mathcal{G}$ over $X_M$ (up to translation) that coincides with the unique steady state of \eqref{0.0} in $X_M$
(up to translation).
It was also proved that, for the critical case that $\alpha \in [n-2,n)$ and $\gamma = \frac{n + \alpha}{n}$,
there exists a unique (up to translation) minimizer of $\mathcal{G}$ over $X_{M_{\rm c}}$ and a steady state of \eqref{0.0} in $X_{M_{\rm c}}$
for some critical mass $M_{\rm c}$.

Chan-Gonz\'alez-Huang-Mainini-Volzone \cite{Chan_2020} proved that, for $n \geq 1$, $\alpha \in (0,1)$ for $n=1$,
and $\alpha \in [n-2,n)$ for $n \geq 2$, and $\frac{2n}{2n - \alpha} < \gamma < \frac{n + \alpha}{n}$,
there exist a unique radial steady state $\tilde{\rho} \in X_m \cap L^\infty(\mathbb{R}^n)$ to \eqref{0.0} in $X_m$ in the sense that
\begin{equation}\label{intstable}
e(\tilde{\rho})(\boldsymbol{x}) =\big(-\frac{\gamma - 1}{\gamma} (\Phi_\alpha* \tilde{\rho})(\boldsymbol{x}) - \mathcal{K}\big)_+
\end{equation}
is satisfied for {\it a.e.} $\boldsymbol{x} \in \mathbb{R}^n$ for some $\mathcal{K} \geq 0$.
Here, \eqref{intstable} is derived by dividing the steady form of \eqref{0.0}, in the attractive case ($\kappa = 1$),
by the density to obtain
\[
 \nabla e(\tilde{\rho})(\boldsymbol{x}) = - \frac{\gamma - 1}{\gamma} \nabla \mathcal{W}_\alpha(\boldsymbol{x}).
\]
Then, in the case that the density is regular enough, there must exist a constant $\mathcal{K}$ such that
\[
e(\tilde{\rho})(\boldsymbol{x}) = - \frac{\gamma - 1}{\gamma} (\Phi_\alpha* \tilde{\rho})(\boldsymbol{x}) - \mathcal{K}
= \big(- \frac{\gamma - 1}{\gamma} (\Phi_\alpha* \tilde{\rho})(\boldsymbol{x}) - \mathcal{K} \big)_+.
 \]
 Delgadino-Yan-Yao \cite{Delgadino_2020} considered non-local potentials of the form $W * \rho$, where $W: \mathbb{R}^n \setminus \{0\} \to \mathbb{R}$ under the following assumptions:
 \begin{enumerate}
     \item[(W1)] $W \in C^\infty(\mathbb{R}^n \setminus \{0\})$ is radially symmetric, $W$ is an attractive potential and $W'(r) > 0$ for $r>0$, where $r$ is the radial variable.
     \item[(W2)] $W$ is no more singular than the locally integrable potential $\Phi_\alpha$ at the origin, for some $\alpha < n$. That is, there exists some $C > 0$ such that $W'(r) \leq C r^{-(\alpha + 1)}$ for $r \in (0,1)$.
     \item[(W3)] There exists $C > 0$ such that $W'(r) \leq C$ for all $r \geq 1$.
     \item[(W4)] Either $W(r)$ is bounded for $r \geq 1$ or there exists some $C > 0$ such that, for all $a,b \geq 0$,
     \[
     W_+(a+b) \leq C(1 + W(1+a) + W(1+b)),
     \]
     where $W_+ := \max \{W,0\}$.
 \end{enumerate}
Under these conditions (W1)--(W4) and $\gamma \geq 2$, there is at most one steady state in $X_M \cap L^\infty(\mathbb{R}^n)$ up to translation. Note that the potential function $\Phi_\alpha$, for $\alpha \in [-1,n)$, satisfies conditions (W1)--(W4). Moreover, in \cite{Carrillo_2019}, they proved that, under the same conditions, the steady states with $W(1+|\boldsymbol{x}|)\tilde{\rho} \in L^1(\mathbb{R}^n)$ are radially decreasing.

 Thus, combining \text{\cite[Theorem 2]{Carrillo_2018}}, \text{\cite[Theorem 3]{Calvez_2021}}, \text{\cite[Proposition 5.16]{Chan_2020}}, \text{\cite[Theorem 1.1]{Delgadino_2020}}, and \text{\cite[Theorem 2.2]{Carrillo_2019}}, we obtain the following theorem.

\begin{theorem}[Existence and Uniqueness of Steady States and Global Minimizers]\label{uniqueness of steady states}
${}$
\begin{enumerate}
        \item[(i)] Suppose that $n \geq 1$, $M > 0$, $\gamma > \frac{n + \alpha}{n}$, and
        \begin{enumerate}
            \item[(a)] $\,\alpha \in (0,1)$ for $n=1$,
            \item[(b)] $\,\alpha \in [n-2,n)$ for $n \geq 2$.
        \end{enumerate}
Then the steady states of \eqref{0.0} in $X_M$, as defined in {\rm Definition \ref{steady}},
are unique {\rm (}up to translation{\rm )} and radially decreasing.
Moreover, if $\frac{n + \alpha}{n} < \gamma < \gamma^*$, as defined in \eqref{gamma*},
then the global minimizers of $\mathcal{G}$ in $X_M$ are steady states and radially decreasing, which implies the uniqueness of the global minimizers in this range.
\item[(ii)] Suppose that $n \geq 2$, $\alpha \in [n-2,n)$, and $\gamma = \frac{n + \alpha}{n}$.
Then there exists a unique minimizer of $\mathcal{G}$
that is a steady state of \eqref{0.0} in $X_{M_{\rm c}}$, which is radially decreasing.
\item[(iii)] Suppose that $n \geq 1$, $M > 0$, $\frac{2n}{2n - \alpha} < \gamma < \frac{n + \alpha}{n}$, and
        \begin{enumerate}
            \item[(a)] $\,\alpha \in (0,1)$ for $n=1$,
            \item[(b)] $\,\alpha \in [n-2,n)$ for $n \geq 2$.
        \end{enumerate}
Then there exists a unique radially decreasing steady state $\tilde{\rho} \in X_m \cap L^\infty(\mathbb{R}^n)$ to \eqref{0.0}
        in $X_m$ of \eqref{intstable}.
\item[(iv)] Suppose that $n \geq 1$, $M > 0$, $\gamma \geq 2$, and $\alpha \in [-1,n-2)$, then there exists a unique radially decreasing steady state $\tilde{\rho} \in X_m \cap L^\infty(\mathbb{R}^n)$ to \eqref{0.0} in $X_m$ of \eqref{intstable}.
\end{enumerate}
\label{exisunisteady}
\end{theorem}

In order to obtain the stability of the unique stationary solution of CEREs in the previous theorem, we need to introduce some new quantities measuring the distance between a general solution and a stationary solution.
Let $\tilde{\rho}$ be a minimizer of $\mathcal{G}$ on $X_M$, and let $\lambda$
be the constant in Theorem \ref{lmul}. For $\varrho \in X_M,$ we define the distance:
\begin{align*}
d(\varrho,\tilde{\rho})&:=\int_{\mathbb{R}^n}\big((\varrho e(\varrho)-\tilde{\rho}e(\tilde{\rho}))+(\Phi_\alpha\ast\tilde{\rho})(\varrho-\tilde{\rho})\big)\,\dd \boldsymbol{x}\\
&=\int_{\mathbb{R}^n}\big((\varrho e(\varrho)-\tilde{\rho}e(\tilde{\rho}))+(\Phi_\alpha\ast\tilde{\rho}-\lambda)(\varrho-\tilde{\rho})\big)\,\dd \boldsymbol{x}\\
&=\int_{\mathbb{R}^n}\big((\varrho e(\varrho)-\tilde{\rho}e(\tilde{\rho}))-(\tilde{\rho}e(\tilde{\rho}))_{\tilde{\rho}}(\varrho-\tilde{\rho})\big)\,\dd \boldsymbol{x},
\end{align*}
where we have used the fact that $\int_{\mathbb{R}^n} \varrho(\boldsymbol{x})\,\dd \boldsymbol{x}=\int_{\mathbb{R}^n} \tilde{\rho}(\boldsymbol{x})\,\dd \boldsymbol{x}=M$ for $\varrho,\tilde{\rho}\in X_M$ and Theorem \ref{lmul}. Notice that, by convexity of $\varrho e(\varrho)$, then $d(\varrho,\tilde{\rho})\geq0$ for any $\varrho\in X_M$.

Let us now connect this to the global finite-energy weak solutions of CEREs obtained
in the previous sections.
Denote $(\rho, \mathcal{M})(t)=(\rho, \mathcal{M})(t,\cdot)$ as given in the statement of Theorem \ref{ns}.
The energy $\mathcal{E}$ of the solution to CEREs is given by
\begin{equation*}
\mathcal{E}(\rho(t),\mathcal{M}(t))
= \int_{\mathbb{R}^n}\Big((\rho e(\rho))(t,\boldsymbol{x})
+ \frac{1}{2} \Big|\frac{\mathcal{M}}{\sqrt{\rho}}\Big|^2(t,\boldsymbol{x}) + \frac{1}{2}\rho(t,\boldsymbol{x})(\Phi_\alpha \ast \rho)(t,\boldsymbol{x})\Big)\,\dd\boldsymbol{x}.
\end{equation*}
Then, for a given solution $(\rho(t),\mathcal{M}(t))$ of CEREs with mass $M$, we have
\begin{equation}\label{Energydiff}
    \mathcal{E}(\rho(t),\mathcal{M}(t)) - \mathcal{E}(\tilde{\rho},0) = \mathcal{G}(\rho(t)) - \mathcal{G}(\tilde{\rho})
    + \frac{1}{2}\int_{\mathbb{R}^n} \Big|\frac{\mathcal{M}}{\sqrt{\rho}}\Big|^2(t,\boldsymbol{x})\,\dd\boldsymbol{x}.
\end{equation}
For $\alpha \in (0,n)$,
we have
\begin{align}\label{difference}
    \begin{split}
        \mathcal{G}(\rho(t)) - \mathcal{G}(\tilde{\rho}) & = d(\rho(t),\tilde{\rho}) + \frac{1}{2} \int_{\mathbb{R}^n} (\rho(t) - \tilde{\rho})\, \Phi_\alpha* (\rho(t) - \tilde{\rho}) \dd \boldsymbol{x}.
    \end{split}
\end{align}
Combining \eqref{Energydiff} with \eqref{difference}, we obtain
\begin{align}\label{relative}
    \begin{split}
    \mathcal{E}(\rho(t),\mathcal{M}(t)) - \mathcal{E}(\tilde{\rho},0)
    & = \mathcal{E}(\rho(t),\mathcal{M}(t)) - \mathcal{G}(\tilde{\rho}) \\
    & = d(\rho(t),\tilde{\rho}) + \frac{1}{2} \int_{\mathbb{R}^n} (\rho(t) - \tilde{\rho})\, \Phi_\alpha* (\rho(t) - \tilde{\rho}) \dd \boldsymbol{x}
    + \frac{1}{2}\int_{\mathbb{R}^n}
    \Big|\frac{\mathcal{M}}{\sqrt{\rho}}\Big|^2(t,\boldsymbol{x})\,\dd\boldsymbol{x}.
    \end{split}
\end{align}
Notice that we can estimate the second term in the last equality, which can be obtained by the H\"older inequality and the HLS inequality from Lemma \ref{simplehls}, as such
\[
\Big|\int_{\mathbb{R}^n} (\rho(t) - \tilde{\rho}) \,\Phi_\alpha* (\rho(t) - \tilde{\rho}) \dd \boldsymbol{x}\Big| \leq C_{n,\alpha} \| \rho(t) - \tilde{\rho} \|_{L^\frac{2n}{2n - \alpha}}^2.
\]

\begin{remark}
Carrillo-Castorina-Volzone {\rm \cite{Carrillo_2015}} proved that, for $n = 2$, $\alpha = 0$, and $\gamma > 1$,
there exists a unique global minimizer of $\mathcal{G}$ in $X_M$. which is the unique radially decreasing, compactly supported,
and smooth in its support steady state
of \eqref{0.0} in the sense of distributions{\rm ;} however, in {\rm \cite{Carrillo_2015}}, only the weak compactness of the minimizing sequence was obtained.
In this paper, we require the strong compactness of the minimizing sequence,
allowing us to control the relative potential energy.
For the Riesz potential {\rm (}{\it i.e.}, $\alpha \in (0,n)${\rm )},
we are able to control the relative potential energy by the relative
$L^\frac{2n}{2n - \alpha}$--norm of the densities.
However, in the logarithmic potential case, it is no longer apparent to find the natural controlling relative entropy or how we might obtain such a bound, so that
we cannot control the relative potential energy and
use \eqref{justify} and \eqref{relative} to obtain \eqref{justify2}, since
\[
\frac{1}{2} \int_{\mathbb{R}^n} (\rho(t) - \tilde{\rho})\, \Phi_0* (\rho(t) - \tilde{\rho}) \dd \boldsymbol{x}
\]
does not have a sign. Thus, we are unable to use a contradiction argument as structured in the following argument for this case.
\end{remark}

\smallskip
\begin{proof}[Proof of Theorem \ref{ns}]
We follow a similar approach as in Rein \cite{Rein_2003}, however, we take advantage of having uniqueness of global minimizers modulo translations.
Suppose the statement in Theorem \ref{ns} is not true, then there exist $\varepsilon > 0$ and a sequence of global weak solutions $(\rho^j,\mathcal{M}^j)$ of CEREs associated to the initial data $\rho^j_0$ and a sequence of times $t_j > 0$ such that
\begin{align}\label{contradict}
d(\rho^j_0,\tilde{\rho})+ \| \rho^j_0 - \tilde{\rho} \|_{L^\frac{2n}{2n - \alpha}}^2
+\frac{1}{2}\int_{\mathbb{R}^n}\Big|\frac{\mathcal{M}^j_0}{\sqrt{\rho^j_0}}\Big|^2\,\dd \boldsymbol{x}<\frac{1}{j},
\end{align}
and
\begin{align}\label{auxx}
d(\rho^j(t_j),T^{\boldsymbol{y}}\tilde{\rho})
+ \| \rho^j(t_j) - T^{\boldsymbol{y}} \tilde{\rho} \|_{L^\frac{2n}{2n - \alpha}}^2
+\frac{1}{2}\int_{\mathbb{R}^n}\Big|\frac{\mathcal{M}^j}{\sqrt{\rho^j}}\Big|^2(t_j) \,\dd \boldsymbol{x} \geq \v
\end{align}
for all $\boldsymbol{y} \in \mathbb{R}^n$. Notice that the sequence of times $t_j$ can be chosen outside the zero measure set, where the terms in \eqref{auxx} are well defined, according to Definition \ref{definition}.

Then, by \eqref{contradict} and \eqref{difference}, we have
\begin{equation*}
    \lim_{j \to \infty} \mathcal{E}(\rho^j_0,\mathcal{M}^j_0) = \mathcal{E}(\tilde{\rho},0) = \mathcal{G}(\tilde{\rho}).
\end{equation*}
Since the energy $\mathcal{E}$ for the weak solutions is non-increasing from the initial energy by definition, we have
\begin{equation*}
    \limsup_{j \to \infty} \mathcal{G}(\rho^j(t_j)) \leq \limsup_{j \to \infty} \mathcal{E}(\rho^j(t_j),\mathcal{M}^j(t_j)) \leq \lim_{j \to \infty} \mathcal{E}(\rho^j_0,\mathcal{M}^j_0) = \mathcal{G}(\tilde{\rho}).
\end{equation*}
Thus, $\rho^j(t_j) \in X_M$ is a minimizing sequence of $\mathcal{G}$ in $X_M$ and, by Theorem \ref{minimizer},
there exists $\boldsymbol{y}_j \in \mathbb{R}^n$ such that $T^{\boldsymbol{y}_j} \rho^j(t_j) \to \tilde{\rho}$
in $L^1(\mathbb{R}^n) \cap L^\frac{n + \alpha}{n}(\mathbb{R}^n)$, due to the uniqueness of $\tilde{\rho}$ up to translation.
Then
\begin{equation*}
     \| \rho^j(t_j) - T^{-\boldsymbol{y}_j} \tilde{\rho} \|_{L^\frac{2n}{2n - \alpha}}
     =  \| T^{\boldsymbol{y}_j} \rho^j(t_j) - \tilde{\rho} \|_{L^\frac{2n}{2n - \alpha}}
     \longrightarrow 0 \qquad\mbox{as $j \to \infty$},
\end{equation*}
since $\frac{2n}{2n - \alpha} \in (1, \frac{n + \alpha}{n})$.
We see that, for any $i$,
\begin{equation}\label{justify}
 \lim_{j \to \infty} \mathcal{E}(\rho^j(t_j),\mathcal{M}^j(t_j)) = \mathcal{G}(\tilde{\rho}) = \mathcal{G}(T^{-\boldsymbol{y}_i}\tilde{\rho}).
\end{equation}
Therefore, by \eqref{relative},
\begin{equation}\label{justify2}
    d(\rho^j(t_j),T^{-\boldsymbol{y}_j}\tilde{\rho})
    + \frac{1}{2}\int_{\mathbb{R}^n}    \Big|\frac{\mathcal{M}^j}{\sqrt{\rho^j}}\Big|^2(t_j)\,\dd \boldsymbol{x} \longrightarrow 0 \qquad \mbox{as $j \to \infty$}.
\end{equation}
This is a contradiction to \eqref{auxx}, which completes the proof.
\end{proof}

\begin{remark}
Suppose that $\tilde{\rho}$ is a unique minimizer of $\mathcal{G}$ in $X_M$. Without loss of generality,
suppose that $0 < \delta = \delta(\varepsilon) < \varepsilon$ satisfies {\rm Theorem \ref{ns}}. Define
\[
A_\varepsilon := \{ \boldsymbol{y} \in \mathbb{R}^n \, : \,
\text{\eqref{close}--\eqref{closer} are satisfied with $\varepsilon$ and $\delta(\varepsilon)$ for a global weak solution of CEREs} \}.
\]
Then, for $\boldsymbol{y} \in A_\varepsilon$, there exists weak solutions $(\rho,\mathcal{M})$ of CEREs such that \eqref{close}--\eqref{closer}
are satisfied for $\varepsilon$ and $\delta(\varepsilon)$:
    \[
         \| \tilde{\rho} - T^{\boldsymbol{y}} \tilde{\rho} \|_{L^\frac{2n}{2n - \alpha}}^2 \leq  \| \tilde{\rho} - \rho_0 \|_{L^\frac{2n}{2n - \alpha}}^2 + \| \rho_0 - T^{\boldsymbol{y}} \tilde{\rho} \|_{L^\frac{2n}{2n - \alpha}}^2 \leq \varepsilon + \delta(\varepsilon).
    \]
Thus, as $\varepsilon \to 0$, we have
    \[
    \sup_{\boldsymbol{y} \in A_\varepsilon} \| \tilde{\rho} - T^{\boldsymbol{y}} \tilde{\rho} \|_{L^\frac{2n}{2n - \alpha}}^2 \longrightarrow 0.
    \]
Since $\tilde{\rho} \in L^\frac{2n}{2n - \alpha}(\mathbb{R}^n)$, we must have
    \[
    \sup_{\boldsymbol{y} \in A_\varepsilon} |\boldsymbol{y}| \longrightarrow 0\qquad \mbox{as $\varepsilon \to 0$}.
    \]
That is, the closer we start to $\tilde{\rho}$, the closer we remain to $\tilde{\rho}$, not just a translation of $\tilde{\rho}$.
\end{remark}

\begin{remark}
In case we no longer have the uniqueness of the minimizers, as in {\rm \cite{Rein_2003}},
we can define the set of minimizers of $\mathcal{G}$ in $X_M$ by
\[
\mathcal{M}_M := \big\{ \tilde{\rho} \in X_M \, : \, \mathcal{G}(\tilde{\rho}) = g_M \big\},
\]
and prove the following corollary{\rm :}

\begin{corollary}[Nonlinear Stability of CEREs]
Suppose that $\alpha \in (0,n - 1)$ with $\gamma > \frac{n + \alpha}{n}$.
Let $(\rho,\mathcal{M})(t, \boldsymbol{x})$ be a global weak solution of the attractive CEREs {\rm (}$\kappa = 1${\rm )}
in the sense of {\rm Definition \ref{definition}} satisfying
\[
\int_{\mathbb{R}^n} \rho_0(\boldsymbol{x})\,\dd \boldsymbol{x}=\int_{\mathbb{R}^n} \tilde{\rho}(\boldsymbol{x})\,\dd \boldsymbol{x}=M.
\]
Then, for any $\v>0,$ there exists $\delta = \delta(\varepsilon)>0$ such that, if
\begin{align*}
\inf_{\tilde{\rho} \in \mathcal{M}_M}
\bigg\{ d(\rho_0,\tilde{\rho})+ \| \rho_0 - \tilde{\rho} \|_{L^\frac{2n}{2n - \alpha}}^2
+\frac{1}{2}\int_{\mathbb{R}^n} \Big|\frac{\mathcal{M}_0}{\sqrt{\rho_0}}\Big|^2\,\dd \boldsymbol{x} \bigg\} <\delta,
\end{align*}
there exists a translation $\boldsymbol{y}\in \R^n$ such that
\begin{align*}
\inf_{\tilde{\rho} \in \mathcal{M}_M}
\bigg\{ d(\rho(t),\tilde{\rho})+ \| \rho(t) - \tilde{\rho} \|_{L^\frac{2n}{2n - \alpha}}^2
+\frac{1}{2}\int_{\mathbb{R}^n} \Big|\frac{\mathcal{M}}{\sqrt{\rho}}\Big|^2(t)\,\dd \boldsymbol{x} \bigg\} <\v
\qquad\mbox{for a.e. $t>0$}.
\end{align*}
Furthermore, let $\tilde{\rho} \in \mathcal{M}_M$ be an isolated
minimizer {\rm (}up to translation{\rm )} of functional $\mathcal{G}$ in $X_M$, that is,
\[
\delta_0 := \inf \Big\{ \| \rho - \tilde{\rho} \|_{L^\frac{2n}{2n - \alpha}}^2 \,  :\,
\rho \in \mathcal{M}_M \setminus \{ T^{\boldsymbol{y}}\tilde{\rho} \,: \, \boldsymbol{y} \in \mathbb{R}^n \} \Big\}>0,
\]
and let
\begin{equation}\label{aux}
(t,\boldsymbol{y}) \longmapsto d(\rho(t),T^{\boldsymbol{y}}\tilde{\rho})
+ \| \rho(t) - T^{\boldsymbol{y}}\tilde{\rho} \|_{L^\frac{2n}{2n - \alpha}}^2
+\frac{1}{2}\int_{\mathbb{R}^n} \Big|\frac{\mathcal{M}}{\sqrt{\rho}}\Big|^2(t)\,\dd \boldsymbol{x}
\end{equation}
be continuous. Then, for any $0 < \v < \frac{\delta_0}{4}$, there exists $0<\delta = \delta(\varepsilon)<\varepsilon$ such that,
if
\begin{align*}
d(\rho_0,\tilde{\rho})+ \| \rho_0 - \tilde{\rho} \|_{L^\frac{2n}{2n - \alpha}}^2
+\frac{1}{2}\int_{\mathbb{R}^n} \Big|\frac{\mathcal{M}_0}{\sqrt{\rho_0}}\Big|^2\,\dd \boldsymbol{x}<\delta,
\end{align*}
there exists a translation $\boldsymbol{y}\in \R^n$ such that
\begin{align*}
d(\rho(t),T^{\boldsymbol{y}}\tilde{\rho})
+ \| \rho(t) - T^{\boldsymbol{y}}\tilde{\rho} \|_{L^\frac{2n}{2n - \alpha}}^2
+\frac{1}{2}\int_{\mathbb{R}^n} \Big|\frac{\mathcal{M}}{\sqrt{\rho}}\Big|^2(t)\,\dd \boldsymbol{x}<\v
\qquad\mbox{for all $t>0$}.
\end{align*}
\end{corollary}
Notice that this result needs the additional continuity assumption of \eqref{aux}{\rm ;} see {\rm \cite{Rein_2003}} for details.
\end{remark}

\medskip
\appendix
\section{Convolution Inequalities}\label{A1}
We now present several convolution inequalities that have been used in the main text of this paper.

\begin{theorem}[\text{\cite[Theorem II.11.2]{Galdi}}]\label{pointwise}
Suppose that $\Omega \subset \mathbb{R}^n$ is bounded, $\alpha \in (0,n)$, and $f \in L^q(\Omega)$ for $q \in (1,\infty)$.
If $\alpha < n(1 - \frac{1}{q})$,
then $|\cdot|^{-\alpha} * (\mathds{1}_\Omega f) \in C^{0,\mu}(\overline{\Omega})$
for $\mu:= \min\big\{1,n(1-\frac{1}{q})-\alpha\big\}$
and there exists some constant $C = C(\Omega,n,q,\alpha)$ such that
    \begin{equation*}
    \big\| |\cdot|^{-\alpha} * (\mathds{1}_\Omega f)\big\|_{C^{0,\mu}(\overline{\Omega})} \leq C \|f\|_{L^q(\Omega)}.
    \end{equation*}
\end{theorem}

\begin{lemma}[Hardy-Littlewood-Sobolev (HLS) Inequality]\label{simplehls}
	Let $\alpha \in (0,n)$ and $1 < p,r < \infty$ such that
	\[
	\frac{1}{p} + \frac{\alpha}{n} = 1 + \frac{1}{r}.
	\]
	Then,
    for all $f \in L^p(\mathbb{R}^n)$,
	\[
	\big\| |\cdot |^{-\alpha} * f \big\|_{L^r(\mathbb{R}^n)} \leq C_{n,\alpha}\|f\|_{L^p(\mathbb{R}^n)},
	\]
   where
    \[
    C_{n,\alpha} = \pi^\frac{\alpha}{2} \frac{\Gamma(\frac{n - \alpha}{2})}{\Gamma( n - \frac{\alpha}{2})}
    \bigg( \frac{\Gamma(\frac{n}{2})}{\Gamma(n)} \bigg)^{\frac{\alpha - n}{n}}.
    \]
\end{lemma}

\begin{lemma}[Variation of HLS Inequality \text{\cite[Theorem 3.1]{Calvez_2017}}]\label{HLSineq}
Suppose that $\alpha \in (0,n)$ and $p \geq \frac{2n}{2n - \alpha}$.
Then there exists $C_* = C_*(\alpha,p,n) > 0$ such that, for any $f \in L^1(\mathbb{R}^n) \cap L^p(\mathbb{R}^n)$,
\begin{equation}\label{HLS}
\bigg| \int_{\mathbb{R}^n} \int_{\mathbb{R}^n} f(\boldsymbol{x}) |\boldsymbol{x} - \boldsymbol{y}|^{-\alpha} f(\boldsymbol{y}) \dd \boldsymbol{x} \dd \boldsymbol{y} \bigg| \leq C_* \| f \|_{L^1}^\frac{p (2n - \alpha) - 2n}{n(p - 1)} \| f \|_{L^p}^\frac{\alpha p}{n(p - 1)}.
\end{equation}
In particular, when $p = \frac{n + \alpha}{n} > \frac{2n}{2n - \alpha}$, for any
$f \in L^1(\mathbb{R}^n) \cap L^{\frac{n + \alpha}{n}}(\mathbb{R}^n)$,
\[
\bigg| \int_{\mathbb{R}^n} \int_{\mathbb{R}^n} f(\boldsymbol{x}) |\boldsymbol{x} - \boldsymbol{y}|^{-\alpha} f(\boldsymbol{y}) \dd \boldsymbol{x} \dd \boldsymbol{y} \bigg|
\leq C_* \| f \|_{L^1}^\frac{n - \alpha}{n} \| f \|_{L^{\frac{n + \alpha}{n}}}^\frac{n + \alpha}{n}.
\]
Furthermore, $C_*$ is the sharp constant that satisfies \eqref{HLS} and is strictly bounded above by the standard HLS constant $C_{n,\alpha}$, i.e., $C_* < C_{n,\alpha}$.
\end{lemma}

\begin{lemma}[Riesz's Rearrangement Inequality \cite{Lieb_2001}]\label{rearrange}
    Let $f, g, h : \mathbb{R}^n \to \mathbb{R}_+$ be measurable functions. Define
    \[
    I(f,g,h):= \int_{\mathbb{R}^n} \int_{\mathbb{R}^n} f(\boldsymbol{x}) g(\boldsymbol{x} - \boldsymbol{y}) h(\boldsymbol{y}) \dd \boldsymbol{x} \dd \boldsymbol{y}.
    \]
    For a given measurable function $\phi : \mathbb{R}^n \to \mathbb{R}_+$, define the symmetric decreasing rearrangement $\phi^*$ of $\phi$ is
    \begin{equation}\label{A.2a}
    \phi^* (\boldsymbol{x}) := \int^\infty_0 \mathds{1}_{\{\boldsymbol{y} \, : \, \phi(\boldsymbol{y}) > t\}^*} (\boldsymbol{x}) \dd t
    \end{equation}
    with
    \[
    A^* = \{ \boldsymbol{x} \, : \, \omega_n |\boldsymbol{x}|^n \leq |A|\}
    \]
    for any Lebesgue measurable set $A \subset \mathbb{R}^n$. Then
    \begin{equation}\label{rearrangeineq}
    I(f,g,h) \leq I(f^*,g^*,h^*).
    \end{equation}
    If $g$ is a strictly symmetric-decreasing function, then the equality holds in \eqref{rearrangeineq}
    if and only if there exists some $\boldsymbol{z} \in \mathbb{R}^n$ such that $f(\boldsymbol{x}) = f^*(\boldsymbol{x} - \boldsymbol{z})$ and $h(\boldsymbol{x}) = h^*(\boldsymbol{x} - \boldsymbol{z})$.
\end{lemma}

\begin{theorem}[The Riesz Composition Formula \text{\cite[Theorem 3.1]{Plessis_1970}}]\label{composition}
 Suppose that $0 < \alpha < n$, $0 < \beta < n$, and $0 < \alpha + \beta < n$. Then
\begin{equation*}
\int_{\mathbb{R}^n} |\boldsymbol{x} - \boldsymbol{z}|^{\alpha - n} |\boldsymbol{z} - \boldsymbol{y}|^{\beta - n} \dd \boldsymbol{z}
= \kappa_{\alpha,\beta}\,|\boldsymbol{x} - \boldsymbol{y}|^{\alpha + \beta - n},
\end{equation*}
where
\begin{equation}\label{A.4a}
 \kappa_{\alpha,\beta}
 = \frac{\pi^{\frac{n}{2}} \Gamma(\frac{\alpha}{2}) \Gamma(\frac{\beta}{2}) \Gamma(\frac{n-\alpha-\beta}{2})}{\Gamma(\frac{n-\alpha}{2}) \Gamma(\frac{n - \beta}{2}) \Gamma(\frac{\alpha + \beta}{2})}.
 \end{equation}
\end{theorem}

\medskip
\section{Construction of the Approximate Initial Data}\label{B1}
Similar to \cite{Chen2021}, we approximate the initial density by a density function away from the vacuum,
which means that our approximate solutions can also be away from the vacuum, giving us simpler calculations,
as we consider the equations in terms of the velocity and the density, rather than the momentum and the density.

\smallskip
Consider a standard radially symmetric mollifier $J(\boldsymbol{x})$ and the associated mollifier:
$$
J_\delta(\boldsymbol{x}) := \frac{1}{\delta^n}J(\frac{\boldsymbol{x}}{\delta}).
$$
We now define the approximation:
\begin{equation*}
    \tilde{\rho}^\varepsilon_0(\boldsymbol{x})
    := \bigg( \int_{\mathbb{R}^n} \sqrt{\rho_0(\boldsymbol{x}- \boldsymbol{y})} J_{\sqrt{\varepsilon}}(\boldsymbol{y}) \dd \boldsymbol{y} + \varepsilon e^{-|\boldsymbol{x}|^2} \bigg)^2.
\end{equation*}
We have defined the approximation as such in order to have a nice formula for $\sqrt{\tilde{\rho}^\varepsilon_0}(\boldsymbol{x})$ with the first part smoothing the square root of the density and the second part avoiding any vacuum states that might occur in the initial density.
Then, as in \cite{Chen2021}, we normalize to ensure that the mass of the approximating initial density remains constant as $\varepsilon$ varies.
We define the initial density approximation:
\begin{equation}\label{approx1}
    \rho^\varepsilon_0(\boldsymbol{x}) := \frac{M}{\int_{\mathbb{R}^n} \tilde{\rho}^\varepsilon_0(\boldsymbol{x}) \dd \boldsymbol{x}} \tilde{\rho}^\varepsilon_0(\boldsymbol{x}).
\end{equation}

Then, similar to \cite{Chen2021}, we have the following lemma:
\begin{lemma}\label{appendix1}
    Let $q \in \{1,\gamma\}$ for $\kappa = 1$ and $q \in \{1,\gamma,\frac{2n}{2n - \alpha}\}$ for $\kappa = -1$. Then
    \begin{align}\label{epcon}
        \begin{split}
		&\int_{\R^n}\rho_0^\v(\boldsymbol{x})\,\dd\boldsymbol{x}=M, \quad
        \|\rho_0^\v\|_{L^q}\leq C\big(\|\rho_0\|_{L^q}+1\big)
		\qquad\,\,\mbox{for all $\v\in(0,1]$},\\
		&\lim_{\v\rightarrow0} \big(\|\rho_0^\v-                     \rho_0\|_{L^q}+\|\sqrt{\rho_0^\v}-  \sqrt{\rho_0}\|_{L^{2q}}\big)=0,\\
		&\v^2\int_{\R^n}\Big|\nabla_{\boldsymbol{x}}\sqrt{\rho^\v_{0} (\boldsymbol{x})}\Big|^2 \dd\boldsymbol{x}
		\leq C\v \big(\|\rho_0\|_{L^1}+1\big) \longrightarrow     0\qquad\mbox{as $\v\rightarrow0$}.
            \end{split}
	\end{align}
\end{lemma}
In the case that $\alpha \in (-1,0]$,
$\Phi_\alpha * \rho^\varepsilon_0$ is well-defined:
\begin{align}\label{boundedlog}
   & \int_{\mathbb{R}^n} \rho^\varepsilon_0(\boldsymbol{x}) |k_\alpha(|\boldsymbol{x} - \boldsymbol{y}|)|
     \dd \boldsymbol{y}\nonumber\\
        & = \frac{M}{\int_{\mathbb{R}^n} \tilde{\rho}^\varepsilon_0(\boldsymbol{y}) \dd \boldsymbol{y}}
        \int_{\mathbb{R}^n} \tilde{\rho}^\varepsilon_0(\boldsymbol{y}) |k_\alpha(|\boldsymbol{x} - \boldsymbol{y}|)|
        \dd \boldsymbol{y} \nonumber\\
        & \leq \frac{2M}{\int_{\mathbb{R}^n} \tilde{\rho}^\varepsilon_0(\boldsymbol{y}) \dd \boldsymbol{y}}
        \int_{\mathbb{R}^n} \Big( |\sqrt{\rho_0} * J_{\sqrt{\varepsilon}}|^2(\boldsymbol{y}) + \varepsilon^2 e^{-2|\boldsymbol{x}
        - \boldsymbol{y}|^2} \Big) |k_\alpha(|\boldsymbol{y}|)| \dd \boldsymbol{y}\nonumber\\
        & \leq \frac{2M}{\int_{\mathbb{R}^n} \tilde{\rho}^\varepsilon_0(\boldsymbol{y}) \dd \boldsymbol{y}}
        \int_{\mathbb{R}^n} \bigg( \Big|\int_{\mathbb{R}^n} \sqrt{\rho_0}(\boldsymbol{y} - \boldsymbol{z}) J_{\sqrt{\varepsilon}}(\boldsymbol{z}) \dd \boldsymbol{z} \Big|^2 + \varepsilon^2 e^{-2|\boldsymbol{y}|^2} \bigg)
        |k_\alpha(|\boldsymbol{x} - \boldsymbol{y}|)| \dd \boldsymbol{y}\nonumber\\
        & \leq \frac{2M}{\int_{\mathbb{R}^n} \tilde{\rho}^\varepsilon_0(\boldsymbol{y}) \dd \boldsymbol{y}}
        \bigg( \Big( \int_{\mathbb{R}^n} J_{\sqrt{\varepsilon}} (\boldsymbol{z}) \sqrt{(|k_\alpha( |\cdot|)| * \rho_0)(\boldsymbol{x} - \boldsymbol{z})} \dd \boldsymbol{z} \Big)^2 + \int_{\mathbb{R}^n} \varepsilon^2 e^{-2|\boldsymbol{y}|^2}\,
        |k_\alpha(|\boldsymbol{x} - \boldsymbol{y}|)| \dd \boldsymbol{y} \bigg)\nonumber\\[1mm]
        & < \infty,
\end{align}
where we have used the Minkowski inequality to obtain the penultimate inequality.
We can also prove the following convergence result for the kernel:

\begin{lemma}\label{kernel1}
 $\Phi_\alpha*\rho^{\varepsilon}_0$ with $\alpha \in (0,n-1)$ satisfies the following{\rm :}
    \begin{equation*}
    \begin{split}
        & \|\Phi_\alpha*\rho_0^\v - \Phi_\alpha * \rho_0\|_{L^{\frac{2n}{\alpha}}} \leq C \|\rho^\varepsilon_0 - \rho_0\|_{L^\frac{2n}{2n - \alpha}} \longrightarrow 0 \qquad \text{ as } \v \to 0, \\[2mm]
        & \| \nabla (\Phi_\alpha*\rho^\varepsilon_0 - \Phi_\alpha * \rho_0)\|_{L^\frac{2n}{\alpha + 2}} \leq C \|\rho^\varepsilon_0 - \rho_0\|_{L^\frac{2n}{2n - \alpha}} \longrightarrow 0 \qquad \text{ as } \v \to 0.
    \end{split}
    \end{equation*}
    For $\alpha \in (-1,0]$ and $K \Subset \mathbb{R}^n$,
    \begin{equation}\label{logcon}
        \|\Phi_\alpha*\rho_0^\v - \Phi_\alpha * \rho_0\|_{L^\infty(K)}
        + \| \nabla (\Phi_\alpha*\rho^\varepsilon_0 - \Phi_\alpha * \rho_0)\|_{L^\frac{2n}{2n - \alpha}(K)}
        \longrightarrow 0 \qquad \text{ as } \v \to 0.
    \end{equation}
    Moreover, there exists a constant $C > 0$ such that, for $\varepsilon > 0$ sufficiently small,
    \begin{equation}\label{epinitial1}
        \int_{\mathbb{R}^n} \rho_0^\varepsilon(\boldsymbol{x}) k_\alpha(1 + |\boldsymbol{x}|^2) \dd \boldsymbol{x} \leq C.
    \end{equation}
\end{lemma}

\begin{proof}
The first part is an application of the HLS inequality from Lemma \ref{simplehls}. For simplicity, we may assume that $K = B_D$ for some $D > 0$.
Then, for ${\hat{\sigma}} > 0$,
        \begin{align*}
        \begin{split}
            & \big\|\Phi_\alpha*\rho_0^\v - \Phi_\alpha * \rho_0 \big\|_{L^\infty(K)} \\
            & \leq \big\|(\mathds{1}_{B_{\hat{\sigma}}}|\Phi_\alpha|)*(\rho_0^\v - \rho_0)
  \big\|_{L^\infty(K)}
            + \big\|(\mathds{1}_{B_{\hat{\sigma}}^c}|\Phi_\alpha|)*(\rho_0^\v - \rho_0) \big\|_{L^\infty(K)} \\
            & =: I_1 + I_2.
        \end{split}
    \end{align*}
Since $\mathds{1}_{B_{\hat{\sigma}}}\Phi_\alpha \in L^\frac{\gamma}{\gamma - 1}$,
then, for $I_1$, by the H\" older inequality,
    \begin{equation}\label{I1}
        I_1 \leq \|\rho_0^\v - \rho_0\|_{L^\gamma} \|\mathds{1}_{B_{\hat{\sigma}}}\Phi_\alpha\|_{L^\frac{\gamma}{\gamma - 1}} \longrightarrow 0
         \qquad \text{ as } \varepsilon \to 0.
    \end{equation}

We now want to show that $\| (\mathds{1}_{B_{\hat{\sigma}}^c} |\Phi_\alpha|)*\rho_0 \|_{L^\infty(B_{D + 1})} \to 0$ as ${\hat{\sigma}} \to \infty$.
We take ${\hat{\sigma}}$ large enough so that, for all $\boldsymbol{x} \in B_{D + 1}$ and $\boldsymbol{y} \in B_{{\hat{\sigma}} - D - 1}^c$,
$1 \leq |\boldsymbol{x}-\boldsymbol{y}| \leq \min\{2|\boldsymbol{y}|,|\boldsymbol{y}|^2\}$.
Then, for any $\boldsymbol{x} \in B_{D + 1}$,
\begin{align*}
\begin{split}
& \big|(\mathds{1}_{B_{\hat{\sigma}}^c} |\Phi_\alpha|)*\rho_0\big|(\boldsymbol{x}) \\
                & = \int_{\mathbb{R}^n} \rho_0(\boldsymbol{y}) \mathds{1}_{B_{\hat{\sigma}}^c}(\boldsymbol{x} - \boldsymbol{y}) |k_\alpha(|\boldsymbol{x} - \boldsymbol{y}|)| \dd \boldsymbol{y}  \\
                & \leq \int_{B_{{\hat{\sigma}} - D - 1}^c} \rho_0(\boldsymbol{y}) |k_\alpha(|\boldsymbol{x} - \boldsymbol{y}|)| \dd \boldsymbol{y} \\
                & \leq C \int_{B_{{\hat{\sigma}} - D - 1}^c} \rho_0(\boldsymbol{y}) |k_\alpha(|\boldsymbol{y}|)| \dd \boldsymbol{y}.
\end{split}
\end{align*}
This implies
\begin{equation*}
\big\| (\mathds{1}_{B_{\hat{\sigma}}^c} |\Phi_\alpha|)*\rho_0\big\|_{L^\infty(B_{D + 1})}
\leq C \int_{B_{{\hat{\sigma}} - D - 1}^c} \rho_0(\boldsymbol{y}) |k_\alpha(|\boldsymbol{y}|)|
\dd \boldsymbol{y} \longrightarrow 0 \qquad \text{ as ${\hat{\sigma}} \to \infty$}.
\end{equation*}

By the Lebesgue dominated convergence theorem, we obtain
 \[
 \int_{\mathbb{R}^n} \tilde{\rho}^\varepsilon_0(\boldsymbol{x}) \dd \boldsymbol{x} \to M \qquad \text{ as $\v \to 0$}.
 \]
 Then, for $\varepsilon$ small enough,
 $\int_{\mathbb{R}^n} \tilde{\rho}^\varepsilon_0(\boldsymbol{x}) \dd \boldsymbol{x} \geq \frac{M}{2}$.
 Moreover, for $\varepsilon \in (0,1)$ small enough, if $\boldsymbol{x} \in K$, by a similar calculation to \eqref{boundedlog}, we have
 \begin{align*}
\begin{split}
 & \big|(\mathds{1}_{B_{\hat{\sigma}}^c}|\Phi_\alpha|)*\rho_0^\v\big|(\boldsymbol{x}) \\
            & \leq \frac{2M}{\int_{\mathbb{R}^n} \tilde{\rho}^\varepsilon_0(\boldsymbol{y}) \dd \boldsymbol{y}}
            \bigg( \Big( \int_{\mathbb{R}^n} J_{\sqrt{\varepsilon}} (\boldsymbol{z}) \sqrt{((\mathds{1}_{B_{\hat{\sigma}}^c}|\Phi_\alpha|) * \rho_0)(\boldsymbol{x} - \boldsymbol{z})} \dd \boldsymbol{z} \Big)^2
            + \int_{\mathbb{R}^n} \varepsilon^2 e^{-2|\boldsymbol{y}|^2} (\mathds{1}_{B_{\hat{\sigma}}^c}|\Phi_\alpha|)(\boldsymbol{x} - \boldsymbol{y}) \dd \boldsymbol{y} \bigg) \\
            & \leq 4 \Big(\big\| (\mathds{1}_{B_{\hat{\sigma}}^c} |\Phi_\alpha|)*\rho_0\big\|_{L^\infty(B_{D + 1})}
            + \big\| (\mathds{1}_{B_{\hat{\sigma}}^c} |\Phi_\alpha|)*e^{-2|\,\cdot\,|^2} \big\|_{L^\infty(B_{D + 1})} \Big).
            \end{split}
\end{align*}
This yields
\begin{equation}\label{I2}
        I_2 \leq 5 \big\| (\mathds{1}_{B_{\hat{\sigma}}^c} |\Phi_\alpha|)*\rho_0 \big\|_{L^\infty(B_{D + 1})}
         + 4 \big\| (\mathds{1}_{B_{\hat{\sigma}}^c} |\Phi_\alpha|)*e^{-2|\cdot|^2} \big\|_{L^\infty(B_{D + 1})}
         \longrightarrow 0 \qquad \text{ as ${\hat{\sigma}} \to \infty$}.
    \end{equation}
Thus, combining \eqref{I1} with \eqref{I2}, we obtain
\begin{equation*}
\big\|\Phi_\alpha*\rho_0^\v - \Phi_\alpha * \rho_0 \big\|_{L^\infty(K)} \longrightarrow 0
\qquad \text{ as $\varepsilon \to 0$}.
\end{equation*}

For the final part, take $1 < p < \min\{\frac{2n}{\alpha + 2},\gamma\}$, by applying the HLS inequality from Lemma \ref{simplehls}, the interpolation argument,
and \eqref{eq4}, we have
\begin{align*}
   \begin{split}
            & \big\|\nabla (\Phi_\alpha*\rho^\varepsilon_0 - \Phi_\alpha * \rho_0)\big\|_{L^\frac{2n}{\alpha + 2}(K)} \\
            & \leq \big\| (|\cdot|^{-(\alpha + 1)} \mathds{1}_{B_1}) * (\rho^\varepsilon_0 - \rho_0)\big\|_{L^\frac{2n}{\alpha + 2}(K)}
            + \big\| (|\cdot|^{-(\alpha + 1)} \mathds{1}_{B_1^c}) * (\rho^\varepsilon_0 - \rho_0)\big\|_{L^\frac{2n}{\alpha + 2}(K)} \\
            & \leq \big\| (|\cdot|^{-(1 + \alpha/2 + n(1-1/p))}) * (\rho^\varepsilon_0 - \rho_0)\big\|_{L^\frac{2n}{\alpha + 2}(K)}
            + \big\| \mathds{1}_{B_1^c} * (\rho^\varepsilon_0 - \rho_0)\big\|_{L^\frac{2n}{\alpha + 2}(K)} \\
            & \leq C_{n,1 + \alpha/2 + n(1-1/p)} \|\rho^\varepsilon_0 - \rho_0\|_{L^p(\mathbb{R}^n)} + |K|^{\frac{\alpha + 2}{2n}} \|\rho^\varepsilon_0 - \rho_0\|_{L^1(\mathbb{R}^n)} \\
            & \leq C_{n,1 + \alpha/2 + n(1-1/p)} \|\rho^\varepsilon_0 - \rho_0\|_{L^1(\mathbb{R}^n)}^{\frac{\gamma - p}{p(\gamma - 1)}} \|\rho^\varepsilon_0 - \rho_0\|_{L^\gamma(\mathbb{R}^n)}^{\frac{\gamma(p-1)}{p(\gamma - 1)}} + |K|^{\frac{\alpha + 2}{2n}} \|\rho^\varepsilon_0 - \rho_0\|_{L^1(\mathbb{R}^n)}.
        \end{split}
    \end{align*}
Thus, we conclude
    \begin{equation*}
        \| \nabla (\Phi_\alpha*\rho^\varepsilon_0 - \Phi_\alpha * \rho_0)\|_{L^\frac{2n}{\alpha + 2}(K)} \longrightarrow 0
        \qquad \text{ as $\varepsilon \to 0$}.
    \end{equation*}

Now, by using the initial assumption \eqref{0.81}, there exists $C>0$ such that
    \[
    \int_{\mathbb{R}^n} \rho_0(\boldsymbol{x}) k_\alpha(1 + |\boldsymbol{x}|^2) \dd \boldsymbol{x} \leq C.
    \]
Thus, applying the H\"older inequality, we have
 \begin{align*}
        \begin{split}
            & \int_{\mathbb{R}^n} \rho_0^\varepsilon(\boldsymbol{x}) k_\alpha(1 + |\boldsymbol{x}|^2) \dd \boldsymbol{x} \\
            & \leq \int_{\mathbb{R}^n} \rho_0(\boldsymbol{x}) k_\alpha(1 + |\boldsymbol{x}|^2) \dd \boldsymbol{x}
             + \Big|\int_{\mathbb{R}^n} \big(\rho_0^\varepsilon(\boldsymbol{x}) - \rho_0(\boldsymbol{x})\big) k_\alpha(1 + |\boldsymbol{x}|^2) \dd \boldsymbol{x} \Big| \\
            & \leq C + \Big|\int_{B_1} \big(\rho_0^\varepsilon(\boldsymbol{x}) - \rho_0(\boldsymbol{x})\big) k_\alpha(1 + |\boldsymbol{x}|^2) \dd \boldsymbol{x} \Big|
            + \Big|\int_{B_1^c} \big(\rho_0^\varepsilon(\boldsymbol{x}) - \rho_0(\boldsymbol{x})\big) k_\alpha(1 + |\boldsymbol{x}|^2) \dd \boldsymbol{x} \Big|\\
            & \leq C + k_\alpha(2) \|\rho_0^\varepsilon - \rho_0\|_{L^1(B_1)}
            + \Big|\int_{B_1^c} \big(\rho_0^\varepsilon(\boldsymbol{x}) - \rho_0(\boldsymbol{x})\big) k_\alpha(2|\boldsymbol{x}|^2) \dd \boldsymbol{x} \Big|\\
            & \leq C + k_\alpha(2) \|\rho_0^\varepsilon - \rho_0\|_{L^1(B_1)} + k_\alpha(2) \|\rho_0^\varepsilon - \rho_0\|_{L^1(B_1^c)}
            + C \Big|\int_{B_1^c} \big(\rho_0^\varepsilon(\boldsymbol{x}) - \rho_0(\boldsymbol{x})\big) k_\alpha(|\boldsymbol{x}|^2) \dd \boldsymbol{x}\Big|\\
            & = C\big( 1 + \|\rho_0^\varepsilon - \rho_0\|_{L^1(\mathbb{R}^n)} + |\Phi_\alpha * (\rho_0^\varepsilon - \rho_0)|(\boldsymbol{0})\big).
        \end{split}
\end{align*}
Now, by \eqref{epcon} and \eqref{logcon}, we have
    \[
    \|\rho_0^\varepsilon - \rho_0\|_{L^1(\mathbb{R}^n)} + |\Phi_\alpha * (\rho_0^\varepsilon - \rho_0)|(\boldsymbol{0}) \longrightarrow 0
    \qquad \mbox{as $\varepsilon \to 0$}.
    \]
Then there exists some $C>0$ such that, for small enough $\varepsilon$, \eqref{epinitial1} holds.
\end{proof}

As in the previous literature, in order to prove the BD entropy and the convergence of $\Omega^T$ to $[0,T] \times (0,\infty)$
in some sense,
we require that
\begin{equation}\label{b}
    \rho_0^{\varepsilon,b}(b) \cong b^{-(n-\beta)},
\end{equation}
where $\beta = \min \{ \frac{1}{2}, (1-\frac{1}{\gamma})n\}$.
To do this, we choose a cut-off function $S \in C^\infty(\mathbb{R})$, such that
\begin{equation*}
    \begin{cases}
    S(z) = 0 \quad \text{ if } z \in (-\infty,0], \\
    S(z) = 1 \quad \text{ if } z \in [1,\infty), \\
    S(z) \text{ is monotonic in } [0,1]. \\
\end{cases}
\end{equation*}
Now, as in \cite{Chen2021}, we define
\begin{equation*}
    \tilde{\rho}^{\varepsilon,b}_0(\boldsymbol{x}) :=
    \Big\{ \sqrt{\rho_0^{\varepsilon}(\boldsymbol{x})}\big(1 - S(2(|\boldsymbol{x}| - (b-1)))\big)
    + b^{-\frac{n-\beta}{2}}S(2(|\boldsymbol{x}| - (b-1))) \Big\}^2.
\end{equation*}
Then $\tilde{\rho}^{\varepsilon,b}_0$ is a radial function and satisfies $\tilde{\rho}^{\varepsilon,b}_0(b) = b^{-(n-\beta)}$.
Now, as before, we normalize $\tilde{\rho}^{\varepsilon,b}_0$ (but on the corresponding domain from problem \eqref{Eul}--\eqref{0.20}: the annulus $\{b^{-1} \leq |\boldsymbol{x}| \leq b\}$) to produce
\begin{equation}\label{approx2}
    \rho^{\varepsilon,b}_0(\boldsymbol{x}) :=\frac{M}{\int_{b^{-1} \leq |\boldsymbol{x}| \leq b} \tilde{\rho}^{\varepsilon,b}_0(\boldsymbol{x}) \dd \boldsymbol{x}} \mathds{1}_{\{b^{-1} \leq |\boldsymbol{x}| \leq b \}}(\boldsymbol{x}) \tilde{\rho}^{\varepsilon,b}_0(\boldsymbol{x}).
\end{equation}
As before, it is direct to see that $\eqref{b}$ is satisfied.
Similar to \cite{Chen2021}, one can show the following lemma:
\begin{lemma}\label{appendix2}
The defined function $\rho^{\varepsilon,b}_0(\boldsymbol{x})$ in \eqref{approx2}  is smooth.
Moreover, for any $q \in \{1,\gamma\}$ when $\kappa = 1$ and $q \in \{1,\gamma,\frac{2n}{2n - \alpha}\}$ when $\kappa = -1$,
the following estimates hold{\rm :}
    \begin{align*}
        \begin{split}
		&\int_{b^{-1}\leq |\boldsymbol{x}|\leq b} \rho_{0}^{\v,b}(\boldsymbol{x})\,\dd\boldsymbol{x}= M \qquad\mbox{for all $\v\in (0,1]$ and $b>1$},\\
		&\int_{|\boldsymbol{x}|\leq b} \Big(\big|\rho_0^{\v,b}(\boldsymbol{x})-\rho_0^\v(\boldsymbol{x})\big|^q
		+\Big| \sqrt{\rho_{0}^{\v,b}(\boldsymbol{x})}
		-\sqrt{\rho_{0}^{\v}(\boldsymbol{x})}\Big|^{2q}\Big)\,\dd \boldsymbol{x}\rightarrow0 \qquad\mbox{as $b\to \infty$},\\
		&\v\int_{|\boldsymbol{x}|\leq b}\Big|\nabla_{\boldsymbol{x}}\sqrt{\rho_0^{\v,b}(\boldsymbol{x})}\Big|^2 \dd\boldsymbol{x}
		\leq C\big(\|\rho_0\|_{L^1}+1\big).
        \end{split}
    \end{align*}
\end{lemma}
We also have the convergence of the initial potentials, represented by the following lemma:

\begin{lemma}\label{kernel2}
$\Phi_\alpha*\rho^{\varepsilon,b}_0$ with $\alpha \in (0,n-1)$ satisfy
    \begin{equation*}
    \begin{split}
        & \|\Phi_\alpha*\rho_0^{\v,b} - \Phi_\alpha * \rho_0^\v \|_{L^{\frac{2n}{\alpha}}} \leq C \|\rho^{\varepsilon,b}_0 - \rho^\varepsilon_0\|_{L^\frac{2n}{2n - \alpha}} \longrightarrow 0 \qquad \text{ as $b \to \infty$}, \\
        & \| \nabla (\Phi_\alpha*\rho^{\varepsilon,b}_0 - \Phi_\alpha * \rho^\varepsilon_0)\|_{L^\frac{2n}{\alpha + 2}} \leq C \|\rho^{\varepsilon,b}_0 - \rho^\varepsilon_0\|_{L^\frac{2n}{2n - \alpha}} \longrightarrow 0 \qquad \text{ as $b \to \infty$}.
    \end{split}
    \end{equation*}
    For $\alpha \in (-1,0]$ and $K \Subset \mathbb{R}^n$,
    \begin{equation*}
        \|\Phi_\alpha*\rho_0^{\v,b} - \Phi_\alpha * \rho^\v_0\|_{L^\infty(K)}+ \| \nabla (\Phi_\alpha*\rho^{\varepsilon,b}_0 - \Phi_\alpha * \rho^\v_0)\|_{L^\frac{2n}{2n - \alpha}(K)} \longrightarrow 0 \qquad \text{ as $b \to \infty$},
    \end{equation*}
    Moreover, there exists a constant $C > 0$ such that, for $\varepsilon > 0$ sufficiently small and $b$ sufficiently large,
    \begin{equation}\label{epinitial2}
        \int_{\mathbb{R}^n} \rho_0^{\varepsilon,b}(\boldsymbol{x}) \log(1 + |\boldsymbol{x}|^2) \dd \boldsymbol{x} \leq C.
    \end{equation}
\end{lemma}
\begin{proof}
    The first part is an application of the HLS inequality from Lemma \ref{simplehls}. The second part follows the same approach as for Lemma \ref{kernel1}.
\end{proof}

We now provide an approximation to the initial velocity. To avoid issues regarding the vacuum, we construct $u^\varepsilon_0$ like so
\begin{equation}\label{approx3}
    u^\varepsilon_0(\boldsymbol{x}) := \frac{1}{\sqrt{\rho_0^\varepsilon(\boldsymbol{x})}} \Big( \frac{m_0}{\sqrt{\rho_0}} \Big)(\boldsymbol{x}).
\end{equation}
As in \cite[Lemma A.8]{Chen2021}, we can prove the following:
\begin{lemma}
    As defined above, $u^\varepsilon_0$ satisfies
    \begin{align*}
  &\int_{\mathbb{R}^n} \rho^\v_0(\boldsymbol{x})|u_0^\varepsilon(\boldsymbol{x})|^2 \dd \boldsymbol{x} = \int_{\mathbb{R}^n} \frac{|m_0(\boldsymbol{x})|^2}{\rho_0(\boldsymbol{x})} \dd \boldsymbol{x}, \\[2mm]
  &\,\,\,\big\|\rho_0^\v u^\v_0 - m_0\big\|_{L^1} \longrightarrow 0 \qquad \text{ as $\varepsilon \to 0$}.
    \end{align*}
\end{lemma}
We can now define $\tilde{u}^{\v,b}_0$:
\begin{equation*}
    \tilde{u}^{\v,b}_0(\boldsymbol{x})
    = \frac{1}{\sqrt{\rho^{\v,b}_0(\boldsymbol{x})}} \int_{\mathbb{R}^n}
    \Big( \frac{m_0 \mathds{1}_{[4b^{-1},b - 2]}}{\sqrt{\rho_0}} \Big)(\boldsymbol{x} - \boldsymbol{y}) J_{b^{-1}}(\boldsymbol{y}) \dd \boldsymbol{y},
\end{equation*}
where we have used the indicator function $\mathds{1}_{[4b^{-1},\,b - 2]}$ so that $\supp (\tilde{u}^{\v,b}_0) \subset [2b^{-1},\,b-1]$,
which implies that the lower boundary condition is satisfied.
We can now add the stress-free boundary condition to the upper boundary.
In order to satisfy \eqref{boundary}, we construct an initial velocity that satisfies
\begin{equation*}
    \Big( p(\rho^{\v,b}_0) - \varepsilon \rho_0^{\v,b} (u^{\v,b}_{0,r} + \frac{n-1}{r}u^{\v,b}_0 ) \Big)(r) = 0
    \qquad \mbox{for $r \in (b - \frac{1}{4},b]$}.
\end{equation*}
Given a solution to the above, then
\begin{align*}
    (u^{\v,b}_0r^{n-1})_r = \frac{1}{\varepsilon} \frac{p(\rho^{\v,b}_0(r))}{\rho^{\v,b}_0(r)} r^{n-1}
    \qquad \mbox{for $r \in (b - \frac{1}{4},b)$.}
\end{align*}
Integrating over $(r,b)$ and choosing $u^{\v,b}_0(b) = 0$, we obtain
\begin{equation*}
    u^{\v,b}_0(r) = -\frac{1}{\varepsilon}\frac{1}{r^{n-1}}\int^b_r \frac{p(\rho^{\v,b}_0(z))}{\rho^{\v,b}_0(z)} z^{n-1} \dd z.
\end{equation*}
Hence, in order to add this boundary condition to $\tilde{u}^{\v,b}_0$, we use the construction from \cite{Chen2021}:
\begin{equation}\label{approx4}
    u^{\v,b}_0(\boldsymbol{x}) := \tilde{u}^{\v,b}_0(\boldsymbol{x}) - \frac{1}{\varepsilon} S(4(|\boldsymbol{x}| - (b - \frac{1}{2}))) 
    \frac{1}{|\boldsymbol{x}|^{n-1}}\int^b_{|\boldsymbol{x}|} \frac{p(\rho^{\v,b}_0(z))}{\rho^{\v,b}_0(z)} z^{n-1} \dd z.
\end{equation}
Then $u^{\v,b}_0$ satisfies \eqref{boundary}. As in \cite[Lemma A.9]{Chen2021}, we obtain the following lemma:
\begin{lemma}\label{appendix4}
    For given $\varepsilon > 0$,
    \begin{align*}
		&\lim_{b\rightarrow\infty}\int_{|\boldsymbol{x}|\leq b}
        \Big|\sqrt{\rho_{0}^{\v,b}(\boldsymbol{x})}u_{0}^{\v,b}(\boldsymbol{x})\Big|^2\dd \boldsymbol{x}
		=\int_{\mathbb{R}^n} \Big|\sqrt{\rho_0^\v(\boldsymbol{x})} u_0^\v(\boldsymbol{x})\Big|^2\dd\boldsymbol{x},\\[1mm]
		&\sqrt{\rho_{0}^{\v,b}(\boldsymbol{x})}u_{0}^{\v,b}(\boldsymbol{x})\longrightarrow \sqrt{\rho_{0}^{\v}(\boldsymbol{x})}u_{0}^{\v}(\boldsymbol{x})
		\qquad \mbox{in $L^2(\{|\boldsymbol{x}|\leq b\})\, $  as $b\rightarrow\infty$}.
    \end{align*}
\end{lemma}

We define the approximate initial energy as follows:
\begin{definition}\label{approxenergy}
Define the energy for the approximate initial data by
    \begin{align*}
        E_0^\v := \int_{\mathbb{R}^n} \bigg ( \frac{1}{2}\bigg | \frac{\M_0}{\sqrt{\rho_0^\v}} \bigg |^2 + \rho_0^\v e(\rho_0^\v) + \frac{\kappa}{2} \rho_0^\v (\Phi_\alpha * \rho_0^\v) \bigg )(\boldsymbol{x}) \dd \boldsymbol{x} < \infty,
    \end{align*}
    and
    \begin{align*}
        E_0^{\v,b} :=
            \int_{\mathbb{R}^n} \bigg ( \frac{1}{2}\bigg | \frac{\M_0}{\sqrt{\rho_0^{\v,b}}} \bigg |^2 + \rho_0^{\v,b} e(\rho_0^{\v,b}) + \frac{\kappa}{2} \rho_0^{\v,b} (\Phi_\alpha * \rho_0^{\v,b}) \bigg )(\boldsymbol{x}) \dd \boldsymbol{x} < \infty.
    \end{align*}
\end{definition}

Combining all together,  we obtain the following lemma:
\begin{lemma}\label{appendix5}
    Take $\rho^\v_0,\rho^{\v,b}_0,u^\v_0$, and $u^{\v,b}_0$ to be defined as in \eqref{approx1}, \eqref{approx2},
    \eqref{approx3}, and \eqref{approx4} respectively.
    Then $(\rho^{\v,b}_0,u^{\v,b}_0) \in C^\infty([b^{-1},b])$ and satisfies \eqref{boundary}.
    For $q \in \{1,\gamma\}$ when $\kappa = 1$ and $q \in \{1,\gamma,\frac{2n}{2n - \alpha}\}$ when $\kappa = -1$, the following statements hold{\rm :}
    \begin{itemize}
        \item[(a)] For all $\v\in(0,1]$,
		\begin{equation*}
			\int_0^\infty \rho_0^\v(r)\,\omega_n r^{n-1} \dd r=M,\qquad
			\v^2\int_{0}^\infty |\partial_r\sqrt{\rho^\v_{0}(r)}|^2\, r^{n-1}\dd r
			\leq C\v (M+1).
		\end{equation*}
		As $\v\rightarrow0$,
		\begin{align*}
                \begin{split}
			&E^\v_0\longrightarrow E_0, \\
			&(\rho_0^{\v}, m_0^{\v})(r)\longrightarrow (\rho_0, m_0)(r) \quad \mbox{in $L^q([0,\infty); r^{n-1}\dd r)\times L^1([0,\infty); r^{n-1}\dd r)$},\\
			&(\Phi_\alpha * \rho^\varepsilon_0)_r\to (\Phi_\alpha * \rho_0)_r \quad \mbox{in $L^\frac{2n}{\alpha + 2}([0,\infty); r^{n-1} \dd r)$}.
                \end{split}
            \end{align*}
            Then there exists $\varepsilon_0 > 0$ such that
            \begin{equation*}
			M<M_{\rm c}^{\v,\alpha}(\gamma) \qquad\mbox{for all $\v\in (0,\v_0]$ and $\gamma \in (\frac{2n}{2n - \alpha}, \frac{n + \alpha}{n}]$}.
		\end{equation*}
        \item[(b)] For all $\v\in(0,1]$,
        \begin{equation*}
			\int_{b^{-1}}^{b} \rho_{0}^{\v,b}(r)\,\omega_n r^{n-1} \dd r= M, \qquad
			\v^2\int_{0}^{b} \big|\partial_r\sqrt{\rho^{\v,b}_{0}(r)}\big|^2\, r^{n-1}\dd r
			\leq C\v (M+1).
		\end{equation*}
		Moreover, for any fixed $\v > 0$, as $b\rightarrow\infty$,
		\begin{align*}
			\begin{split}
				& E_0^{\v,b} \longrightarrow E_0^\v, \\
				&(\rho_0^{\v,b}, m_0^{\v,b})(r)\longrightarrow (\rho_0^\v, m^\v_0)(r) \qquad
				\mbox{\rm in $L^q([0, b]; r^{n-1}\dd r)\times L^1([0, b]; r^{n-1}\dd r)$},\\
				&(\Phi_\alpha * \rho^{\varepsilon,b}_0)_r\to (\Phi_\alpha * \rho^\varepsilon_0)_r \qquad \mbox{\rm in $L^\frac{2n}{\alpha + 2}([0,\infty); r^{n-1} \dd r)$}.
			\end{split}
		\end{align*}
            Then there exists $\mathfrak{B}(\v) > 0$ such that
            \begin{equation*}
			M<M_{\rm c}^{\v,b,\alpha}(\gamma) \qquad\mbox{for all $\v\in (0,\v_0]$, $b \geq \mathfrak{B}(\v)$, and $\gamma \in (\frac{2n}{2n - \alpha}, \frac{n + \alpha}{n}]$}.
		\end{equation*}
    \end{itemize}
\end{lemma}

\medskip
\noindent{\bf Acknowledgments.}
JAC, GQC, and DF acknowledge support by EPSRC grant EP/V051121/1. JAC was also partially supported by the Advanced Grant Nonlocal-CPD (Nonlocal PDEs for Complex Particle Dynamics: Phase Transitions, Patterns and Synchronization) of the European Research Council Executive Agency (ERC) under the European Union Horizon 2020 research and innovation programme (grant agreement No. 883363). GQC was also partially supported by EPSRC grants EP/L015811/1 and EP/V008854. DF was also partially supported by the Fundamental Research Funds for the Central Universities No. 2233100021 and No. 2233300008.
For the purpose of open access, the authors have applied a CC BY public copyright license to any Author Accepted Manuscript (AAM) version
arising from this submission.

\bigskip
\medskip
\noindent{\bf Conflict of Interest:} The authors declare that they have no conflict of interest.
The authors also declare that this manuscript has not been previously published,
and will not be submitted elsewhere before your decision.

\bigskip
\noindent{\bf Data availability:} Data sharing is not applicable to this article as no datasets were generated or analyzed during the current study.

\bigskip

\bigskip

\bibliography{bibliography}

\end{document}